\documentclass[11pt]{amsart}

\usepackage[margin=1in]{geometry}

\usepackage{amsmath, amscd, amsthm, amssymb,hyperref}

\def\C{{\mathbf C}}
\def\R{{\mathbf R}}
\def\Z{{\mathbf Z}}
\def\Q{{\mathbf Q}}
\def\A{{\mathbf A}}
\def\g{{\mathfrak g}}
\def\m{{\mathfrak m}}
\def\k{{\mathfrak k}}
\def\p{{\mathfrak p}}
\def\sl{{\mathfrak sl}}
\def\su{{\mathfrak su}}
\def\h{{\mathfrak h}}

\def\Vell{{\mathbb{V}_{\ell}}}

\newtheorem{theorem}{Theorem}[section]

\newtheorem{lemma}[theorem]{Lemma}

\newtheorem{proposition}[theorem]{Proposition}

\newtheorem{corollary}[theorem]{Corollary}

\newtheorem{claim}[theorem]{Claim}

\theoremstyle{definition}
\newtheorem{definition}[theorem]{Definition}

\theoremstyle{remark}
\newtheorem{remark}[theorem]{Remark}

\newcommand{\mm}[4]{\left(\begin{smallmatrix} #1 & #2\\ #3 & #4\end{smallmatrix}\right)}
\newcommand{\mb}[4]{\left(\begin{array}{cc} #1 & #2\\ #3 & #4\end{array}\right)}

\newcommand{\floor}[1]{\lfloor #1 \rfloor}

\DeclareMathOperator{\tr}{tr}

\DeclareMathOperator{\SO}{SO}
\DeclareMathOperator{\Spin}{Spin}
\DeclareMathOperator{\GSpin}{GSpin}
\DeclareMathOperator{\Sp}{Sp}

\DeclareMathOperator{\SU}{SU}

\DeclareMathOperator{\SL}{SL}
\DeclareMathOperator{\GL}{GL}

\DeclareMathOperator{\diag}{diag}

\begin{document}
	\title{Automatic convergence and arithmeticity of modular forms on exceptional groups}
	\author{Aaron Pollack}
	\address{Department of Mathematics\\ University of California San Diego\\ La Jolla, CA USA}
	\email{apollack@ucsd.edu}
	\thanks{Funding information: AP has been supported by the NSF via grant numbers 2101888 and 2144021.}
	
	\begin{abstract} We prove that the space of cuspidal quaternionic modular forms on the groups of type $F_4$ and $E_n$ have a purely algebraic characterization.  This characterization involves Fourier coefficients and Fourier-Jacobi expansions of the cuspidal modular forms.  The main component of the proof of the algebraic characterization is to show that certain infinite sums, which are potentially the Fourier expansion of a cuspidal modular form, converge absolutely.  As a consequence of the algebraic characterization, we deduce that the cuspidal quaternionic modular forms have a basis consisting of forms all of whose Fourier coefficients are algebraic numbers.
	\end{abstract}
	
	\maketitle
	
	\setcounter{tocdepth}{1}
	\tableofcontents
	\section{Introduction} 
	Holomorphic modular forms on groups $G$ with an associated Hermitian tube domain have a rich theory.  The group $G$ has an associated collection of Shimura varieties, and the holomorphic modular forms can be considered sections of coherent vector bundles on these varieties.  In line with the fact that the Shimura varieties can be defined over a number field, the holomorphic modular forms have an algebraic structure: There is a basis of the space of holomorphic modular forms on $G$, such that every classical Fourier coefficient of every element of this basis is an algebraic number.
	
	Going back to Gross-Wallach \cite{grossWallachI, grossWallachII}, Wallach \cite{wallach}, and Gan-Gross-Savin \cite{ganGrossSavin}, these authors have singled out for further study a collection of groups $G$, called the quaternionic groups, and certain automorphic forms on $G$, called the quaternionic modular forms.  The quaternionic groups, as we define them, consist of the collection $G_{2,2}, B_{3,3}, F_{4,4}, E_{6,4}, E_{7,4}, E_{8,4}$ and $D_{n,4}, B_{n,4}$ with $n \geq 4$.  Here by $X_{n,r}$ we mean a group of Dynkin type $X_n$ and real rank $r$.  Their symmetric spaces never have $G$-invariant complex structure, so the quaternionic groups $G$ do not have a theory of holomorphic modular forms.  Nevertheless, in prior work, the quaternionic modular forms have appeared to behave similarly to classical holomorphic modular forms.
	
	\subsection{Algebraicity of Fourier coefficients}
	In particular, extending and refining work of Wallach \cite{wallach} and Gan-Gross-Savin \cite{ganGrossSavin}, in \cite{pollackQDS} we gave a precise notion of Fourier expansion and Fourier coefficients of the quaternionic modular forms (QMFs) along a maximal Heisenberg parabolic subgroup $P$ of $G$.  Thus, associated to a QMF $\varphi$ on the group $G$, is a collection of complex numbers $a_{\varphi}(w)$, one for each $w$ in a certain rational vector space $W$, called the Fourier coefficients of $\varphi$.  Given the existence of this collection of complex numbers, it makes sense to ask if they have any arithmetic properties, as do the classical Fourier coefficients of holomorphic modular forms.
	
	In \cite{pollackETF}, we gave substantial evidence for this in the case $G = G_2$.  Quaternionic modular forms have a \emph{weight}, which is an integer $\ell$ at least $1$.  In \cite{pollackETF} we proved that the space of cuspidal modular forms on $G_2$ of even weight at least $6$ has a basis consisting of QMFs with all Fourier coefficients in the cyclotomic extension of $\Q$.  The proof in \cite{pollackETF} was constructive: We showed that every cusp form of even weight $\ell \geq 6$ on $G_2$ is an exceptional theta lift from an anisotropic group of type $F_{4,0}$, and we calculated the Fourier coefficients of these lifts.  This gives a somewhat algorithmic way of writing down the cuspidal QMFs on $G_2$, and we implemented these formulas in level one in \cite{pollackComputationG2}.
	
	One purpose of this paper is to resolve in the affirmative the question of whether the cuspidal quaternionic modular on the larger exceptional groups have an algebraic structure, defined in terms of Fourier coefficients.  For a quaternionic group $G$ and a subring $R$ of the complex numbers, let $S_\ell(G,R)$ denote the space of cuspidal quaternionic modular form on $G$ of weight $\ell$, all of whose Fourier coefficients are in $R$.
	
	\begin{theorem}[Algebraicity Theorem] \label{thm:algIntro} Suppose $G$ is a rational quaternionic group, of type $F_{4,4}$, $E_{6,4}$, $E_{7,4}$, or $E_{8,4}$.  Then $S_{\ell}(G,\C)$ has a basis consisting of modular forms all of whose Fourier coefficients lie in $\overline{\Q}$, the algebraic closure of $\Q$ in $\C$.  In other words, $S_{\ell}(G,\C) = S_{\ell}(G,\overline{\Q}) \otimes_{\overline{\Q}} \C$.
	\end{theorem}
	
	Our proof of algebraicity on $G_2$ does not generalize to the larger groups in Theorem \ref{thm:algIntro}, and our proof of algebraicity on these larger groups, as we will explain, does not specialize to $G_2$.
	
	More broadly, the primary purpose of this paper is to give a purely algebraic characterization of the cuspidal quaternionic modular forms on the groups $G$ in Theorem \ref{thm:algIntro}.  Specifically, the primary purpose of this paper is to prove Theorem \ref{thm:algebraizationIntro} below. To setup this theorem, we first delve into Fourier-Jacobi expansions and outline the proof of Theorem \ref{thm:algIntro}.
	
	\subsection{Fourier-Jacobi expansions}
	The first step in the proof of Theroem \ref{thm:algIntro} is to develop two notions of Fourier-Jacobi expansions for cuspidal quaternionic modular forms.  Besides a maximal Heisenberg parabolic subgroup $P$, the groups $G$ in Theorem \ref{thm:algIntro} also have two other maximal parabolic subgroups with which we work: A parabolic subgroup $Q$, whose Levi factor $M_Q$ has an $A_1$ quotient, and a parabolic subgroup $R$, whose Levi factor $M_R$ is of type $B$ or $D$.  
	
	The unipotent radical $N_Q$ of $Q$ is three-step.  Taking a non-degenerate character on the second step, we can define a Fourier-Jacobi coefficient of a cuspidal QMF $\varphi$ on $G$.  This Fourier-Jacobi coefficient is an automorphic form on $\widetilde{\SL_2}$, the double cover of $\SL_2$.  We prove that it corresponds to a holomorphic modular form, and relate its Fourier coefficients to the Heisenberg Fourier coefficients of $\varphi$.
	
	Likewise, the unipotent radical $N_R$ of $R$ is two-step.  Taking a non-degenerate character on the second step, we can again define a Fouier-Jacobi coeficient of a cuspidal QMF $\varphi$ on $G$.  This Fourier-Jacobi coefficient is now an automorphic form on a group of type $\SO(2, n)$.  We prove that it corresponds to a holomorphic modular form, and relate its Fourier coefficients to the Heisenberg Fourier coefficients of $\varphi$.  This expansion builds on and uses calculations from the paper \cite{apawMS}, which proved the existence of a Fourier-Jacobi coefficient in the case of the orthogonal quaternionic groups, and \cite{mcGladeFJ}, which handles a more general Fourier-Jacobi coefficient on the groups of type $B$ and $D$.
	
	\begin{theorem}[Fourier-Jacobi expansion] \label{thm:FJintro} Let $\varphi$ be a cuspidal quaternionic modular form of weight $\ell$ on the group $G$.
		\begin{enumerate}
			\item Let $G$ be any quaternionic group.  The Fourier-Jacobi coefficients of $\varphi$ along the parabolic $Q$ are holomorphic cuspidal modular forms on $\widetilde{\SL}_2$ of weight $\ell' = \ell+1 - \dim(J)/2$, where $G$ is associated to the cubic norm structure $J$.  
			
			\item Let $G$ be a quaternionic group of type $F_4$ or $E_n$, of rational (and real) rank four.  Assume that it is associated to $J = H_3(C)$ where $C$ is a rational compsotion algebra.  The Fourier-Jacobi coefficients of $\varphi$ along the parabolic $R$ are holomorphic cuspidal modular forms on (a group isogenous to) $\SO(2,\dim(C)+3)$ of weight $\ell_1 = \ell - \dim(C)$.  
		\end{enumerate}
	\end{theorem}

	\subsection{Converse theorem} The Heisenberg Fourier coefficients of a QMF $\varphi$ on $G$ are locally constant functions $a_{w}: G(\A_f) \rightarrow \C$, where $w \in W_J(\Q)$ runs over certain vectors in a rational vector space.   The existence of the Fourier-Jacobi modular forms from Theorem \ref{thm:FJintro} associated to $\varphi$ imposes many relations among the $a_w(g_f)$.  We loosely call these relations $P, Q,$ and $R$ symmetries, corresponding to the relations imposed by thinking about the Fourier expansion along the associated parabolic.  We detail these relations below in Definition \ref{defn:symmetries}.  Additionally, the fact that $\varphi$ is a QMF implies that the $a_w$ satisfy a moderate growth condition: the values $|a_w(g_f)|$ grow polynomially in the norm of $w$ for any fixed $g_f \in G(\A_f)$.
	
	The second step in the proof of Theorem \ref{thm:algIntro} is to prove that, conversely, if the functions $a_w(g_f)$ satisfy the $P$ and $R$ symmetries and grow polynomially, then they are the Fourier coefficients of a cuspidal QMF on $G$. (The $Q$ symmetries are not needed at this step.) Inuitively, the idea is that one writes down, using the $a_w(g_f)$, the putative QMF $\Psi$--a function on $G(\A)$--as an infinite sum.  The fact that the $a_w$ grow slowly means this sum converges absolutely.  Using that the $a_w$ satisfy the $P$ and $R$ symmetries, one can show that $\Psi$ is left invariant by $R(\Q)$ and another group $M_J^1(\Q) \not\subset R(\Q)$, and thus by $G(\Q)$. 
	
	\begin{theorem}[Converse Theorem] \label{thm:converseIntro} Suppose $G$ is a quaternionic exceptional group, of type $F$ or $E$.  Let $\{a_w\}_{w}: G(\A_f) \rightarrow \C$ be a set of functions, one for each $w \in W_J(\Q)$.  Assume that the $a_w$ are uniformly smooth, in that there is an open compact subgroup $U \subseteq G(\A_f)$ so that $a_w$ is right $U$-invariant for every $w \in W_J(\Q)$.  Assume moreover:
		\begin{enumerate}
			\item the $a_w$ satisfy the $P$ and $R$ symmetries;
			\item the numbers $|a_w(g_f)|$ grow polynomially with $w$ for each $g_f \in G(\A_f)$.  
		\end{enumerate}
		Then there is a cuspidal quaternionic modular form $\Psi$ on $G(\A)$, defined as an infinite sum, whose Fourier coefficients are the $a_w(g_f)$.
	\end{theorem}
	
	In this step, we use parabolic subgroup $R$, which does not have any analogue on $G_2$.  This is the reason why this argument does not specialize to $G_2$.  Moreover, this step does not apply to the groups of type $B_{n,4}$ and $D_{n,4}$, because $M_J^1(\Q) \subseteq R(\Q)$ for the groups of type $B$ and $D$.  We hope to prove the arithmeticity of the QMFs on groups of type $B$ and $D$ in the near future.
	
	Our proof of the Converse Theorem uses an ``Identity theorem" for quaternionic functions, which might be of independent interest.  The Identity theorem asserts that if a quaternionic function is $0$ on a sufficiently large subset of its domain, then it is identically $0$.
	
	\subsection{The automatic convergence theorem} The crucial final step in the proof of Theorem \ref{thm:algIntro} is what can be called an ``Automatic convergence theorem".  We prove that any collection of functions $a_w: G(\A_f) \rightarrow \C$ that satisfies the $P$, $Q$ and $R$ symmetries necessarily grows polynomially with $w$.  In other words, the sum defining $\Psi$ above from the $a_w(g_f)$ automatically converges absolutely, once one knows that the $a_w$ satisfy the requisite symmetries.  
	
	\begin{theorem}[Automatic Convergence] \label{thm:ACintro} Suppose $G$ is a quaternionic group of type $F_4$ or $E_n$ of rational (and real) rank four.  Let $\{a_w\}_{w}$ be a collection of functions that are uniformly smooth, and satsify the $P, Q$ and $R$ symmetries.  Then $|a_w(g_f)|$ grows polynomially with $w$ for every $g_f \in G(\A_f)$.
	\end{theorem}
	
	The automatic convergence theorem has antecedents in work of Ibukiyama-Poor-Yuen \cite{IPY}, Bruinier \cite{bruinierGenus2}, Raum \cite{raumGenus2}, Bruinier-Raum \cite{bruinerRaum2015,bruinierRaum2024}, and Xia \cite{xiaUnitary}, all of who proved similar results for holomorphic modular forms on symplectic groups or unitary groups, in various cases and in different degrees of generality. The techniques of these works are somewhat geometric, taking advantage of the ring structure on holomorphic modular forms and the existence of the Shimura variety.
	
	Our proof of the automatic convergence theorem for the quaternionic groups uses some of the ideas of \cite{bruinerRaum2015}, and some new ones: a lot of reduction theory, and a ``quantitative Sturm bound".  The quantitative Sturm bound says that if the ``first" several Fourier coefficients of a holomorphic modular form on a group of type $\SO(2,n)$ are small, then all the Fourier coefficients of this modular form are correspondingly small.  We intend to use these same ideas to give a new proof of an automatic convergence theorem for Siegel modular forms in a separate paper.
	
	\subsection{Algebraic characterization}
	Combining the results from Theorems \ref{thm:FJintro}, \ref{thm:converseIntro}, \ref{thm:ACintro}, and the fact that holomorphic modular forms can be characterized algebraically, one obtains a purely algebraic characterization of the elements of $S_{\ell}(G,\C)$: they can be identified with collection of functions $a_w: G(\A_f) \rightarrow \C$ that satisfy the $P$, $Q$, and $R$ symmetries.  \emph{No analytic or convergence criterion is needed.}  More precisely, for a subring $T$ of $\C$, let $S_{\ell}^{alg}(G,T)$ be the vector space of $T$-valued functions $a_w: G(\A_f) \rightarrow T$ that are jointly smooth, and satisfy the $P, Q$ and $R$ symmetries.
	
	\begin{theorem}[Algebraization of modular forms] \label{thm:algebraizationIntro} Suppose $G$ is a quaternionic exceptional group, of type $F$ or $E$.  The Fourier expansion map $S_{\ell}(G,\C) \rightarrow S_{\ell}^{alg}(G,\C)$ is a linear isomorphism.
	\end{theorem}
	
	Theorem \ref{thm:algebraizationIntro} is the main theorem of this paper.   As a consequence of it, we immediately deduce the algebraicity of the Fourier coefficients of cuspidal QMFs, i.e., Theorem \ref{thm:algIntro}.
	
	\begin{proof}[Proof of Theorem \ref{thm:algIntro}]  The $P,Q$ and $R$ symmetries can all be described in terms of linear relations among the $a_w(g_f)$ with algebraic coefficients.  By transcendental Galois theory, e.g., \cite[Theorem 9.29]{milneGalois}, one has $S_{\ell}^{alg}(G,\C) = S_{\ell}^{alg}(G,\overline{\Q}) \otimes_{\overline{\Q}} \C$.  Thus the result follows from Theorem \ref{thm:algebraizationIntro}.
	\end{proof}
	
	\section{Notation}
	In this section, we collect together much of the notation we use throughout the paper.
	
	\subsection{The quaternionic groups and subgroups}
	We use notation exactly as in \cite[sections 2,3,4]{pollackQDS}, unless stated otherwise.  Thus, $J$ denotes a cubic norm structure, and $J^\vee$ the dual structure.  We assume throughout that the trace pairing on $J \otimes \R$ is positive-definite; this is the assumption that leads to quaternionic groups, as opposed to other forms of the groups.  We will not comment again on this assumption.
	
	We let $M_J$ denote the identity component of the group of linear automorphism of $J$ that preserve the norm $N_J$ on $J$ up to scaling, and $M_J^1$ the subgroup that fixed the norm exactly.  Let $\m(J)$ denote the Lie algebra of $M_J$ and $\m(J)^0$ the Lie algebra of $M_J^1$.  For $B \in J$ with $N_J(B) \neq 0$, set $A_J^B$ the subgroup of $M_J$ that fixes $B$. 
	
	We let $W_J = \Q \oplus J \oplus J^\vee \oplus \Q$ denote the Freudenthal construction attached to $J$, and $H_J$ the identity component of the similitude group preserving (up to scaling) the natural symplectic $\langle \,, \rangle$ and quartic form $q_J$ on $W_J$.  We write $\nu: H_J \rightarrow \GL_1$ for the similtude, and set $H_J^1$ the kernel of $\nu$.  Set $\h(J)$ the Lie algebra of $H_J$ and $\h(J)^0$ the Lie algebra of $H_J^1$.  One has $\h(J)^0 = J \oplus \m(J) \oplus J^\vee$, and this is a $\Z$-grading, with $J$ in degree $1$, $\m(J)$ in degree $0$, and $J^\vee$ in degree $-1$.
	
	We write $\g_J$ for the Lie algebra associated to $J$ in section 4 of \cite{pollackQDS}.  Specifically, $\g_J = \sl_3 \oplus \m(J)^0 \oplus (V_3 \otimes J) \oplus (V_3 \otimes J)^\vee$.  This is a $\Z/3\Z$-grading.  One also has the $\Z/2\Z$-grading $\g(J) = \sl_2 \oplus \h(J)^0 \oplus V_2 \otimes W_J$.  Here $V_2$, respectively, $V_3$, is the standard representation of $\sl_2$, respectively, $\sl_3$.
	
	The group $G_J^{ad} = \mathrm{Aut}(\g(J))^0$ denotes the identity component of the automorphism group of the Lie algebra $\g(J)$.  We let $G_J \rightarrow G_J^{ad}$ denote a linear algebraic cover.  Sometimes, we drop $J$ from the notation, and write $G$ instead of $G_J$.  For the exceptional groups, $G_2, F_4$ and $E_8$ are both simply-connected and adjoint, so necessarily $G_J = G_J^{ad}$ in these cases.  For the case of $E_6, E_7$, explicit simply-connected covers are constructed in \cite[section 2.3 and 2.4]{pollackE8}.  For simplicity, we assume that $G_J(\R)$ is connected.  This is automatic if $G_J$ is of exceptional type.  
	
	\subsection{A restriction on the cubic norm structure}
	Our Fourier-Jacobi expansion for the parabolic subgroup $Q$ can be defined for arbitrary cubic norm structures.  The parabolic subgroup $R$--and thus the associated Fourier-Jacobi coefficients--can only be defined when the cubic norm structure $J$ satisfies a certain property.  Moreover, we can at this point only prove the automatic convergence theorem when $J$ satisfies a slightly more stringent condition.  In this subsection, we detail these assumptions on $J$.
	
	For the parabolic subgroup $R$ to exist, we assume that $J$ contains a rank one element.  See \cite[Definition 4.2.9 and Definition 4.3.2]{pollackLL} for the definition of rank of an element of $J$ and $W_J$.  More specifically, we will assume that $J$ is of the form $H_3(C)$ for a composition algebra $C$.  Thus, we are exluding the case of $J = \Q$, which corresponds to $G_J = G_2$, and the groups of type $B$ and $D$.  For $j \in \{1,2,3\}$, we let $e_{jj}$ be the element in $J$ with $1$ in the $(j,j)$ location and $0$'s elsewhere.

	\subsection{Parabolic subgroups of the quaternionic groups}\label{subsec:parabolic}
	In this subsection, we define the parabolic subgroups $P,Q$, and $R$.  We call $P$ the Heisenberg parabolic subgroup, $Q$ the $A_1$-parabolic subgroup, and $R$ the orthogonal parabolic subgroup.
	
	We first define the Heisenberg parabolic subgroup.  We have a five-step $\Z$-grading on $\g(J)$, see \cite[section 4.3]{pollackQDS}.  In the notation of this reference, 
	\[\g(J) = \Q E_{13} \oplus (e \otimes W_{J}) \oplus \h(J) \oplus (f \otimes W_J) \oplus \Q E_{31}\]
	where $e \otimes W_J = \Q E_{12} \oplus (v_1 \otimes J) \oplus (\delta_3 \otimes J^\vee) \oplus \Q E_{23}$ and $f \otimes W_J = \Q E_{32} \oplus (v_3 \otimes J) \oplus (\delta_1 \otimes J^\vee) \oplus \Q E_{21}$.  The Heisenberg parabolic subgroup $P$ is the one whose Lie algebra is $ \Q E_{13} \oplus (e \otimes W_{J}) \oplus \h(J)$.  The Levi subgroup $M_P$ has Lie algebra $\h(J)$, and the unipotent radical $N_P$ has Lie algebra $ \Q E_{13} \oplus (e \otimes W_{J})$.  Put differently, the Hiesenberg parabolic subgroup $P = M_P N_P$ is the one associated to grading on $\g(J)$ defined by the element $h_P :=E_{11} - E_{33} \in \g(J)$.  That is, $ad(E_{11}-E_{33})$ has eigenvalues $2,1,0,-1,-2$ on $g(J)$, and $Lie(M_P)$ is the $0$ eigenspace, whereas $Lie(N_P)$ is the direct sum of the $1$ and $2$ eigenspaces.  We let $Z$ denote the subgroup of $N_J$ whose Lie algebra is the $2$-eigenspace of $ad(h_P)$.  Thus $Lie(Z)$ is spanned by $E_{13}$.
	
	We next define the $A_1$-parabolic subgroup $Q$.  Set $h_Q = E_{11}+E_{22}-2E_{33}$.  Then $ad(h_Q)$ has eigenvalues $3,2,1,0,-1,-2,-3$ on $\g(J)$.  We let $Q = M_Q N_Q$ be the corresponding parabolic subgroup, so that $Lie(M_Q)$ is the $0$ eigenspace of $ad(h_Q)$, whereas $Lie(N_Q)$ is the direct sum of the positive eigenspaces.  The subgroup $Q$ preserves the two-dimensional subspace $\g(J)^{ad(h_Q)=3} = \Q E_{13} \oplus \Q E_{23}$.  This defines a group homomorphism $Q \rightarrow \GL_2$.
	
	We now define the orthogonal parabolic subgroup $R$, assuming that $J$ satisfies assumption $R$. For $\gamma \in J^\vee$ and $x \in J$, let $\Phi'_{\gamma,x} \in \m(J)^0$ be the associated Lie algebra element; see \cite[section 3.3]{pollackQDS}. Set $h_R = \frac{2}{3}(E_{11}+E_{22}-2 E_{33}) + \Phi'_{e_{11},e_{11}}$.  Here $e_{11} \in J, J^\vee$ precisely because $J$ satisfies assumption $R$.  We let $R = M_R N_R$ be the parabolic subgroup associated to $h_R$.  Thus $Lie(M_R)$ is the $0$-eigenspace of $ad(h_R)$ on $\g(J)$, whereas $Lie(N_R)$ is the direct sum of the positive eigenspaces for $ad(h_R)$ on $\g(J)$.  If $\g(J)$ is exceptional, then $ad(h_R)$ takes the eigenvalues $2,1,0,-1,-2$ on $\g(J)$, while if $\g(J)$ is of type $B$ or $D$, then $ad(h_R)$ takes the eigenvalues $2,0,-2$ on $\g(J)$.

	\subsection{Actions} If $U$ is a vector space with a left action of a group $T$, and $\langle \,,\,\rangle$ is a non-degenerate bilinear form on $U$ that is invariant up to scaling for the action of $T$, we define the right action of $U$ so that $\langle u_1 \cdot t, u_2 \rangle = \langle u_1, t \cdot u_2\rangle$ for all $u_1, u_2 \in U$.
	
	\subsection{Lattices} Let $\mathcal{O}_C \subseteq C$ be a maximal order in the composition algebra $C$.  Set $J_0 \subseteq J = H_3(C)$ to be the lattice of elements with diagonal entries in $\Z$ and off-diagonal entries in $\mathcal{O}_C$.  Set $\Lambda_0 = \Z \oplus J_0 \oplus J_0 \oplus \Z \subseteq W_J(\Q)$.  
	
	Let $V_3(\Z) = \Z v_1 \oplus \Z v_2 \oplus \Z v_3$ be the standard lattice in the defining representation of $\SL_3$.  We fix a lattice $\Lambda_\g \subseteq \g(J)$ so that $\Lambda_\g$ contains $V_3(\Z) \otimes J_0$,  $V_3(\Z)^\vee \otimes J_0$, $E_{ij}$ for $i \neq j$, and $\Lambda_\g$ is closed under the Lie bracket.  This can be done.  We can (and do) also assume that $\Lambda_\g$ preserves a lattice in a faithful representation of the group $G$.  Consequently, if $p$ is prime number, the exponential $\exp(p v) \in G(\Q_p)$ is defined, if $v \in \Lambda_\g \otimes \Z_p$.

	If $\Lambda$ is a lattice and $\lambda \in \Lambda$ is nonzero, we write $\mathrm{cont}(\lambda;\Lambda)$ for the content of $\lambda$ with respect to $\Lambda$.  Thus, $\mathrm{cont}(\lambda;\Lambda)$ is the largest positive integer $n$ so that $n^{-1} \lambda \in \Lambda$.  If $\lambda \in \Lambda \otimes \Q$, and $m$ is a positive integer so that $m \lambda \in \Lambda$, we define $\mathrm{cont}(\lambda;\Lambda) = m^{-1} \mathrm{cont}(m\lambda;\Lambda)$.  This is well-defined.
	
	\section{Review of quaternionic modular forms}
	In this section, we briefly review quaternionic modular forms.  
	
	\subsection{Generalities} Let $J$ be a cubic norm structure and $G_J$ a rational quaternionic group.  Let $\ell \geq 1$ be an integer.  Let $K_J \subseteq G_J(\R)$ be the maximal compact subgroup as defined by the Cartan involution in \cite[section 4]{pollackQDS}.   Recall that we assume $G_J(\R)$, and thus $K_J$, is connected.  The Lie algebra of $K_J$ has a distinguished $\su_2$ as a direct factor; see \cite[section 6]{pollackQDS}. The conjugation action defines a surjection $K_J \rightarrow \SU_2/\mu_2 = \mathrm{Aut}(\su_2)$.  Let $\Vell = \mathrm{Sym}^{2\ell}(\C^2)$ denote the $(2\ell+1)$-dimensional irreducible representation of $\SU_2/\mu_2$, pulled back to $K_J$.  In \cite[section 6]{pollackQDS}, we endow the $\su_2 \otimes \C \subseteq \g(J) \otimes \C$ with an $\sl_2$-triple, which gives rise to a basis $\{x^{2\ell}, x^{2\ell-1}y, \ldots, x y^{2\ell-1},y^{2\ell}\}$ of $\Vell$.  
	
	Quaternionic modular forms are defined to be automorphic forms on $G_J$ that are annihilated by a certain Schmid differential operator $D_{\ell}$.  We review this now.  Suppose then that $F: G_J(\R) \rightarrow \Vell$ is a smooth function, satisfying $F(gk) = k^{-1} F(g)$ for all $g \in G_J(\R)$ and $k \in K_J$.  Let $\g(J) \otimes \C = \k \oplus \p$ be the Cartan decomposition of $\g(J)$, let $\{X_{\alpha}\}$ be a basis of $\p$ and $\{X_\alpha^\vee\}$ the dual basis of $\p^\vee$.  Define $\widetilde{D} F = \sum_{\alpha}{ X_\alpha F \otimes X_\alpha^\vee}$, so that $\widetilde{D} F$ takes values in $\Vell \otimes \p^\vee$.  Now, there is an identification $\p \approx V_2(\C) \otimes W_J$, where recall $V_2$ denotes the two-dimensional representation of $\SL_2$.  Consequently, there is a $K_J$-equivariant projection $pr_{D}:\Vell \otimes \p^\vee \rightarrow \mathrm{Sym}^{2\ell-1}(\C^2) \otimes W_J$.  Define $D_{\ell} = pr_D \circ \widetilde{D}$.
	
	\begin{definition} Suppose $\ell \geq 1$ is an integer.  A quaterionic modular form on $G_J$ of weight $\ell$ is a smooth function $\varphi: G_J(\Q)\backslash G_J(\A) \rightarrow \Vell$ satisfying
		\begin{enumerate}
			\item $\varphi$ is of moderate growth and $\mathcal{Z}(\g(J))$-finite;
			\item $\varphi(gk) = k^{-1} \varphi(g)$ for all $k \in K_J$ and $g \in G_J(\A)$;
			\item $D_\ell \varphi \equiv 0$.
		\end{enumerate}
		Here that $\varphi$ is smooth means that there is an open compact subgroup $U \subseteq G_J(\A_f)$ so that $\varphi$ is right-invariant by $U$, and, for each $g_f \in G_J(\A_f)$, the function $\varphi(g_f g_\infty): G_J(\R) \rightarrow \Vell$ is smooth in the usual sense.  One says that $\varphi$ is cuspidal if, as usual, the constant term of $\varphi$ along the unipotent radical of every proper rational parabolic subgroup is identically $0$.  It follows from the main result of \cite{pollackQDS} that $\varphi$ is cuspidal if and only if $\varphi$ is bounded.
	\end{definition}
	
	\subsection{The Fourier expansion}
	Quaternionic modular forms have a semi-classical Fourier expansion.  The exact shape of this expansion is the main result of \cite{pollackQDS}.  For $w \in W_J(\R)$, let $\chi_w: N_J(\R) \rightarrow \C^\times$ be the unitary character given by $\chi_w(n) = e^{i \langle w, \overline{n} \rangle}$, where $\overline{n}$ is the image of $n$ in $W_J(\R)$, via the map $N_J(\R)/Z(\R) \stackrel{\log}{\rightarrow} W_J(\R)$.  If $w \neq 0$, there is a defined in \cite{pollackQDS} an explicit, smooth, moderate growth function $W_w: G_J(\R) \rightarrow \Vell$ satisfying 
	\begin{enumerate}
		\item $W_w(n g) = \chi_w(n) W_w(g)$ for all $w \in N_J(\R)$ and $g \in G_J(\R)$;
		\item $W_w(gk) = k^{-1} W_w(g)$ for all $g \in G_J(\R)$ and $k \in K_J$;
		\item $D_{\ell} W_w \equiv 0$
	\end{enumerate}
	In fact, it is proved in \cite{pollackQDS}, extending a result of Wallach from \cite{wallach}, that the space of such functions is at most one-dimensional.  A specific element of this space is singled out.
	
	In order for the space of such generalized Whittaker functions to be nonzero, the element $w$ must satisfy a condition called \emph{positive semi-definiteness}.  Let $r_0(i) = (1, -i 1_J, -1_J, i ) \in W_J(\C)$.  The element $w$ is said to be positive semi-definite if $\langle w , g r_0(i) \rangle \neq 0$ for all $g \in M_P(\R)$.  The element $w$ is said to be positive definite if in addition $q_J(w) \neq 0$, in which case $q_J(w) < 0$ (in our normalization of $q_J$).  We write $w > 0$ if $w$ is positive-definite.  If $w$ is positive semi-definite, then for $g \in M_P(\R)$ and $\alpha_w(g) = \langle w, g r_0(i) \rangle$,
	
	\[W_w(g) = \nu(g)^{\ell} |\nu(g)| \sum_{-\ell \leq v \leq \ell} \left(\frac{|\alpha_w(g)|}{\alpha_w(g)}\right)^v K_v(|\alpha_w(g)|) \frac{x^{\ell+v} y^{\ell-v}}{(\ell+v)!(\ell-v)!}.\]
	This formula, together with the $N_J(\R)$ and $K_J$-equivariance conditions, uniquely determines $W_w(g)$.
	
	Let $\psi: \Q\backslash \A \rightarrow \C^\times$ be the standard additive character.   For each $w \in W_J(\Q)$, define a character $\xi_w: N_J(\Q)\backslash N_J(\A) \rightarrow \C^\times$ as $\xi_w(n) = \psi(\langle w, \overline{n} \rangle)$, where again $\overline{n}$ is the image of $n$ in $W_J(\A)$ via the $\log$ map.  We have $\xi_{w}|_{N_J(\R)} = \chi_{2\pi w}$.  
	
	Suppose now that $\varphi$ is a weight $\ell$ QMF on $G_J(\A)$. We can take the constant term of $\varphi$ along $Z$, and Fourier expand along $Z(\A) N_J(\Q) \backslash N_J(\A)$ to obtain $\varphi_Z(g) = \varphi_{N_J}(g) + \sum_{ w \neq 0}{\varphi_w(g)}$, where
	\[\varphi_w(g) = \int_{[N_J]}{\xi_w^{-1}(n) \varphi(ng)\,dn}.\]
	By the main theorem of \cite{pollackQDS}, we have $\varphi_w(g_f g_\infty) = a_w(g_f) W_{2\pi w}(g_\infty)$ for some locally constant function $a_w: G_J(\A_f) \rightarrow \C$.  The function $a_w$ is called the $w$ Fourier coefficient of $\varphi$.  If $\varphi$ is cuspidal, then 
	\[\varphi_Z(g_f g_\infty) = \sum_{w > 0}{a_w(g_f) W_{2\pi w}(g_\infty)},\]
	the sum being over positive-definite $w$.
	
	In fact, one can recover the entire function $\varphi$ from the $a_w$'s.  If $w = (a,b,c,d) \in W_J$, let $a(w) =a$, $b(w) = b$, etc.  Suppose $\varphi$ is cuspidal.  Then
	\begin{align*}
		\varphi(g) &= \sum_{w \in W_J(\Q), a(w) = 0}{\varphi_w(g)} + \sum_{\gamma \in B_2(\Q)\backslash \SL_2(\Q)}\sum_{w \in W_J(\Q), a(w) \neq 0}{\varphi_w(\gamma g)} \\
		&= \sum_{w \in W_J(\Q), a(w) = 0}{a_w(g_f)W_{2\pi w}(g_\infty)} + \sum_{\gamma \in B_2(\Q)\backslash \SL_2(\Q)}\sum_{w \in W_J(\Q), a(w) \neq 0}{a_w(\gamma_f g_f) W_{2\pi w}(\gamma_\infty g_\infty)} \
	\end{align*}
	Here the $\SL_2$ is embedded in $M_Q$, and $B_2 \subseteq \SL_2$ is the upper-triangular Borel subgroup.
	
	\section{The Weil representation}\label{sec:Weil}
	In this section, we collect together results we will need about the Weil representation and theta functions.
	
	\subsection{Heisenberg groups}
	For us, a Heisenberg group $H$ is an extension $1 \rightarrow Z \rightarrow H \rightarrow W \rightarrow 1$ with the following properties:
	\begin{enumerate}
		\item the subgroup $Z$ and the quotient $W$ are vector groups, i.e., isomorphic to a finite sum of copies of the additive group $\mathbf{G}_a$;
		\item the subgroup $Z$ is the center of $H$;
		\item for a Zariski-open set of linear maps $\ell: Z \rightarrow \mathbf{G}_a$, the alternating pairing $\langle\,,\, \rangle_{H,\ell}: W \times W \rightarrow \mathbf{G}_a$ given by $\langle w_1, w_2 \rangle_{H,\ell} = \ell([w_1, w_2])$ (well-defined because $Z$ is central) is non-degenerate.  In particular, $W$ is even-dimensional.
	\end{enumerate}
	It would be more conventional to rephrase the above as follows: Let $H_{\ell}$ be the extension $1 \rightarrow \mathbf{G}_a \rightarrow H_{\ell} \rightarrow W \rightarrow 1$ obtained from $H$ by pushout along $\ell: Z \rightarrow \mathbf{G}_a$.  Then $H_{\ell}$ is a Heisenberg group in the usual sense of the word.
	
	Heisenberg groups have Weil representations and theta functions.   Suppose first $k$ is a local field.  Fix $\ell \in Z^\vee$ so that the corresponding alternating pairing is non-degenerate.  Let $\psi: k \rightarrow \C^\times$ be an additive character.  Let $W  = X \oplus Y$ be a Lagrangian decomposition.  The character $\psi_{\ell}=\psi \circ \ell: Z(k) \rightarrow \C^\times$ extends trivially to a character of the abelian subgroup $YZ(k)$ of $H(k)$, which we denote by $\psi_{Y,\ell}$.  The Weil representation of $H(k)$ is the smooth induced representation $Ind_{YZ(k)}^{H(k)}(\psi_{Y,\ell})$.  It is denoted by $\omega_{\psi_{Y,\ell}}$, although we will sometimes drop the subcripts.  We identify the space of this representation with $S(X(k))$, the Schwartz-(Bruhat) space on $X(k)$.
	
	If $\ell, X,Y$ are defined over our ground field $\Q$, then the representations just produced tensor together to give a representation of $H(\A)$ on $S(X(\A))$.  From this global representation, we can define $\theta$-functions, as follows.  Suppose $\phi \in S(X(\A))$.  One defines
	\[\theta_\phi(h) = \sum_{\xi \in X(\Q)}{(\omega_{\psi_{Y,\ell}}(h)\phi)(\xi)} = \sum_{\xi \in X(\Q)}{(\omega_{\psi_{Y,\ell}}(\xi h)\phi)(0)}\]
	The function $\theta_\phi$, defined on $H(\A)$, is in fact left-invariant by $H(\Q)$.  It is also smooth and of moderate growth on $H(\A)$.
	
	\subsection{Symplectic groups}
	Let $\widetilde{\Sp(W)} \rightarrow \Sp(W)$ denote the metaplectic two-fold cover.  The group $\Sp(W)$ acts on the Heisenberg group $H_{\ell}$.  In many references, a right action of $\Sp(W)$ on $W$ and $H_{\ell}$ is assumed.  We will relate these particular right and left actions via $w \cdot g =  g^{-1} \cdot w$. 
	
	Suppose again $k$ is a local field.  From the action of $\Sp(W)$ on $H$, we have a semi-direct product $J_{\ell}:=H_{\ell} \rtimes \widetilde{\Sp(W)}$, called the Jacobi group.  The representation of $H(k)$ on $S(X(k))$ extends to a representation of $\widetilde{J}(k):=H_{\ell}(k) \rtimes \widetilde{\Sp(W)}(k)$.  This is again called the Weil representation, and we denote it again by $\psi_{Y,\ell}$.  As $k$ varies over the completions of $\Q$, the representations for the various $k$ piece together to give a representation of $\widetilde{J}_{\ell}(\A)=H_{\ell}(\A) \rtimes \widetilde{\Sp(W)}(\A)$.
	
	If $g  = h r \in H_{\ell}(\A) \rtimes \widetilde{\Sp(W)}(\A)$, we can define $\theta_\phi(g) = \sum_{\xi \in X(\Q)}{\omega_{\psi_{Y,\ell}}(g)\phi(\xi)}$.  This function is automorphic on the Jacobi group $\widetilde{J}_{\ell}(\A)$.
	
	We will need a couple of formulas for this Weil representation.  Let $N_{S,Y} \subseteq \Sp(W)$ be the unipotent radical of the Siegel parabolic subgroup $P_{S,Y}$ stabilizing $Y \subseteq W$ for the right action.  Then, there is a unique splitting of $N_{S,Y}(k)$ into $\widetilde{\Sp(W)(k)}.$  Suppose $n \in N_{S,Y}$ has matrix form $\mm{1}{\beta}{}{1}$, so that $\beta \in Hom(X,Y)$.  Then, using the splitting,
	\[\omega_{\psi_{Y,\ell}}(n)\phi(x) = \psi_{\ell}( \langle x, x \cdot \beta \rangle/2) \phi(x) = \psi(\ell([x,x\cdot \beta])/2) \phi(x)= \psi(\ell([x,x\cdot n])/2) \phi(x).\]
	
	Let $\det_{Y}: P_{S,Y} \rightarrow \mathbf{G}_m$ denote the determinant for the action of $p \in P_{S,Y}$ on $Y$.  Let $P_{S,Y}^1$ denote the subgroup with $\det_Y$ equal to $1$.  Then $P_{S,Y}^1(k)$ splits uniquely into $\widetilde{\Sp(W)}(k)$, and for $p \in P_{S,Y}^1(k)$ one has
	\[\omega_{\psi_{Y,\ell}}(p)\phi(x) = |\det_Y(p)|^{-1/2} \psi(\ell([pr_{X}(xp), xp])/2) \phi(pr_{X}(xp))\]
	where $pr_{X}: W \rightarrow W/Y \simeq X$ is the projection.  Of course, the first term is $1$ for $p \in P_{S,Y}^1$.  However, the formula remains accurate for $p$ in the identity component of $P_{S,Y}(\R)$, which is why we have written it this way.
	
	\section{The Fourier-Jacobi expansion for $Q$}
	In this section, we derive the Fourier-Jacobi coefficients of a quaternionic modular form associated to the parabolic subgroup $Q$.
	
	\subsection{The general Fourier-Jacobi coefficient}
	Recall the parabolic subgroup $Q$, together with its Levi decomposition $Q = M_Q N_Q$.  Let $N_Q^3$ denote the subgroup of $N_Q$ with Lie algebra $Lie(N_Q^3) = \Q E_{13} + \Q E_{23}$.  Then $N_Q/N_Q^3$ is a Heisenberg group in the sense of section \ref{sec:Weil}. Its center $Z(N_Q/N_Q^3)$ has Lie algebra identified with $\delta_3 \otimes J^\vee \subseteq \g(J)$.   Fix $B \in J(\Q)$ with nonzero norm.   Then $B$ gives linear map on the center of $N_Q/N_Q^3$.  Let $H_B$ be the pushout of $N_Q/N_Q^3$ along the map $Z(N_Q/N_Q^3) \rightarrow \mathbf{G}_a$ given by $B$.  The group $H_B$ is a Heisenberg group, with center $Z_B:=\mathbf{G}_a$ and abelianization the vector group $W_J^B = \mathrm{Span}(v_1,v_2) \otimes J$.   We set $X = v_2 \otimes J$ and $Y = v_1 \otimes J$.  This gives a Lagrangian decomposition of $W_J^B$.  Concretely, the symplectic form on $W_J^B$ is determined by $\langle v_1 \otimes y, v_2 \otimes x \rangle = (B,x,y)_J$.  Here $(\,,\,,\,)_J$ is the unique symmetric trilinear form satisfying $(z,z,z)_J = 6 N_J(z)$.
	
	The group $M_Q$ acts on the space $v_3 \otimes J$, via the adjoint action on $\g(J)$.  Let $M_Q^B$ denote the subgroup of $M_Q$ that fixes the element $v_3 \otimes B \in v_3 \otimes J$.  The adjoint action of $M_Q$ on $N_Q/N_Q^3$ gives a homomorphism $M_Q^B \rightarrow \Sp(W_J^B)$.  For a local field $k$, let $\widetilde{M_Q^B}(k) \rightarrow M_Q^B(k)$ denote the pullback of $\widetilde{\Sp(W_J^B)} \rightarrow \Sp(W_J^B)$.  
	
	Suppose $\varphi$ is a cuspidal automorphic form on $G_J$.  Let $\phi \in S(X(\A)) = S(J(\A))$.  We can now define the Fourier-Jacobi coefficient of $\varphi$ associated to the pair $(B, \phi)$, which is an automorphic function on $\widetilde{M_Q^B}(\A)$.
	
	\begin{definition} Let the notation be as above.  For $g = hr \in H_B(\A) \rtimes \widetilde{M_Q^B}(\A)$, let
		\[\theta_\phi(g) = \sum_{\xi \in X(\Q)}{ \omega(g)\phi(\xi)},\]
		be the theta function. One can inflate $\theta_\phi$ to a function on $N_Q/N_Q^3(\A) \rtimes \widetilde{M_Q^B}(\A)$. For $r \in \widetilde{M_Q^B}(\A)$, let $\overline{r}$ denote its image in $M_B^Q(\A)$.  Let $\varphi_{NQ^3}$ denote the constant term of $\varphi$ along $N_Q^3$.  The Fourier-Jacobi coefficient of $\varphi$ associated to $(B,\phi)$ is defined as
		\[\mathrm{FJ}_{B,\phi}(\varphi)(r) = \int_{H_B(\Q)\backslash H_B(\A)}{\varphi(h\overline{r}) \theta_\phi(hr)\,dh}.\]
		It is an automorphic form on $\widetilde{M_Q^B}(\A)$.
	\end{definition}
	
	Regarding this Fourier-Jacobi coefficient, we prove the following proposition.  For $w \in W_J(\Q)$, recall the Fourier coefficient $\varphi_w$ of $\varphi$ along the Heisenberg unipotent subgroup.  To set up the proposition, for $d \in \Q$ and $g \in \widetilde{M_Q^B}(\A)$ define 
	
	\[\mathcal{F}_{B,d,\phi}(\varphi)(g) = \int_{J(\A)}{ \varphi_{(0,B,0,d)}(\exp( v_2 \otimes x) \overline{g}) (\omega_{\psi_B}(g)\phi)(x)\,dx}.\]
	
	\begin{proposition}\label{prop:FJQ1}One has
		\[ \mathrm{FJ}_{B,\phi}(\varphi)(g) = \sum_{d \in \Q}{\mathcal{F}_{B,d,\phi}(\varphi)(g)},\]
		and this is its Fourier expansion along the unipotent group $\exp(\A E_{12}) \subseteq \widetilde{M_Q^B}(\A)$. In particular, $\mathcal{F}_{B,d,\phi}$ is the $(-d)^{th}$ Fourier coefficient of $\mathrm{FJ}_{B,\phi}(\varphi)$ along the unipotent group $\exp(\A E_{12})$.
	\end{proposition}
	\begin{proof} Let $\varphi_{(0,B,*,*)} = \sum_{\gamma \in J^\vee, d \in \Q}{\varphi_{(0,B,\gamma,d)}}$ and let $\varphi_{(0,B,0,*)} = \sum_{d \in \Q}{\varphi_{(0,B,0,d)}}$.  With this notation, we have
		
		\begin{align*}\mathrm{FJ}_{B,\phi}(\varphi)(g) &= \int_{H_B(\Q) Z_B(\A)  \backslash H_B(\A)}{\varphi_{(0,B,*,*)}(hg) \Theta_\phi(hg)\,dh} \\ 
			&= \int_{Y(\Q) Z_B(\A)\backslash H_B(\A)}{\varphi_{(0,B,*,*)}(yxg) \omega_{\psi_B}(yxg)\phi(0)\,dh} \\
			&= \int_{Y(\A)Z_B(\A)\backslash H_B(\A)}{\varphi_{(0,B,0,*)}(xg) \omega_{\psi_B}(xg)\phi(0)\,dh} \\
			&= \sum_{d \in \Q}{\mathcal{F}_{B,d,\phi}(\varphi)(g)}.\end{align*}
		
		The proof is completed by verifying that $\mathcal{F}_{B,d,\phi}(\varphi)(g)$ has the correct equivariance property with respect to $\exp( u E_{12})$.  To verify this, we need to compute $\exp(-u E_{12}) \exp( v_2 \otimes x) \exp( u E_{12})$.  One has
		
		\begin{align*}\exp(-uE_{12}) \exp( v_2 \otimes x) \exp( u E_{12}) &= \exp( Ad(\exp(-u E_{12})) v_2 \otimes x) \\ 
			&= \exp(v_2 \otimes x- u v_1 \otimes x)\\
			&=  \exp(v_2 \otimes x- u v_1 \otimes x) \exp(-v_2 \otimes x) \exp(v_2 \otimes x) \\
			&= \exp(-u v_1 \otimes x+ u \delta_3 \otimes x^\# + A) \exp(v_2 \otimes x)
		\end{align*}
		where the last line uses Baker-Campbell-Hausdorff and $A \in N_Q^3$.  Thus 
		\[\varphi_{(0,B,0,d)}(\exp(v_2 \otimes x) \exp(u E_{12})g) = \psi(-du) \psi(-u(B,x^\#)) \varphi_{(0,B,0,d)}(\exp(v_2 \otimes x) g).\]
		As $\omega_{\psi_B}(\exp(u E_{12}))\phi(x) = \psi(u (B,x^\#))\phi(x)$, this proves the proposition.
	\end{proof}
	
	\subsection{Holomorphic modular forms}\label{subsec:holMFQ}
	In this subsection, we use the Fourier-Jacobi coefficient studied in Proposition \ref{prop:FJQ1} to show that certain linear combinations of Fourier coefficients of a quaternionic modular form are the Fourier coefficients of a holomorphic modular form on $\widetilde{\SL_2}$. 
	
	Suppose $\varphi$ is a cuspidal quaternionic modular form on $G_J$ of weight $\ell$, with Fourier expansion
	\[ \varphi_Z(g) = \sum_{w \in W_J(\Q), w > 0}{a_w(g_f) W_{2\pi w,\ell}(g_\infty)}.\]
	If $\phi \in S(J(\A_f))$ is a Schwartz-Bruhat function at the finite places, $d \in \Q^\times$, and $r_f \in \widetilde{M_Q^B}(\A_f)$ and $g_f \in G(\A_f)$, set
	\[A^Q_{\varphi,B,d}(r_f,g_f;\phi) = \int_{J(\A_f)}{a_{(0,B,0,d)}(\exp(v_2 \otimes x) \overline{r_f}g_f) (\omega_{\psi_B}(r_f)\phi)(x)\,dx}.\]
	
	We have $\SL_3 \rightarrow G_J^{ad}$ via the $\Z/3\Z$-model of the Lie algebra $\g(J)$.  Embed $\SL_2$ into $\SL_3$ via the upper left $2 \times 2$ block.  Let $j_{E_{12}}: \SL_2 \rightarrow M_Q^B \subseteq G_J$ be the associated map into $G_J$, and $j_{E_{12}}': \widetilde{\SL_2} \rightarrow \widetilde{M_Q^B}$ the lift to the double cover.  
	
	On $\widetilde{\SL_2}(\R)$ times the upper half-plane $\mathcal{H} = \{z \in \C: Im(z) > 0\}$, let $j_{1/2}: \widetilde{\SL_2}(\R) \times \mathcal{H} \rightarrow \C^\times$ be the canonical squareroot of $cz+d$, $\overline{g_\infty} = \mm{a}{b}{c}{d} \in \SL_2(\R)$ and $z \in \mathcal{H}$.  For an a half-integer $r$, $n \in \Q$ and $g_\infty \in \widetilde{\SL_2}(\R)$, let 
	\[\mathcal{W}_{\SL_2,r,n}(g_\infty) = j_{1/2}(g_\infty,i)^{-2r} e^{2\pi i n (\overline{g_\infty} \cdot i)}.\]
	We have the following theorem.  
	\begin{theorem} Let the notation be as above, so that $\phi \in S(J(\A_f))$ and $\varphi$ is a quaternionic modular form of weight $\ell$.  Let $B \in J$ be positive-definite and set $\ell' = \ell+1 - \dim(J)/2$.  Suppose $g_f \in \widetilde{\SL_2}(\A_f)$ and $g_\infty \in \widetilde{\SL_2}(\R)$.  Then
		\[ \sum_{n \in \Q_{> 0}}{A_{\varphi,B,-n}^Q(j_{E_{12}}'(g_f);\phi) \mathcal{W}_{\SL_2,\ell',n}(g_\infty)}\]
		is the Fourier expansion of an automorphic form on $\widetilde{\SL_2}(\A)$ corresponding to a holomorphic modular form of weight $\ell'$.
	\end{theorem}
	\begin{proof} The proof is to compute $\mathrm{FJ}_{B, \phi \otimes \phi_\infty}(\varphi)(j_{E_{12}}'(g))$ for a specific choice of $\phi_\infty \in S(J(\R))$.  Namely, we will let $\phi_\infty$ be the Gaussian $\phi_0$ of Proposition \ref{prop:GaussianQ}.
		
		For $g \in \widetilde{\SL_2}(\R)$, we must compute the integral
		\[I_{Q,B,d}(\phi_\infty;g) = \int_{J(\R)}{ W_{2\pi(0,B,0,d)}(\exp(v_2 \otimes x) \overline{g}), \omega_{\psi_B}(g)\phi_\infty)(x)\,dx}\]
		when $\phi_\infty=\phi_{0}$ is the Gaussian.  More exactly, for $v \in \Vell$ appropriately chosen, we must compute $\langle I_{Q,B,d}(\phi_\infty,g), v \rangle_{K_J}$, where $\langle \,,\, \rangle_{K_J}$ is the $K_J$-invariant pairing on $\Vell$.  Before specializing to the Gaussian, we work a bit more generally.
		
		We begin by manipulating $I_{Q,B,d}(\phi_\infty,g)$ when $g = h(y) := \exp(\log(y)(E_{11}-E_{22})/2)$, $y \in \R_{>0}$. We have $h(y)^{-1} \cdot v_2 X = y^{1/2} v_2 \otimes X$ and $\omega(h(y))\phi(X) = y^{\dim(J)/4} \phi(y^{1/2} X)$.  Finally, $\nu(h(y)) = y^{1/2}$, as one checks by acting on $E_{13}$.  Thus, making a change of variable,
		\begin{align*} I_{Q,B,d}(\phi_\infty;h(y)) &= (y^{1/2})^{\ell'} \int_{J(\R)}{W_{2\pi (0,B,0,d) \cdot h(y)}(\exp(v_2 \otimes x)) \phi_\infty(x)\,dx} \\
			&= (y^{1/2})^{\ell'} \int_{J(\R)}{W_{2 \pi (0,B,0, yd)}(\exp(v_2 \otimes x)) \phi_\infty(x)\,dx}
		\end{align*}
		as $(0,B,0,d) \cdot h(y) = \nu(h(y)) h(y)^{-1} \cdot (0,B,0,d) = (0,B,0, y d).$
		
		We now write out $I_{Q,B,d}(\phi_\infty,h(y))$ more explicitly.  First observe that, for $w = 2\pi (0,B,0,dy)$, we have $\alpha_w(\exp(v_2 \otimes x))^* = -2 \pi ((B,(x+i1)^\#)+dy)$.  Consequently, if $B > 0$ and $d < 0$, then
		\[I_{Q,B,d}(\phi_\infty;h(y)) = (-1)^v y^{\ell'/2} \int_{J(\R)}{\phi_\infty(x) \left(\frac{(2\pi B,(x+i1)^\#) - \mu}{|(2\pi B,(x+i1)^\#)-\mu|}\right)^v K_v(|(2\pi B,(x+i1)^\#)-\mu|)\,dx}\]
		where $\mu_0 = 2\pi|dy|$.  
		
		Let $V \subseteq J(\R)$ be the set of $X \in J(\R)$ so that $(B,1_J, X)_J = 0$.  One has $J(\R) = \R \cdot 1_J \oplus V$.  Assume that $\phi_\infty(t_0 1 + v) = \phi_{1}(t_0)\phi_{V}(v)$ for Schwartz functions $\phi_1$, $\phi_V$ on $\R \cdot 1$ and $V$, respectively.  If $x = t_0 1 + v$, then $(2\pi B,(x+i1)^\#) = 2\pi (B,1) (t_0+i)^2 + 2\pi (B,v^\#)$.    Let $\lambda >0$ be the squareroot of $2\pi (B,1)$.  Set $\mu_1(v) = \lambda^{-2}(\mu_0+ |(2\pi B,v^\#)|)$.  Then we wish to evaluate
		\[\int_{\R \times V}{e^{-\lambda^2 t_0^2} \phi_V(v) \left(\frac{(t_0+i)^2-\mu_1(v)}{|(t_0+i)^2-\mu_1(v)|}\right)^v K_v(\lambda^2( (t_0+i)^2-\mu_1(v)|)\,dt_0\,dv}.\]
		Set $\mu(v) = \lambda^2 \mu_1 = 2\pi |dy| + 2\pi |(B,v^\#)|$.  Consequently, making a change of variables,
		\begin{align*} I_{Q,B,d}(e^{-2\pi (B,1)t_0^2} &\otimes \phi_V(v);h(y)) = C_B (-1)^{v} y^{\ell'/2} \\
			&\,\,\,\ \times  \int_{\R \times V}{e^{-t_0^2} \phi_V(v) \left(\frac{(t_0+\lambda^2 i)^2-\mu(v)}{|(t_0+\lambda^2 i)^2-\mu(v)|}\right)^v K_v(( (t_0+\lambda^2 i)^2-\mu(v)|)\,dt_0\,dv}\end{align*}
		for a positive constant $C_B$ that only depends on $B$.  We now apply Theorem \ref{thm:defInt1} to obtain that 
		\[I_{Q,B,d}(\phi_0;h(y)) = C'' y^{\ell'/2} e^{-2 \pi y |d|} \otimes (x+y)^{2\ell}\]
		if $B > 0$ and $d < 0$, for some nonzero complex number $C''$ that depends on $B$ but does not depend on $y$ or $d$ or $v$.  Here $(x+y)^{2\ell} \in \Vell$ and this $y$ is not to be confused with the $y$ in $h(y)$.
		
		We claim that $\langle I_{Q,B,d}(\phi_{0};g), (x-y)^{2\ell} \rangle_{K_J} = C'' \mathcal{W}_{\SL_2, \ell',|d|}(g)$.  This claim finishes the proof of the theorem.  To see the claim, note that we have already computed that $I_{Q,B,d}(\phi_0;g)$ has the same left-equivariance property under $\exp(\R E_{12})$ as does $\mathcal{W}_{\SL_2,\ell',|d|}$.  By our computation just made, they have the same restriction to the subgroup of $h(y)$'s, $y \in \R_{>0}$.  They also have the same restriction to $\widetilde{\SO(2)}$.  To see this last fact, observe that, in the notation of \cite{pollackQDS}, \[pr_{\su_2}(i (E_{12}-E_{21})) = pr_{\su_2}(u(-iv_3)) = \mathfrak{so}_3(-i v_3) = \mathfrak{so}_3(v_1-iv_3)/2 +  \mathfrak{so}_3(-v_1-iv_3)/2 =\frac{1}{2}(e_{\ell}+f_{\ell}).\]
		Thus $(E_{12}-E_{21}) \cdot (x-y)^{2\ell} = i \ell (x-y)^{2\ell}$.  Applying Proposition \ref{prop:GaussianQ} gives the result.
	\end{proof}
	
	\subsection{The Gaussian}
	The purpose of this subsection is to prove the following result.
	\begin{proposition}\label{prop:GaussianQ} Assume $B > 0$ is positive definite. Let $V = \{v \in J: (B,1,v) = 0\}$.  Define $\phi_0$ on $J(\R)$ as 
		\[\phi_{0}(t_0 1 + v) = e^{-2 \pi (B,1) t_0^2} e^{2 \pi (B,v^\#)}.\]
		where $t_0 \in \R$ and $v \in V$.  Then $v \mapsto (B,v^\#)$ is a negative-definite function on $V$, so that $\phi_0$ is a Gaussian.  In particular, $\phi_0 \in S(J(\R))$.  Let $d \omega_{\psi_B}$ denote the differential of the Weil representation $\omega_{\psi_B}$.  Then
		\[d\omega_{\psi_B}(E_{12}-E_{21}) \phi_0 = \left(\frac{-i}{2}\right)(\dim(J)-2)\phi_0.\]
	\end{proposition}
	We begin by establishing the fact that $v \mapsto (B,v^\#)$ is a negative-definite quadratic form on $V$.
	
	\begin{lemma} If $B,\sigma$ are positive definite, and $(v, B \times \sigma) = 0$, then $(B, v^\#) \leq 0$, with equality if and only if $v=0$.\end{lemma}
	\begin{proof} Observe that the condition $(B,\sigma, v) = 0$ is invariant under the action of $M_J$, and likewise the quantity $(B,v,v)$ is invariant under this action.   Thus we can use the $M_J$ action to assume that $B = 1$.  Then, we can use the $A_J$ action to assume that $\sigma = (\alpha_1, \alpha_2, \alpha_3)$ is diagonal with $\alpha_j > 0$ for each $j$.  In this case, $B \times \sigma = \diag(\alpha_2+\alpha_3, \alpha_1+\alpha_3, \alpha_1+\alpha_2)$.  
		
		If $v$ has diagonal entries $(v_1, v_2, v_3)$, then $(B, v^\#) = (1,v^\#) =  v_1 v_2 + v_2 v_3 + v_3 v_1$.  That $(B,\sigma, v) = 0$ means that $(\alpha_2+\alpha_3)v_1 + (\alpha_1+\alpha_3)v_2 + (\alpha_1+\alpha_2)v_3 = 0$.  Solving for $v_3$, we obtain
		\[ (B,v^\#) = v_1 v_2 - (v_1+v_2)\left(\left(\frac{\alpha_2+\alpha_3}{\alpha_1 + \alpha_2}\right) v_1 + \left(\frac{\alpha_1+\alpha_3}{\alpha_1 + \alpha_2}\right)v_2\right).\]
		Consequently,
		\begin{align*} (\alpha_1+\alpha_2)(B,v^\#) &= (\alpha_1+\alpha_2)v_1v_2 - (\alpha_2+\alpha_3)v_1(v_1+v_2) - (\alpha_1+\alpha_3) v_2 (v_1+v_2) \\
			&= - [(\alpha_2+\alpha_3)v_1^2 + (\alpha_1+\alpha_3)v_2^2]-2\alpha_3 v_1 v_2 \\ &= -\alpha_2 v_1^2 - \alpha_1 v_2^2 - \alpha_3 (v_1+v_2)^2 \\
			&\leq 0.
		\end{align*}
		This completes the proof.
	\end{proof}
	
	We will now compute $d\omega_{\psi_B}(E_{12}-E_{21})\phi_0$.  To do this, we work a little more generally.  Suppose then that $B \in J$ has nonzero norm.  Consider the map $J \rightarrow J^\vee$ given by $x \mapsto B \times x$.  This map is invertible.  In fact, set $\iota_B: J^\vee \rightarrow J$ given by $\iota_B(y) = N_J(B)^{-1}(B^\# \times y - \frac{1}{2}(B,y)B)$.  Then $\iota_B$ is the inverse of $x \mapsto B \times x$.
	
	We define a differential operator on the Schwartz space $S(J(\R))$, called $D_B$, as follows.  Let $J = \R \times V$ be our decomposition of $J$, where $V = \{v \in J: (B,1,v) = 0\}$.  Let $\{x_0,x_1, \ldots, x_r\}$ be a basis of $J$, with $x_0$ spanning $\R 1$ and $x_1, \ldots, x_r$ a basis of $V$.  Define $\widetilde{D}: C^\infty(J(\R)) \rightarrow C^\infty(J(\R)) \otimes J^\vee$ as 
	\[\widetilde{D}f = \sum_{j} \partial_{x_j} f \otimes x_j^\vee.\]
	The operator $\widetilde{D}$ is independent of the choice of basis.
	
	We let $\widetilde{D} \circ \widetilde{D}: C^\infty(J(\R)) \rightarrow C^\infty(J(\R)) \otimes J^\vee \otimes J^\vee$ be the composition of $\widetilde{D}$ with itself, i.e., 
	\[\widetilde{D} \circ \widetilde{D}f = \sum_{j,k}\partial_{x_j} \partial_{x_k} f \otimes x_k^\vee \otimes x_j^\vee.\]
	We now set 
	\[D_B = \frac{1}{4\pi i} \mathrm{pair}((\iota_B \circ 1)(\widetilde{D} \circ \widetilde{D})).\]
	That is,
	\[D_B(f) = \frac{1}{4\pi i} \sum_{j,k}(\iota_B(x_k^\vee), x_j^\vee)\partial_{x_j} \partial_{x_k} f .\]
	
	Assume $B > 0$ is positive definite.  We define a Gaussian $\phi_0$ on $J(\R)$ as 
	\[\phi_{0}(t_0 1 + v) = e^{-C_0 \pi (B,1) t_0^2} e^{C_V \pi (B,v^\#)}\]
	for positive constants $C_0, C_V$ to be determined.
	
	\begin{lemma}\label{lem:DB1} One has
		\[(4\pi i) D_B\phi_0 = (-C_0 \pi + 2 C_0^2 \pi^2 (B,1) t_0^2 + C_V \pi (\dim(J)-1) + 2C_V^2 \pi^2 (B,v^\#)) \phi_0.
		\]
	\end{lemma}
	\begin{proof} Let $x_0 = 1$.  Then $x_0^\vee = \frac{1}{2(B,1)}(1 \times B)$.  We compute from the definition.
		
		One has 
		\[\widetilde{D} \phi_0 = (-2 C_0 \pi (B,1) t_0 \otimes x_0^\vee + (\sum_{j=1}^{r}C_V \pi (B,x_j,v) \otimes x_j^\vee)) \phi_0.\]
		Differentiating again,
		\begin{align*}\widetilde{D}^2 \phi_0 &= -2C_0 \pi (B,1) \phi_0 \otimes x_0^\vee \otimes x_0^\vee + 4 C_0^2 \pi^2 (B,1)^2 t_0^2 \phi_0 \otimes x_0^\vee \otimes x_0^\vee \\
			&\,\,\,+ \sum_{j=1}^{r}(-2C_0 \pi (B,1) t_0)(C_V \pi (B,x_j,v)) \phi_0 \otimes (x_0^\vee \otimes x_j^\vee + x_j^\vee \otimes x_0^\vee) \\
			&\,\,\, + \sum_{j,k = 1}^{r} C_V \pi (B,x_j, x_k) \phi_0 \otimes x_j^\vee \otimes x_k^\vee \\
			&\,\,\, + \sum_{j,k = 1}^{r} C_V^2 \pi^2 (B,x_j, v)(B,x_k,v) \phi_0 \otimes x_j^\vee \otimes x_k^\vee
		\end{align*}
		We have $x_0^\vee \otimes x_0^\vee = (2(B,1))^{-2}((1 \times B) \otimes (1 \times B))$.  Additionally, 
		\[\sum_{j,k}(B,x_j,x_k) \otimes x_j^\vee \otimes x_k^\vee = \sum_{k}(B \times x_k) \otimes x_k^\vee\]
		and
		\[\sum_{j,k}(B,x_j,v)(B,x_k,v) \otimes x_j^\vee \otimes x_k^\vee = (B \times v) \otimes (B \times v).\]
		We now apply $\iota_B$.  To do this, note the following facts:
		\begin{itemize}
			\item $(\iota_B(x_0^\vee),x_0^\vee) = (2(B,1))^{-2}(1, 1 \times B) = (2(B,1))^{-1}$.
			\item For $j \geq 1$, $(\iota_B(x_0^\vee), x_j^\vee) = (2(B,1))^{-1}(1, x_j^\vee) = 0$.
			\item For $j \geq 1$, $(\iota_B(x_j^\vee), x_0^\vee) = (2(B,1))^{-1}(\iota_B(x_j^\vee), B \times 1) = (2(B,1))^{-1}(x_j^\vee,1) = 0.$  Here we are using that $(\iota_B(B \times y), B \times y') = (y, B \times y') = (B \times y, y')$.
			\item $\sum_{k=1}^{r}(\iota_B(B \times x_k),x_k^\vee) = \dim(V) = \dim(J) -1$.
			\item $(\iota_B(B \times v), B \times v) = 2(B,v^\#).$
		\end{itemize}
		Applying these computations, we obtain
		\[(4\pi i) D_B\phi_0 = (-C_0 \pi + 2 C_0^2 \pi^2 (B,1) t_0^2 + C_V \pi (\dim(J)-1) + 2C_V^2 \pi^2 (B,v^\#)) \phi_0.
		\]
	\end{proof}
	If $C_V = C_0$, then
	\[(4\pi i)D_B \phi_0(x) = \pi C_0(\dim(J)-2) \phi_0(x) + 2 C_0^2 \pi^2 (B,x^\#) \phi_0(x).\]
	
	We also must compute $D_B$ applied to the exponentials $e^{2 \pi i C(B, x \times y)}$, as a function of $x$.  The result is:
	\begin{lemma}\label{lem:DB2} One has
		\[D_B(e^{2\pi i (B \times y, x)}) = 2 \pi i C^2 (B,y^\#) e^{2\pi i C(B \times y,x)}.\]
	\end{lemma}
	\begin{proof} We have 
		\[\widetilde{D} e^{2\pi i C (B \times y,x)} = 2 \pi i C e^{2\pi i C(B \times y,x)} \otimes (B \times y)\]
		and so
		\[\widetilde{D}^2 e^{2\pi i C (B \times y,x)} = (2 \pi i C)^2 e^{2\pi i C(B \times y,x)} \otimes (B \times y) \otimes (B \times y).\]
		Thus 
		\[D_B(e ^{2\pi i C(B \times y,x)}) = 2 \pi i C^2 (B,y^\#)e^{2\pi i C(B \times y,x)}.\]
	\end{proof}
	
	Let $C_0 = C_V = 2$.  We can now compute how the Lie algebra element $E_{12}-E_{21}$ acts on the Gaussian $\phi_0(x)$ via the Weil representation.
	\begin{proposition} Let $C_0 = C_V = 2$.  Let $d \omega_{\psi_B}$ denote the differential of the Weil representation $\omega_{\psi_B}$.  Then
		\[d\omega_{\psi_B}(E_{12}-E_{21}) \phi_0 = \left(\frac{-i}{2}\right)(\dim(J)-2)\phi_0.\]
	\end{proposition}
	\begin{proof} For $\phi \in S(J(\R))$, one has $\exp(u E_{12}) \phi(x)= e^{2\pi i u (B,x^\#)} \phi(x)$.  Thus $d \omega_{\psi_B}(E_{12}) \phi(x) = 2 \pi i (B,x^\#) \phi(x)$.  In $\sl_2$, we have $-E_{21} = Ad(J_2)(E_{12})$, where $J_2  = \mm{0}{1}{-1}{0}$.    Up to scalar multiple, the element $J_2$ acts via the Weil representation as Fourier transform for the non-degenerate pairing on $J$ given by $x,y \mapsto (B, x \times y)$.  Precisely, there is a nonzero complex number $\gamma$ so that if $J_2' \in \widetilde{\SL_2}(\R)$ maps to $J_2$, and $\phi \in S(J(\R))$, then
		\[ \omega_{\psi_B}(J_2') \phi(y) = \gamma \int_{J(\R)}{ e^{-2\pi i (B, x \times y)} \phi(x)\,dx}.\]
		Applying Lemma \ref{lem:DB2}, one has $d\omega_{\psi_B}(-E_{21})\phi = D_B \phi$.  Indeed, let $\phi = \omega_{\psi_B}(J_2') \phi'$.  Then 
		\[d \omega_{\psi_B}(-E_{21})\phi(y) = (2 \pi i (B,y^\#))\gamma \int_{J(\R)}{e^{-2\pi i (B,x \times y)} \phi'(x)\,dx}\]
		while
		\[D_B \phi(y) = \gamma \int_{J(\R)}{ e^{-2\pi i(B, x \times y)} D_B \phi'(x)\,dx}.\]
		Applying Lemma \ref{lem:DB1} gives the result.
	\end{proof}	
	
	\section{The Fourier-Jacobi expansion for $R$ I: Splittings}\label{sec:splittings}
	In this section, we begin to develop the Fourier-Jacobi expansion along the parabolic subgroup $R \subseteq G_J$.
	
	\subsection{Preliminaries}
	Recall from subsection \ref{subsec:parabolic} the parabolic subgroup $R$, together with its Levi decomposition $R = M_R N_R$. Let us write $Lie(N_R) = V_8 \oplus V_7$, where $V_8$ is the subspace with $h_R$-eigenvalue $1$ and $V_7$ is the subspace with $h_R$ eigenvalue $2$.  One has $V_8 = C^8$ and $V_7 = \Q^3 \oplus C \oplus \Q^3$.  Thus, if $G_J = F_4$, then $V_8$ has dimension $8$ and $V_7$ has dimension $7$.  For other groups, these are not the dimensions.  If $G_J$ is an orthogonal group, then $V_8 = 0$.  For $? \in \{8,7\}$ and $k \in \{0,1,2\}$ we let $V_{?}^{[k]}$ denote the subspace of $V_{?}$ with $h_{P}$ eigenvalue equal to $k$.  Then $V_8 = V_8^{[0]} \oplus V_8^{[1]}$, while $V_7 = V_7^{[0]} \oplus V_7^{[1]} \oplus V_7^{[2]}$.
	
	In case $J = H_3(C)$, for $x_j \in C$, let $V(x_1,x_2, x_3) = \left(\begin{array}{ccc} 0 & x_3 & x_2^* \\ x_3^* & 0 & x_1 \\ x_2 & x_1^* & 0 \end{array}\right)$.  Set $E = 1_J- e_{11} \in J,J^\vee$.  Set $H_2(C) = \mathrm{Span}(e_{22},e_{33}, V(x_1,0,0)) \subseteq J, J^\vee$.  One has
	\begin{itemize}
		\item $V_8^{[0]} = \mathrm{Span}(\Phi_{E,V(0,u_2,u_3)},v_2 \otimes V_{(0,u_2',u_3')})$
		\item $V_8^{[1]} = \mathrm{Span}(v_1 \otimes V(0,v_2,v_3),\delta_3 \otimes V(0,v_2',v_3'))$
		\item $V_7^{[0]} = \mathrm{Span}(v_2 \otimes e_{11})$
		\item $V_7^{[1]} = \mathrm{Span}(v_1 \otimes e_{11}, \delta_3 \otimes H_2(C),E_{23})$
		\item $V_7^{[2]} = \mathrm{Span}(E_{13})$.
	\end{itemize}
	
	The group $M_R$ preserves a quadratic form on $V_7$, up to similitude.  Namely, define the elements $b_i, b_{-j}$ by an equality of lists
	\[(b_1, b_2, b_3, b_{-3},b_{-2},b_{-1}) =  (E_{13}, v_1 \otimes e_{11}, \delta_{3} \otimes e_{22}, \delta_{3} \otimes e_{33}, -E_{23}, v_2 \otimes e_{11}).\]
	A general element $v$ of $V_7$ can be written as $v = \left(\sum_{j \in \{\pm 1,\pm 2,\pm 3\}}{\alpha_j b_j}\right) + \delta_3 \otimes V(\beta,0,0)$ with $\alpha_j \in \Q$ and $\beta \in C$.  Define 
	\[q_{V_7}(v) = \alpha_1 \alpha_{-1} + \alpha_2 \alpha_{-2} + \alpha_3 \alpha_{-3} - n_C(\beta).\]
	
	\begin{proposition} The adjoint action of $M_R$ on $V_7$ preserves the quadratic form $q_{V_7}$ up to scaling. \end{proposition}
	\begin{proof} This is proved in \cite[Proposition 5.2.1]{pollackNTM} in case $G_J = E_8$, and the general case can be proved by the same argument.\end{proof}
	
	We will need, at various points below, a computation of the commutator $[x,y]$, if $x \in V_{8}^{[0]}$ and $y\in V_8^{[1]}$.  We do this now.  Suppose $u = (u_2,u_3) \in C^2$, and likewise $u' = (u_2',u_3'), v = (v_2,v_3), v'= (v_2',v_3') \in C^2$.  We write $(u,u')_X$ for the element $\Phi_{E,V(0,u_2,u_3)} + v_2 \otimes V(0,u_2',u_3')$ of $V_{8}^{[0]}$ and $(v,v')_Y$ for the element $v_1 \otimes V(0,v_2,v_3) + \delta_3 \otimes V(0,v_2',v_3')$ of $V_{8}^{[1]}$.
	\begin{lemma}\label{lem:commutatorV8} Let $(u,v) = (u_2,v_2)_C + (u_3, v_3)_C$ and likewise for $(u',v')$.  Then
		\[ [(u,u')_X, (v,v')_Y] = (u,v) v_1 \otimes e_{11} + \delta_3 \otimes (e_{11} \times (u \times v')) - \delta_{3} \otimes (u' \times v) - (u',v') E_{23}.\]
	\end{lemma}
	\begin{proof} This is a direct computation.  Indeed, computing from the definition and using $E \times V(0,A,B) = - V(0,A,B)$ gives
		\begin{align*}
			[(u,u')_X, (v,v')_Y] &= v_1 \otimes (-E \times (V(0,u_2,u_3) \times V(0,v_2,v_3))) \\ &\,\,\,+ \delta_3 \otimes (-V(0,u_2, u_3) \times V(0,v_2',v_3') - (V(0,u_2,u_3), V(0,v_2',v_3'))E) \\ &\,\,\, - \delta_3 \otimes (V(0,u_2',u_3') \times V(0,v_2,v_3)) - E_{23} \cdot (V(0,u_2',u_3'), V(0,v_2',v_3')).
		\end{align*}
		This is then seen to be equal to the quantity in the statement of the lemma.	
	\end{proof}
	
	Fix $T \in V_7$.  Define an alternating pairing $\langle \,, \, \rangle_{V_8,T}$ on $V_8$ via $\langle w_1, w_2 \rangle_{V_8,T} = (T, [w_1, w_2])_{V_7}$.   If $T \in V_{7}$, we say that $T$ is \emph{normal} if $T \in \mathrm{Span}(\delta_3 \otimes H_2(C))$ and $q_{V_7}(T) \neq 0$.
	\begin{lemma}\label{lem:symplecticV8} Suppose $T = \delta_3 \otimes T' \in V_{7}$ is normal.   Then $\langle \,,\, \rangle_{V_8,T}$ is non-degenerate and $V_8 = V_{8}^{[0]} \oplus V_{8}^{[1]}$ is a Lagrangian decompsotion.  One has
		\[ \langle (u,u')_X, (v,v')_Y \rangle_{V_8,T} = (T', u \times v')_J - (T, e_{11} \times (u' \times v))_J.\]
	\end{lemma}
	\begin{proof} This follows from Lemma \ref{lem:commutatorV8}.
	\end{proof}
	
	From now on, we assume that $T \in V_{7}$ is normal.  Define $J_{2,V_8}: V_8 \rightarrow V_8$ as $J_{2,V_8}((u,u')_X) = (u',-u)_{Y}$ and $J_{2,V_8}((v,v')_{Y}) = (v',-v)_{X}$.  We observe that $J_{2,V_8}$ preserves the symplectic form $\langle \,,\, \rangle_{V_8,T}$.  Moreover, $J_{2,V_8}^2 = -1_{V_8}$ is negative the identity on $V_8$.
	
	The group $N_R$, together with a non-degenerate normal element $-T$, gives a Heisenberg group.  Namely, we have a linear map $V_7 \rightarrow \mathbf{G}_a$ given by $v \mapsto (-T,v)_{V_7}$, where $(x,y)_{V_7} = q_{V_7}(x+y) - q_{V_7}(x) - q_{V_7}(y)$ is the bilinear form associated to the quadratic form $q_{V_7}$.  Let $\chi_T$ denote the character on $V_7$ as $\chi_T(v) = \psi(-(T,v)_{V_7})$. Let $\omega_{\chi_T}$ be the Weil representation of $N_R \rtimes \Sp(V_{8}, \langle \,,\, \rangle_{V_8,T})$ associated to this linear map on $V_7$.  We associate this representation to the Langrangian decomposition $V_8 = V_{8}^{[0]} \oplus V_{8}^{[1]}$, so that is acts on $S(V_8^{[0]}(\A))$.  If $\phi \in S(V_8^{[0]}(\A))$, we let 
	\[\Theta_{\phi}(hg) = \sum_{\xi \in V_8^{[0]}(\Q)}{(\omega_{\chi_T}(hg)\phi)(\xi)}\]
	be the theta function.
	
	\subsection{Splittings: non-commutative case}
	Let $M_R^T$ denote the subgroup of $M_R$ that stabilizes $T \in V_7$, and $M'$ its derived group.  (For ease of notation, we drop the $R,T$, even though this group does depend on $T$.)  Then we have a natural map $M_R^T \rightarrow \Sp(V_8, \langle \,,\, \rangle_{V_8,T})$.  The purpose of the rest of this section is to show that when $J = H_3(C)$, there is a splitting $M' \rightarrow \widetilde{\Sp(V_8)}$ into the double cover.  
	
	Let $k = \Q_v$ be a completion of $\Q$.  We now take up the task of providing a splitting $M_R^T(k) \rightarrow \widetilde{\Sp(V_8)}(k)$ when $\dim(C) \geq 4$, i.e., when $C$ is noncommutative.  The result uses the Rao cocycle \cite[Theorem 5.3]{rao}, which we review now.  For $g_1, g_2 \in \Sp(V_8)(k)$, one has $c_{Rao}(g_1, g_2) \in \mu_2$, and the map $c_{Rao}: \Sp(V_8)(k) \times \Sp(V_8)(k) \rightarrow \mu_2$ is a coycle, i.e., 
	\[c_{Rao}(g_1,g_2) c_{Rao}(g_1g_2,g_3) = c_{Rao}(g_1,g_2g_3) c_{Rao}(g_2,g_3).\]
	
	To define $c_{Rao}(g_1,g_2)$, we need a few preliminaries.  First, fix a symplectic basis $e_1, \ldots, e_n$, $e_1^*, \ldots, e_n^*$, of $V_8 = X \oplus Y$, with $X = \mathrm{Span}(e_1, \ldots, e_n)$ and $Y = \mathrm{Span}(e_1^*,\ldots,e_n^*)$.  For a subset $S \subseteq \{1,\ldots,n\}$, let $\tau_S$ be the element of $\Sp(V_8)$ (acting on the right of $V_8$) that takes $e_j^* \mapsto e_j$, $e_j \mapsto - e_j^*$ for $j \in S$ and is the identity on the other basis elements.	 See before Lemma 2.13 in \cite{rao}.
	
	Let $P_{V_8}$ denote the Siegel parabolic subgroup of $\Sp(V_8)$, which by definition stabilizes $Y$ for the right action.  For an integer $j \in \{0,1,\ldots,n\}$, let $\Omega_j$ be the subset of $\Sp(V_8)$ from \cite[Lemma 2.14]{rao}, so that $\Omega_j = P_{V_8} \tau_{S} P_{V_8}$ for any $S$ with $|S| = j$.  One has $\Sp(V_8)$ is the disjoint union of the $\Omega_j$, $j = 0, 1, \ldots, n$.
	
	Let $x: \Sp(V_8)(k) \rightarrow k^\times/(k^\times)^2$ be the map of \cite[Lemma 5.1]{rao}.  On $\Omega_j$, $x(p_1 \tau_S p_2) = \det(p_1 p_2|Y)$ if $|S| = j$.  For $g_1, g_2 \in \Sp(V_8)(k)$, let $\rho = q(Y,Yg_2^{-1},Y g_1)$ be the Leray invariant of these three isotropic subspaces; see \cite[Definition 2.10]{rao}.  Finally, if $g_1 \in \Omega_{j_1}$, $g_2 \in \Omega_{j_2}$ and $g \in \Omega_{j}$, let $\ell$ be the integer (\cite[Proof of Theorem 5.3]{rao}) satisfying $2\ell = j_1+j_2-j-\dim(\rho)$.  If $h(\rho)$ denotes the Hasse invariant of $\rho$, then
	\[c_{Rao}(g_1,g_2) = (x(g_1),x(g_2)) (-x(g_1g_2), x(g_1)x(g_2)) ((-1)^{\ell},\det(\rho)) (-1,-1)^{\ell (\ell-1)/2} h(\rho).\]
	
	The double cover $\widetilde{\Sp(V_8)}(k)$ is the set $\Sp(V_8)(k) \times \{\pm 1\}$ with multiplication $(g_1,\epsilon_1) (g_2, \epsilon_2) = (g_1g_2, \epsilon_1\epsilon_2 c_{Rao}(g_1, g_2))$.  It carries a Weil representation on the Schwartz-Bruhat space $S(X(k))$.  Our aim for the rest of the section is to prove the following result.
	
	\begin{proposition}\label{prop:splittingC} Suppose $J = H_3(C)$ with $\dim(C) \geq 4$.  Then $c_{Rao}(g_1, g_2) = 1$ for all $g_1, g_2 \in M_R^T(k)$.  Consequently, the map $M_R^T(k) \rightarrow \widetilde{\Sp(V_8)}(k)$ given by $g \mapsto (g,1)$ is a group homomorphism.
	\end{proposition}
	
	We begin with some $\SL_2$'s that map to $M_R$.  Define
	\begin{itemize}
		\item $e_1 = E_{12}$, $f_1 = E_{21}$, $h_1 = [e_1, f_1] = E_{11}-E_{22}$
		\item $e_2 = \delta_3 \otimes e_{11}$, $f_2 = - v_3 \otimes e_{11}$, 
		\[h_2 = [e_2, f_2] = \frac{1}{3}(E_{11}+E_{22}-2E_{33})- \Phi'_{e_{11},e_{11}}.\]
	\end{itemize}
	\begin{lemma}\label{lem:sl2triple} For $j = 1,2$, $e_j, h_j, f_j$ is an $\sl_2$-triple.  If $T \in \delta_3 \otimes H_2(C)$, then these $\sl_2$-triples are in $Lie(M_R^T)$.	\end{lemma}
	\begin{proof} That the $e_j, h_j, f_j$ form an $\sl_2$-triple is clear for $j=1$, and is immediately checked from the definition for $j=2$.  To see that they live in $Lie(M_R)$, recall that $h_R = \frac{2}{3}(E_{11}+E_{22}-2E_{23}) + \Phi'_{e_{11},e_{11}}$, and $Lie(M_R)$ is the $0$-eigenspace of the adjoint action of $h_R$ on $\g(J)$. One computes that $h_R$ commutes with $e_1, f_1, e_2, f_2$, so these $\sl_2$-triples lie in $Lie(M_R)$.  Finally, to see that they are in $Lie(M_R^T)$. one uses that $\Phi'_{X,e_{11}} = 0$ if $X \in H_2(C)$.
	\end{proof}
	
	For $j=1,2$, from Lemma \ref{lem:sl2triple}, we have corresponding maps $\iota_j: \SL_2 \rightarrow M_R^T$.   We use the $\sl_2$-triples of Lemma \ref{lem:sl2triple} to help give a Bruhat decomposition for $M_R^T$.  To do this, note that for $t \in k^\times$, one has
	\begin{equation}\label{eqn:wt} \mb{1}{t}{}{1} \mb{1}{}{-t^{-1}}{1} \mb{1}{t}{}{1} = \mb{}{t}{-t^{-1}}{}.\end{equation}
	Let $w_1   = \iota_1(\mm{}{1}{-1}{})$, $w_2 = \iota_2(\mm{}{1}{-1}{})$, and $w_3 = w_1 w_2 = w_2 w_1$.  Let $P_R^T = P \cap M_R^T$, where $P$ is the Heisenberg parabolic subgroup of $G_J$.   The group $P_R^T$ is a parabolic subgroup of $M_R^T$.
	
	\begin{lemma} Assume $q_{V_7}(T) > 0$.  One has a Bruhat decomposition
		\[ M_R^T = P_R^T \sqcup P_R^T w_1 P_R^T \sqcup P_R^T w_3 P_R^T.\]
	\end{lemma}
	\begin{proof} Recall the elements $b_1 = E_{13}, b_{-1} = v_2 \otimes e_{11}$, $b_2  = v_1 \otimes e_{11}, b_{-2} = -E_{23}$ of $\g(J)$.  By the assumption on $T$, the group $M_R^T$ is isogenous to a special orthogonal group of a quadratic space with Witt rank two.  The parabolic group $P_R^T$ stabilizes an isotropic line $\Q b_1$ in the orthogonal representation.  Thus the lemma follows from the Bruhat decomposition on the special orthogonal group, as soon as we see how $w_1, w_3$ act on $V_7$.
		
		Thus, we must compute how $w_1, w_2$ act on $b_1, b_2, b_{-2}, b_{-1}$ and all elements of the form $\delta_3 \otimes X$ with $X \in H_2(C)$.  For the latter, note that $e_1, f_1, e_2, f_2$ annihilate all of $\delta_3 \otimes H_2(C)$, so $w_1, w_2$ fix these Lie algebra elements.
		
		By \eqref{eqn:wt}, 
		\begin{align*} Ad(w_1) &= \exp(ad(e_1)) \exp(-ad(f_1)) \exp(ad(e_1)),\\
			Ad(w_2) &= \exp(ad(e_2)) \exp(-ad(f_2)) \exp(ad(e_2)).\end{align*}
		We compute: 
		\begin{align*}
			w_1(b_1) &= \exp(ad(e_1)) \exp(-ad(f_1))\exp(ad(e_1)) E_{13} = \exp(ad(e_1)) \exp(-ad(f_1))E_{13} \\ &= \exp(ad(e_1))(E_{13}-E_{23}) = - E_{23} = b_{-2}\\
			w_1(b_2) &= \exp(ad(e_1)) \exp(-ad(f_1))\exp(ad(e_1)) v_1 \otimes e_{11} = \exp(ad(e_1))\exp(-ad(f_1)) v_1 \otimes e_{11} \\ &= \exp(ad(e_1))(v_1 \otimes e_{11}-v_2 \otimes e_{11}) = - v_2 \otimes e_{11} = - b_{-1}\\
			w_1(b_{-1}) &= \exp(ad(e_1)) \exp(-ad(f_1))\exp(ad(e_1)) v_2 \otimes e_{11} \\ &= \exp(ad(e_1)) \exp(-ad(f_1))(v_2 \otimes e_{11}+v_1 \otimes e_{11}) \\ &=\exp(ad(e_1)) v_1 \otimes e_{11} = v_1 \otimes e_{11} = b_2
		\end{align*}
		\begin{align*}
			w_1(b_{-2}) &= \exp(ad(e_1)) \exp(-ad(f_1))\exp(ad(e_1))(-E_{23}) = \exp(ad(e_1)) \exp(-ad(f_1))(-E_{23}-E_{13}) \\ &= \exp(ad(e_1))(-E_{13}) = - E_{13} = -b_1.\\
			w_2(b_1) &= \exp(ad(e_2)) \exp(-ad(f_2))\exp(ad(e_2)) E_{13} =  \exp(ad(e_2)) \exp(-ad(f_2)) E_{13} \\ &= \exp(ad(e_2))(E_{13}-v_1 \otimes e_{11}) = - v_1 \otimes e_{11} = -b_2\\
			w_2(b_2) &= \exp(ad(e_2)) \exp(-ad(f_2))\exp(ad(e_2)) v_1 \otimes e_{11} = \exp(ad(e_2)) \exp(-ad(f_2))(v_1 \otimes e_{11}+E_{13}) \\ &= \exp(ad(e_2)) E_{13} = E_{13} = b_1 \\
			w_2(b_{-1}) &= \exp(ad(e_2)) \exp(-ad(f_2))\exp(ad(e_2)) v_2 \otimes e_{11} = \exp(ad(e_2)) \exp(-ad(f_2))(v_2 \otimes e_{11}+E_{23}) \\ &= \exp(ad(e_2)) E_{23} = E_{23} = - b_{-2}\\
			w_2(b_{-2}) &= \exp(ad(e_2)) \exp(-ad(f_2))\exp(ad(e_2))(-E_{23}) = \exp(ad(e_2)) \exp(-ad(f_2))(-E_{23}) \\ &= \exp(ad(e_2))(v_2 \otimes e_{11}-E_{23}) = v_2 \otimes e_{11} = b_{-1}.
		\end{align*}
		The lemma follows.
	\end{proof}
	
	We now compare the $w_1, w_3$ with elements $\tau_S$, $S \subseteq \{1,2,\ldots, n\}$, where $n = 4 \dim(C)$.  We begin by computing the action of $w_1, w_2$ on $V_8 = X \oplus Y$.  Recall $X = \mathrm{Span}\{(u,u')_X\}$, $Y = \mathrm{Span}\{(v,v')_Y\}$ with 
	\[(u,u')_X = \Phi_{E, V(0,u_2,u_3)} + v_2 \otimes V(0,u_2',u_3')\]
	and
	\[(v,v')_Y = v_1 \otimes V(0,v_2,v_3) + \delta_3 \otimes V(0,v_2',v_3').\]
	\begin{lemma}\label{lem:wjactionXY} One has the following identities:
		\begin{enumerate}
			\item $w_1((u,u')_X) = (u,0)_X + (u',0)_Y$;
			\item $w_1((v,v')_Y) = (0,-v)_X + (0,v')_Y$;
			\item $w_2((u,u')_X) = (0,u)_Y + (0,u')_X$;
			\item $w_2((v,v')_Y) = (v,0)_Y + (-v',0)_X$;
			\item $w_3((u,u')_X) = (u',u)_Y$;
			\item $w_3((v,v')_Y) = (-v',-v)_X$.
		\end{enumerate}
	\end{lemma}
	\begin{proof} The last two identities follows from the first four, and the definition $w_3 = w_1 w_2$.
		
		One immediately verifies $w_1((u,0)_X) = (u,0)_X$ and $w_1((0,u')_X) = (u',0)_Y$.  The second identity is similarly checked.  For the third and fourth identities, we compute in detail.  Recall
		\[w_2 = \exp(ad(\delta_3 \otimes e_{11})) \exp(ad(v_3 \otimes e_{11})) \exp(ad(\delta_3 \otimes e_{11})).\]
		Using the equality $\Phi_{E, V(0,u_2,u_3)} = \Phi_{V(0,u_2,u_3),e_{11}}$, one has
		\begin{align*}
			w_2((u,0)_X) &= \exp(ad(e_2)) \exp(ad(-f_2))(\Phi_{E,V(0,u_2,u_3)}+ \delta_3 \otimes V(0,u_2,u_3)) \\ &= \exp(ad(e_2))(\delta_3 \otimes V(0,u_2,u_3)) = \delta_{3} \otimes V(0,u_2,u_3) = (0,u)_Y.
		\end{align*}
		The equality $w_2((0,u')_X) = (0,u')_X$ is immediately verified, as is $w_2((v,0)_Y) = (v,0)_Y$.  Finally,
		\begin{align*}
			w_2((0,v')_Y) &= \exp(ad(e_2)) \exp(-ad(f_2)) \delta_{3} \otimes V(0,v_2',v_3') \\&= \exp(ad(e_2))(\delta_{3} \otimes V(0,v_2',v_3')-\Phi'_{V(0,v_2',v_3'),e_{11}}) \\ &= -\Phi'_{V(0,v_2',v_3'),e_{11}} = -\Phi_{E, V(0,v_2',v_3')} = (-v',0)_X.
		\end{align*}
		This completes the proof.
	\end{proof}
	
	We now relate the elements $w_1, w_2, w_3 \in M_R^T(k)$ to the $\tau_S$.  Recall that $P_{V_8} \subseteq \Sp(V_8)$ denotes the Siegel parabolic subgroup of $\Sp(V_8)$ that stabilizes $Y \subseteq V_8$ for the right action.
	\begin{lemma}\label{lem:wjsqaure} If $j= 1,2$, then there is $m_j, m_j' \in P_{V_8}(k)$ and $S_j \subseteq \{1,2,\ldots, 4\dim(C)\}$ so that $w_j = m_j \tau_{S_j} = \tau_{S_j} m_j'$ with $\det(m_j|Y) \in (k^\times)^2$, $\det(m_j'|Y) \in (k^\times)^2$ and $|S_j| = 2 \dim(C)$.  Let $\tau = \tau_{S}$ for $S = \{1,2,\ldots, 4 \dim(C)\}$.  For $j = 3$, there is $m_3, m_3' \in P_{V_8}(k)$ so that $w_3 = m_3 \tau = \tau m_3'$ with $\det(m_3|Y) \in (k^\times)^2$ and $\det(m_3'|Y) \in (k^\times)^2$.
	\end{lemma}
	\begin{proof} Let $T = \delta_3 \otimes T'$, with $T' = \mm{t_{11}}{t_{12}}{t_{12}^*}{t_22} \in H_2(C)$.  Because $q_{V_{7}}(T) = t_{11}t_{22}-n_C(t_{12}) > 0$, $t_{11} t_{22} \neq 0$.  Applying Lemma \ref{lem:symplecticV8} and Lemma \ref{lem:wjactionXY}, one computes 
		\begin{align*} \langle (v,v')_Y, w_1^{-1}(v,v')_Y \rangle_{V_8,T} &= (T' \times e_{11}, v \times v)_J \\
			&= -2(t_{11}n_C(v_3) + t_{22}n_C(v_2) + (t_{12},v_2, v_3))\\
			&= -2( t_{11}n_C(v_3 + t_{11}^{-1} (t_2 v_2)^*) + (t_{22}-t_{11}^{-1}n_C(t_{22})) n_C(v_2).\end{align*}
		Using $2|\dim(C)$, one obtains that the determinant of this quadratic form is $1 \in k^\times/(k^\times)^2$.  The statements for $w_1$ follow.  The proof of the statements for $w_2$ and $w_3$ are similar.
	\end{proof}
	
	Next, we evaluate $x(p)$ for $p \in P_{R}^T$.
	\begin{lemma}\label{lem:x(p)} If $p \in P_R^T = P \cap M_R^T$, then $\det(p|Y)  = \nu(p)^{2\dim(C)}$.  Consequently, $x(g) = 1 \in k^\times/(k^\times)^2$ for all $g \in M_R^T$.
	\end{lemma}
	\begin{proof} The subspace $Y \subseteq W_J$ is a non-degenerate symplectic subspace of $W_J$.  It is preserved by $M_R^T$.  Thus, if $p \in P_R^T$, $p$ preserves $Y$ and also scales the symplectic form on $Y$ that is the restriction of the one on $W_J$.  The first part of the lemma follows.  The second part now follows from the definition of $x(g)$ and Lemma \ref{lem:wjsqaure}.
	\end{proof}
	
	To prove $c_{Rao}(g_1,g_2) = 1$ for all $g_1, g_2 \in M_R^T(k)$, it suffices to restrict to certain special $g_1, g_2$.  This is made precise in the next two lemmas.
	\begin{lemma}\label{lem:opencell} Suppose $c_{Rao}(x,y) = 1$ whenever $y \in P_R^T w_3 P_R^T$.  Then $c_{Rao}(g_1,g_2) = 1$ for all $g_1, g_2 \in M_R^T$.
	\end{lemma}
	\begin{proof} Recall the cocycle relation 
		\[c_{Rao}(g_1,g_2) c_{Rao}(g_1g_2,g_3) = c_{Rao}(g_1,g_2g_3) c_{Rao}(g_2,g_3).\]
		Choose $g_3$ in the open cell $P_R^T w_3 P_R^T$ so that $g_2 g_3$ is also in the open cell.  Then $c_{Rao}(g_1g_2,g_3)$, $c_{Rao}(g_1,g_2g_3)$, $c_{Rao}(g_2,g_3)$ all equal $1$ by assumption, so $c_{Rao}(g_1, g_2) = 1$ by the cocylce relation.
	\end{proof}
	
	Let $M_{P_R}^T = M_P \cap M_R^T$, $N_{P_R}^T = N_{P} \cap M_R^T$.
	\begin{lemma}\label{lem:cRaoRed1} Suppose $c_{Rao}(w_j,nw_3) =1$ for $n \in N_{P_R}^T(k)$ and $j \in \{1,3\}$.  Then $c_{Rao}(g_1,g_2) = 1$ for all $g_1, g_2 \in M_R^T(k)$.
	\end{lemma}
	\begin{proof} Applying Lemma \ref{lem:x(p)}, one checks that if $g_1 \in P_R^T$, then $c_{Rao}(g_1, g_2) = 1$.  By Lemma \ref{lem:opencell}, we may assume $g_2$ is in the open cell.  By Lemma \ref{lem:x(p)} and \cite[Corollary 5.5 (3) and (4)]{rao}, we may assume $g_1 = w_j$ for $j = 1,3$ and $g_2 = n w_3$ for some $n \in N_{P_R}^T(k)$.  This gives the lemma. 	\end{proof}
	
	To finally prove Proposition \ref{prop:splittingC}, we will evaluate $c_{Rao}(w_j, n w_3)$ for $j = 1,3$.  We break the proof into two more lemmas.
	\begin{lemma}\label{lem:cRaoRedconj} Suppose $n \in N_{P_R}^T(k)$.  Then there exists $m \in M_{P_R}^T(k)$ so that $m n m^{-1}$ is of the form $\exp(a E_{12} + b \delta_3 \otimes e_{11})$, $a,b \in k$.
	\end{lemma} 
	\begin{proof}The conjugation action of $M_{P_R}^T$ on $N_{P_R}^T$ is isogenous to that of an orthogonal group acting on the underlying quadratic space that defines it. The lemma follows.  More concretely, suppose $T = \delta_3 \otimes T'$.  Let $S \in H_2(C)$ satisfy $(S,T')_J = 0$, and let $S' \in H_2(C)$ satisfy $(e_{11} \times T',S') = 0$.  Then $\exp(v_2 \otimes S)$ and $\exp(\delta_2 \otimes S')$ are in $M_{P_R}^T$.  One can use the action of these elements to prove the lemma.\end{proof}
	
	\begin{lemma}\label{lem:cRaoRed2} Suppose $c_{Rao}(w_1, n w_3) = 1$ for all $n$ of the form $n = \exp(a E_{12})$, and $c_{Rao}(w_3, n w_3) = 1$ for all $n$ of the form $\exp(a E_{12} + b \delta_{3} \otimes e_{11})$.  Then $c_{Rao}(g_1, g_2) = 1$ for all $g_1, g_2 \in M_R^T(k)$.
	\end{lemma}
	\begin{proof} By Lemma \ref{lem:cRaoRed1}, we must only consider $c_{Rao}(w_j, n w_2)$ for $j = 1,3$ and $n \in N_{P_R}^T(k)$.  Suppose first $j=3$.  Applying Lemma \ref{lem:cRaoRedconj} and \cite[Corollary 5.5 (3) and (4)]{rao}, we get the desired reduction.  Now suppose that $j=1$.  If $U \in H_2(C)$, one computes $Ad(w_1)( \delta_3 \otimes U) = \delta_3 \otimes U$ and $Ad(w_1)( \delta_3 \otimes e_{11}) = \delta_3 \otimes e_{11}$.  Suppose $n \in N_{P_R}^T$.  We can write $n= n_1 n_2$ with $n_1 = \exp(v_1 \otimes U + b \delta_3 \otimes e_{11})$ and $n_2 = \exp(a E_{12})$.  Then using \cite[Corollary 5.5 (3) and (4)]{rao} again, $c_{Rao}(w_1, n w_3) = c_{Rao}(w_1 n_1, n_2 w_3) = c_{Rao}(w_1, n_2 w_3)$ by conjugating the $w_1$ past the $n_1$.  This proves the lemma.
	\end{proof}
	
	To finally prove Proposition \ref{prop:splittingC}, we need to calculate the integer $\ell$ and the Leray invariant that arises in $c_{Rao}(w_j , n w_3)$ for the $n$ of the special form that appear in Lemma \ref{lem:cRaoRed2}.
	\begin{proof}[Proof of Proposition \ref{prop:splittingC}] Let us first compute $[ a E_{12} + b \delta_3 \otimes e_{11}, (u,u')_X]$.  One gets
		\begin{align*}
			[ a E_{12} + b \delta_3 \otimes e_{11}, (u,u')_X] &=a v_1 \otimes V(0,u_2',u_3') - b \delta_3 \otimes \Phi_{E, V(0,u_2,u_3)}(e_{11})\\
			&= a v_1 \otimes V(0,u_2',u_3') + b \delta_3 \otimes V(0,u_2,u_3).
		\end{align*}
		Thus 
		\begin{equation}\label{eqn:nXaction}[ a E_{12} + b \delta_3 \otimes e_{11}, (u,u')_X] = (au',bu)_Y.\end{equation}.
		Now, if $g \in \Sp(V_8)$, then $g \in \Omega_j$ if and only if $\dim pr_{X}(Yg) = j$, where $pr_X: V_8 \rightarrow X$ is the projection with kernel $Y$.  
		
		Suppose $g_1 = w_3$ and $g_2 = n w_3$ with $n = \exp(a E_{12} + b \delta_3)$.  Using this characterization of $\Omega_j$ and equation \eqref{eqn:nXaction}, one sees $g_1 g_2 \in \Omega_j$ where $j = n = 4\dim(C)$ if $ab \neq 0$, $j = 2 \dim(C)$ if exactly one of $a,b$ is not $0$, and $j = 0$ if both $a,b=0$.  
		
		The Leray invariant $\rho = q(Y, Y g_2^{-1}, Yg_1) = q(Y, X n^{-1}, X) = - q(Y,X,Xn^{-1})$.  If $ab \neq 0$, then $\dim(\rho) = n = 4 \dim(C)$, and $\rho$ is the quadratic form on $X$ given by 
		\begin{align*} \rho((u,u')_X) &= -\frac{1}{2} \langle (u,u')_X, (u,u')_X \cdot n^{-1} \rangle_{V_8,T} = -\frac{1}{2} \langle (u,u')_X, (au',bu)_Y \rangle_{V_8,T} \\ &= a (T' \times e_{11}, (u')^\#) - b(T', u^\#).\end{align*}
		In particular, $\ell = 0$ in this case.
		
		To simplify further, suppose $\rho_1$ is a quadratic form and $\alpha \in k^\times$.  Then one quickly verifies 
		\[h(\alpha \rho_1) = (\alpha,\alpha)^{\dim(\rho_1)(\dim(\rho_1)-1)/2} (\alpha, \det(\rho_1))^{\dim(\rho_1)-1} h(\rho_1).\]
		If $\rho_2$ is another quadratic form, then $h(\rho_1 \oplus \rho_2) = h(\rho_1) h(\rho_2) (\det(\rho_1),\det(\rho_2))$.  Combining these identities, if $\beta \in k^\times$ and $d_1 = \dim(\rho_1)$, then
		\begin{equation}\label{eqn:hasseSum} h(\alpha \rho_1 \oplus \beta \rho_1) = ((\alpha,\alpha)(\beta,\beta))^{d_1(d_1-1)/2}  (\alpha \beta, \det(\rho_1))^{d_1-1} (\alpha^{d_1} \det(\rho_1), \beta^{d_1} \det(\rho_1)).\end{equation}
		
		Now, by the change of variables used in the proof of Lemma \ref{lem:wjsqaure}, one can assume $T'$ is diagonal in $H_2(C)$.  Then $\rho$ is of the form $\alpha \rho_1 \oplus \beta \rho_1$, with $\rho_1 = \mu_2 n_{C} \oplus \mu_3 n_C$ for some $\mu_2, \mu_3 \in k^\times$.  In particular, because $2|\dim(C)$, $\det(\rho_1)$ is a square and $4|\dim(\rho_1)$.  By equation \eqref{eqn:hasseSum}, $h(\rho)=1$.  We have thus proved that $c_{Rao}(w_3, \exp(a E_{12} + b \delta_3 \otimes e_{11}) w_3) = 1$ when $ab \neq 0$.
		
		We next handle the case when $a\neq 0$ but $b=0$.  In this case, $j_1 = n = 4\dim(C)$, $j_2 = n = 4\dim(C)$, $j = 2 \dim(C)$, and $\dim(\rho) = 2\dim(C)$.  Thus $4|\ell$.  The Leray invariant $\rho$ is of the form  $\mu_2 n_{C} \oplus \mu_3 n_C$ for some $\mu_2, \mu_3 \in k^\times$.  Because $C$ is a quaternionic algebra or octonion algebra, $\det(n_C)$ is a square.  By equation \eqref{eqn:hasseSum}, $h(\rho)=1$.
		
		The case where $a=0$ but $b\neq 0$ is nearly identical.  Finally then, we comute $c_{Rao}(w_1, \exp(a E_{12}) w_3)$ for $a \in k$.   If $a =0$, then $w_1 w_3 \in \Omega_j$ with $j = 2\dim(C)$.  The Leray invariant is trivial in this case, so $h(\rho) = 1$, $\dim(\rho) = 0$, and $2\ell = 4\dim(C)$.  Thus $c_{Rao}(w_1, w_3) = 1$.
		
		If $a \neq 0$, then $w_1 \exp(a E_{12}) w_3 \in \Omega_j$ with $j = n = 4\dim(C)$. In this case, $\dim(\rho)=2\dim(C)$ and $\rho$ is again of the form $\mu_2 n_{C} \oplus \mu_3 n_C$ for some $\mu_2, \mu_3 \in k^\times$.  Thus $\ell=0$, $h(\rho)=1$ and $c_{Rao}(w_1, \exp(aE_{12})w_3) = 1$.  This completes the proof of the proposition.
	\end{proof}
	
	\subsection{Splittings: commutative case}
	In this subsection, we prove that $M'(k)$ splits into $\widetilde{\Sp(V_8)}(k)$ when $\dim(C) = 1$ or $2$, i.e., when $C$ is commutative.  We will reduce to the result of Kudla \cite{kudlaSplitting} that the unitary group $U(V)$ has a splitting into the metaplectic cover of a symplectic group associated with this unitary group.  We begin with the following proposition.
	
	Let $J_2'': C^2 \rightarrow C^2$ given by $(x_2, x_3) \mapsto (x_3, -x_2)$.  Let $T = \delta_3 \otimes T'$ and set $T'' = e_{11} \times T'$. Define a map $J_T: V_8 \rightarrow V_8$ as follows.
	\begin{enumerate}
		\item $J_T((u,0)_X) = (J_2''(T' \times u),0)_X$
		\item $J_T((0,u')_X) = (0,-J_2''(T'' \times u'))_X$
		\item $J_T((v,0)_Y) = (-J_2''(T'' \times v),0)_X$
		\item $J_T((0,v')_X) = (0,J_2''(T' \times v'))_X$
	\end{enumerate}
	If $T' = \mm{c_2}{r_1}{r_1^*}{c_3}$, set $n_{H_2(C)}(T') = c_3 c_3 - n_C(r_1)$.
	
	\begin{proposition}\label{prop:JT} The map $J_T: V_8 \rightarrow V_8$ satisfies $J_T^2 = - n_{H_2(C)}(T') 1_{V_8}$.  Moreover, $J_T$ commutes with the action of $M'$ on $V_8$.
	\end{proposition}
	\begin{proof} The fact that $J_T^2 = - n_{H_2(C)}(T') 1_{V_8}$ is checked directly in coordinates.
		
		We now argue regarding the commutativity.  The group $M'$ contains the elements $v_2 \otimes U$ and $\delta_2 \otimes U'$, where $U, U' \in H_2(C)$ satisfy $(U,T') = 0$ and $(T'',U') = 0$.  It is a tedious but straightforward computation to check that $J_T$ commutes with these Lie algebra elements.  (One uses the fact that $(T',U) = 0$ and $(T'', U') = 0$.)  Next, one checks that $J_T$ commutes with the elements in $N_{P_R}^T$.  To do this, by the observation of Lemma \ref{lem:cRaoRedconj}, one only must check this commutativity on elements of the form $a E_{12} + b \delta_3 \otimes e_{11}$.  These checks are immediate.  Next one makes the same computation on the opposite nilradical.  Finally, $Lie(M')$ is generated by $Lie(N_{P_R}^T)$ and its opposite (this is true generally in orthogonal groups), so the commutativity holds on all of $Lie(M')$.  The proposition follows.
	\end{proof}
	
	\begin{proposition}\label{prop:SUsplitting} Suppose $\dim(C) = 1$ or $2$ and $J = H_3(C)$.  Let $k = \Q_v$ be a completion of $\Q$.  Then there is a splitting $s_v:M'(k) \rightarrow \widetilde{\Sp(V_8)}(k)$.
	\end{proposition}
	\begin{proof} By Proposition \ref{prop:JT}, the group $M'$ is a subgroup of a special unitary group $\SU(V_8,T)$ that sits in $\Sp(V_8)$.  It is proved by Kudla \cite{kudlaSplitting} that the unitary group splits into the metaplectic group, which is a central extension of $\Sp(V_8)$ by $\C^\times$.  As the special unitary group is its own derived group \cite[Theorem 7.1, Proposition 7.6, Theorem 7.6]{platonovRapinchuk}, and the derived group of the metaplectic group is the double cover of $\Sp(V_8)(k)$, there is a splitting $\SU(V_8,T)(k) \rightarrow \widetilde{\Sp(V_8)}(k)$.  Restricting this splitting to $M'(k)$, we obtain $s_v$.  We remark that the splitting on $\SU(V_8,T)$ is unique because every homomorphism $\SU(V_8,T)(k) \rightarrow \mu_2$ is trivial, because $\SU(V_8,T)(k)$ is its own derived group.
	\end{proof}

	\begin{remark} Suppose again that $\dim(C) =1$ or $2$. As the group of rational points $\Sp(V_8)(\Q)$ splits into $\widetilde{\Sp(V_8)}(\A)$, we have a splitting $s_\Q: \SU(V_8,T)(\Q) \rightarrow \widetilde{\Sp(V_8)}(\A)$.  We have another splitting $\SU(V_8,T)(\Q) \rightarrow \widetilde{\Sp(V_8)}(\A)$, by pieceing together the local splittings of the proof of Proposition \ref{prop:SUsplitting}.  By \cite[Theorem 9.1]{platonovRapinchuk}, $\SU(V_8,T)(\Q)$ is its own derived group, so there are no nontrivial homomorphisms $\SU(V_8,T)(\Q) \rightarrow \mu_2$.  Consequently, these two splittings agree.  As the splittings on $M'$ are restricted from those of $\SU(V_8,T)$, the two global splittings on $M'(\Q)$ agree.
	\end{remark}
	
	\section{The Fourier-Jacobi expansion for $R$ II: Computation}\label{sec:FJcomp}
	Let $M'(\A) \rightarrow \widetilde{\Sp(V_8)}(\A)$ denote the splitting constructed in section \ref{sec:splittings}.  Using this map, we can define a Fourier-Jacobi coefficient of a cusp form $\varphi$ with respect to a theta function $\Theta_\phi$.  Namely, if $\varphi$ is a cuspidal automorphic form on $G_J(\A)$, we define the Fourier-Jacobi coefficient $\mathrm{FJ}^R_{T,\phi}(\varphi)$ as
	\[\mathrm{FJ}^R_{T,\phi}(\varphi)(g) = \int_{N_R(\Q)\backslash N_R(\A)}{\varphi(h g) \Theta_\phi(hg)\,dh}.\]
	Here $h \in N_R$ and $g \in M'(\A)$.  
	
	For a cusp form $\varphi$, we set $\varphi_{\chi_T}(g) =\int_{[V_7]}{\chi_T(n)\varphi(\exp(n)g)\,dn}$.  Let $\widetilde{\mathcal{R}}$ denote the set $w \in W_J(\Q)$ with $\xi_{w}|_{V_7^{[1]}}= \chi_T^{-1}$ and $\chi_w$ is trivial on $V_{8}^{[1]}$.  For $w \in \widetilde{\mathcal{R}}$, set
	\[\mathrm{FJ}^{R}_{\phi,w}(g) = \int_{(M_P \cap N_R)(\A)}{\varphi_w(xg) (\omega_\chi(xg)\phi)(0)\,dx}.\]
	Note that $M_P \cap N_R = \exp(V_8^{[0]} + V_7^{[0]})$.  Let $\mathcal{R}$ denote a set of representatives for $\widetilde{\mathcal{R}}/\exp(v_2 \otimes e_{11}(\Q))$; the group $\exp(\Q (v_2 \otimes e_{11}))$ acts freely on $\widetilde{\mathcal{R}}$. 
	
	We have the following proposition.  Let $P_R = P \cap M_R = M_{P_R} N_{P_R}$ and recall $M_{P_R}^T = M_P \cap M_R^T$, $N_{P_R}^T = N_{P} \cap M_R^T$.  The group $N_{P_R}^T$ is codimension $1$ in $N_{P_R}$. We write $X = V_8^{[0]}$, $Y = V_8^{[1]}$.
	\begin{proposition}\label{prop:FJT1} Assume $T$ is normal.  Then one has $\mathrm{FJ}^{R}_{T,\phi}(g) = \sum_{w \in \mathcal{R}}{\mathrm{FJ}^{R}_{\phi,w}(g)}$, and this is the Fourier expansion of the automorphic form $\mathrm{FJ}^{R}_{T,\phi}$ on $M'$ along the unipotent group $N_{P_R}^T$.
	\end{proposition}
	\begin{proof} Unfolding the sum defining $\Theta_\phi$, we obtain
		\[\mathrm{FJ}^{R}_{\phi,\chi_T}(g) = \int_{ Y(\Q) V_7(\A)\backslash N_R(\A)}{ \varphi_{\chi}(hg) (\omega_\chi(hg)\phi)(0)\,dh}.\]
		We integrate over $Y(\Q)\backslash Y(\A)$ to obtain that 
		\[\mathrm{FJ}^{R}_{\phi,\chi_T}(g) = \int_{Y(\A)V_7(\A)\backslash N_R(\A)}{\varphi_{\chi_T,V_8^{[1]}}(hg) (\omega_\chi(hg)\phi)(0)\,dh}\]
		where $\varphi_{\chi_T,V_8^{[1]}}(g) = \int_{[V_8^{[1]}]}{\varphi_{\chi_T}(\exp(n)g)\,dn}$.
		
		Let $\chi'$ be the character of $N_P \cap N_R = V_8^{[1]} + V_7^{[1]} + V_7^{[2]}$ that is the restriction of $\chi_T^{-1}$ on the $V_7$ parts and is trivial on $V_8^{[1]}$.  Then 
		\[\varphi_{\chi_T,V_8^{[1]}}(x) = \int_{\Q\backslash \A}{ \varphi_{\chi'}(\exp(s v_2 \otimes e_{11}) x)\,ds}.\]
		Moreover, 
		\[\varphi_{\chi'}(x) = \sum_{w \in W_J(\Q): \chi_w|_{N_P \cap N_R} = \chi'}{\varphi_w(x)}.\]
		
		Observe moreover that if $w \in \widetilde{\mathcal{R}}$, $\mu \in \Q$, $\mu \neq 0$, then $ \exp(\mu v_2 \otimes e_{11}) w \in \widetilde{\mathcal{R}}$, and is not equal to $w$.   We obtain
		\[\varphi_{\chi'}(x) = \sum_{\mu \in \Q, w \in \mathcal{R}}{\varphi_{w}(\exp(\mu v_2 \otimes e_{11}) x)}.\]
		
		Thus
		\[\varphi_{\chi,V_8^{[1]}}(x) = \sum_{w \in \mathcal{R}} \int_{\A}{\varphi_{w}(\exp(s v_2 \otimes e_{11}) x)\,dx}.\]
		From the above we obtain
		\[\mathrm{FJ}^{R}_{T,\phi}(g) = \sum_{w \in \mathcal{R}} \int_{(N_R \cap M_{P})(\A)}{\varphi_w(x g) (\omega_{\chi}(xg)\phi)(0)\,dx}.\]
		
		It remains to check that this is the Fourier expansion of $\mathrm{FJ}^{R}_{T,\phi}(g)$.  Suppose $x \in N_R \cap M_P$ and $n \in N_P \cap M_R^T$.  Then $xn = n x_1 x$ for some $x_1 \in N_R \cap N_P$.  Because $n$ acts as the identity on $V_{8}^{[1]}$, $\omega_{\chi_T}(n)\phi'(0) = \phi'(0)$ for any $\phi' \in S(X(\A))$.  One obtains
		\[\mathrm{FJ}^R_{\phi,w}(ng) = \xi_w(n) \mathrm{FJ}^R_{\phi,w}(g).\]
		The proposition now follows.
	\end{proof}
	
	\subsection{The Gaussian} In this subsection, we analyze a certain Gaussian key to our computations.  Recall $T = \delta_3 \otimes T'$.  We assume $T'$ is positive-definite, i.e., $\tr(T') > 0$ and $n_{H_2(C)}(T') > 0$.  Let $\epsilon \in \Sp(V_8)$ be the map defined as $\epsilon((u,u')_X) = (-u,u')_X$ and $\epsilon((v,v')_Y) = (v,-v')_Y$.  Set $J_2' \in \Sp(V_8)$ as $J_2' = \epsilon \circ w_3$.  Thus $J_2'((u,u')_X) = (u',-u)_Y$ and $J_2'((v,v')_Y) = (v',-v)_X$.  Observe that for $x = (u,u')_X \in X(\R)$, 
	\[ \langle x, J_2'(x) \rangle_{V_8,T} = -(T',u \times u) - (T' \times e_{11}, u' \times u'),\]
	which is a positive-definite quadratic form on $X(\R)$, using that $T'$ is positive-definite. For a positive constant $C$ to be specified below and $x \in X(\R)$, we define $\phi_{0}(x) = e^{-C \langle x, J_2'(x) \rangle}$.
	
	The action of $M_R^T$ on $V_7$ preserving the quadratic form induces a homomorphism $Lie(M_R^T) \rightarrow \wedge^2 V_7$.  Set $e = \delta_3 \otimes e_{11} - E_{12}$ and $f = -v_3 \otimes e_{11} - E_{21}$.  Then $e-f \mapsto (b_1 + b_{-1}) \wedge (b_2 + b_{-2})$.  We will compute the action of $e-f$ on the Gaussian $\phi_0$ by the differential of the Weil representation $d \omega$.
	
	\begin{lemma}\label{lem:AdJ2e} One has $Ad(J_2')(e) = - f$.
	\end{lemma}
	\begin{proof}  Both $Ad(J_2')(e)$ and $f$ are trivial on $X \subseteq V_8$, and on $Y \subseteq V_8$ one has
		\begin{align*}
			Ad(J_2')(e)((v,v')_Y) &= J_2' \circ ad(e) (-v',v)_X = J_2' \circ ([-E_{12}+ \delta_{3} \otimes e_{11}, -\Phi_{E,v'} + v_2 \otimes v]) \\ 
			&= J_2' \circ (-v_1 \otimes v - \delta_{3} \otimes v') = J_2'((-v,-v')_Y) \\
			&= (-v',v)_X;\\
			ad(-f)((v,v')_Y) &= [v_3 \otimes e_{11}+E_{21}, v_1 \otimes v + \delta_3 \otimes v'] \\
			&= -\Phi_{E,v'} + v_2 \otimes v = (-v',v)_X.
		\end{align*}
		This gives the lemma.
	\end{proof}
	
	One has 
	\[[e,(u,u')_X] = [\delta_{3} \otimes e_{11}-E_{12}, \Phi_{E,u} + v_2 \otimes u'] = \delta_3 \otimes u - v_1 \otimes u' = (-u',u)_Y.\]
	Thus
	\begin{align*}\frac{1}{2}\langle [e,(u,u')_X], (u,u')_X \rangle &= \frac{1}{2} \langle (u,u')_X, (u',-u)_Y \rangle =  \frac{1}{2} \rangle (u,u')_X, J_2'(u,u')_X \rangle \\ &= -(T',u^\#) - (T' \times e_{11}, (u')^\#).\end{align*}
	For ease of notation, define the quadratic form $q^0_T: X(\R) \rightarrow \R$ as $q_T^0(x) = \frac{1}{2}\langle x, J_2' x \rangle$.  As mentioned above, this quadratic form is positive-definite.	If $x = (u,u') \in X(\R)$ and $\phi \in S(X(\R))$ then $d\omega(e) \phi(x) = -2\pi i q_T^0(x) \phi(x)$.  (Recall that our Weil representation in this case was defined using the element $-T \in V_7$, hence the minus sign here.)
	
	We will use this computation and Lemma \ref{lem:AdJ2e} to compute $d\omega(-f)$.  Define $\widetilde{D}: S(X(\R)) \rightarrow S(X(\R)) \otimes X^\vee$ as
	\[\widetilde{D}\phi(x) = \sum_{\alpha} X_\alpha \phi \otimes X_\alpha^\vee\]
	where $\{X_{\alpha}\}_\alpha$ is a basis of the vector space $X(\R)$. 
	
	Identify $Y \rightarrow X^\vee$ via 
	\[y \mapsto \{x \mapsto \langle x, y \rangle_{V_8}\}.\]
	Now let $D_T  = \mathrm{pair} \circ J_2' \circ \widetilde{D}^2$.   That is,
	\[D_T \phi = \sum_{\alpha,\beta} X_\beta X_\alpha \phi \otimes \langle (J_2')^{-1} X_\alpha^\vee, X_\beta^\vee \rangle\]
	where we identify $X^\vee \simeq Y$ via the symplectic pairing as above and then apply $(J_2')^{-1}: Y \rightarrow X$.
	
	\begin{lemma} One has $d\omega(-f)\phi = - \frac{1}{4\pi i } D_T\phi$ for all $\phi \in S(X(\R))$.
	\end{lemma}
	\begin{proof} Suppose $y \in Y(\R)$.  We first compute $D_T e^{-2\pi i \langle x, y \rangle_{V_8}}$.  
		We have 
		\[\widetilde{D} e^{-2\pi i \langle x,y \rangle_{V_8}} = (-2\pi i) e^{-2\pi i \langle x,y \rangle_{V_8}}  \otimes y, \text{ and } \widetilde{D}^2 e^{-2\pi i \langle x,y \rangle_{V_8}} = (-2\pi i)^2 e^{-2\pi i \langle x,y \rangle_{V_8}}  \otimes y \otimes y.\]
		Thus
		\[D_T e^{-2\pi i \langle x, y \rangle_{V_8}} = (-2\pi i)^2 \langle (J_2')^{-1}y,y \rangle_{V_8} e^{-2\pi i \langle x, y \rangle_{V_8}},\]
		so 
		\[D_T e^{-2\pi i \langle x, J_2' x'\rangle_{V_8}} = (-2\pi i)^2 \langle x', J_2' x' \rangle e^{-2\pi i \langle x, J_2'x' \rangle_{V_8}}.\]
		
		Now, observe that if $g \in \widetilde{Sp(V_8)}(\R)$ is of the form $g = m \tau$, with $m$ in the Levi of the Siegel parabolic, then 
		\[\omega(g)\phi(x) = A_g \int_{X(\R)}{e^{-2\pi i \langle x, g x' \rangle_{V_8,T}} \phi(x')\,dx'}\]
		for some $A_g \in \C^\times$.  Suppose that $\phi = \omega(J_2') \phi'$.  By Lemma \ref{lem:AdJ2e}
		\[d\omega(-f)\phi(x) = A_{J_2'} \int_{X(\R)}{e^{-2\pi i \langle x, J_2' x' \rangle}(-2\pi i) q_T^0(x')\phi'(x')\,dx'}.
		\]
		
		On the other hand, $D_T \phi(x)$ is computed as
		\begin{align*}
			D_T\phi(x) &= D_T\left( A_{J_2'} \int_{X(\R)}{e^{-2\pi i \langle x, J_2'x' \rangle} \phi'(x')\,dx'}\right)\\
			&= A_{J_2'} \int_{X(\R)}{e^{-2\pi i \langle x, J_2'x' \rangle} (-2\pi i)^2 \langle x',J_2' x' \rangle \phi'(x')\,dx'}.
		\end{align*}
		The lemma follows.
	\end{proof}
	
	We are now ready to compute $d\omega(e-f) \phi_0$.
	\begin{lemma}\label{lem:KGaussianR1} If $C = \pi$ so that $\phi_0(x) = e^{-\pi\langle x, J_2' x \rangle}$, then $d\omega(e-f) \phi_0  = - \left(\frac{i n}{2}\right)\phi_0$ where $n = \dim(X) = 4 \dim(C)$.
	\end{lemma}
	\begin{proof} We have $d\omega(e) \phi_0(x) = -\pi i \langle x, J_2' x \rangle \phi_0(x)$.  We now compute $D_T \phi_0$.  Let $e_1, \ldots, e_n$ be a basis of $X(\R)$. 
		
		One has $\widetilde{D}\phi_0(x) = -2 C \phi_0(x) \otimes J_2' x$.  Thus 
		\[ \widetilde{D}^2 \phi_0(x) = (-2C)^2\phi_0(x) \otimes J_2'x \otimes J_2'x + (-2C)\phi_0(x) \otimes \sum_{i,j} J_2' e_i \otimes e_j^{\vee}.\]
		Thus
		\[ D_T \phi_0(x) = ((-2C)^2 \langle x, J_2' x \rangle + (-2C)n)\phi_0(x).\]
		Thus if $C = \pi$ then
		\[ d\omega(-f)\phi_0 = -\frac{1}{4\pi i} D_T \phi_0 = (i \pi \langle x, J_2' x \rangle - \frac{i}{2} n) \phi_0.\]
	\end{proof}
	
	Because $\phi_0$ is a Gaussian, it is easy to determine how other compact Lie algebra elements act. 
	
	\begin{lemma}\label{lem:KGaussianR2} Let $\widetilde{U}(V_8;J_2')$ denote the subgroup of $\widetilde{\Sp(V_8)}(\R)$ that commutes with (a preimage) of $J_2'$, and let $\SU(V_8;J_2')$ denote its derived group.  If $k \in \SU(V_8;J_2')$, then $\omega(k) \phi_0 = \phi_0$.  In particular, suppose $k \in M'(\R)$ is in the derived group of the commutator of $\SO(2) \approx \{ \exp(t (e-f))\} \subseteq M'(\R)$, then $\omega(k)\phi_0 = \phi_0$.
	\end{lemma}
	\begin{proof} The statement about the action of $\SU(V_8;J_2')$ on the Gaussian in the Schrodinger model of the Weil representation is well-known.  The second statement follows from the fact that $\exp(\frac{\pi}{2} ad(e-f)) = -J_2'$.  To see this equality, recall that we have proved the following identities:
		\begin{itemize}
			\item $ad(e)((u,u')_X) = (-u',u)_Y$
			\item $ad(e)((v,v')_Y)  = 0$
			\item $ad(-f)((u,u')_X) = 0$
			\item $ad(-f)((v,v')_Y) = (-v',v)_X$.
		\end{itemize}
		If $U$ is an endomorphism of some vector space with $U^2=-1$, then $\exp(t U ) = \cos(t) 1 + \sin(t) U$.  Thus $\exp(\frac{\pi}{2} ad(e-f)) = -J_2'$ as claimed.
	\end{proof}
	
	\subsection{The explicit integral}
	The purpose of this subsection is to compute the $\C$-valued integral
	\[ I_{T,\infty}(w,g;\phi)=\int_{(M_P \cap M_R^T)(\R)}{\langle W_{\ell,w}(xg), (x + y)^{2\ell} \rangle_K (\omega_{\chi_T}(xg)\phi)(0)\,dx}\]
	if $g \in M'(\R)$, $w \in \mathcal{R}$, and where $\phi = \phi_0$ is the Gaussian.  By Proposition \ref{prop:FJT1}, $I_{T,\infty}(ng,\phi) = \xi_w(n) I_{T,\infty}(g,\phi)$ if $n \in N_{P_R}^T(\R)$.  
	
	Some of our arguments also work on the vector-valued integral
	\[ J_{T,\infty}(w,g;\phi)=\int_{(M_P \cap M_R^T)(\R)}{ W_{\ell,w}(xg) (\omega_{\chi_T}(xg)\phi)(0)\,dx},\]
	and we will phrase some computations as pertaining to this integral.
	
	Let $K_{M'}$ denote the identity component of the maximal compact subgroup of $M'(\R)$ determined by our Cartan involution on $\g(J)$.  We next compute $I_{T,\infty}(gk;\phi_0)$ for $k \in K_{M'}$.  We begin the following lemma.   
	\begin{lemma}\label{lem:pr(e-f)} Let $\mathrm{pr}_{\su_2}: \g(J) \rightarrow \su(2)$ be the projection to the long root $\su_2$, and let $e_{\ell},h_{\ell}, f_{\ell}$ denote the $\sl_2$-triple in $\su_2 \otimes \C$ from \cite[section 5.1]{pollackQDS}.  Then $\mathrm{pr}_{\su(2)}(e-f) = i(e_{\ell}+f_{\ell})$.  Moreover, if $k \in K_J$ is in the derived group of the centralizer of $\SO(2) \approx \{ \exp(t (e-f))\} \subseteq M'(\R)$, then $k$ acts trivially on $\Vell$.
	\end{lemma}
	\begin{proof}
		Recall from \cite[section 5.1]{pollackQDS} that
		\begin{itemize}
			\item $e_{\ell} = \frac{1}{4}(ie+f) \otimes r_0(i)$
			\item $f_{\ell} = \frac{1}{4}(ie-f) \otimes r_0(-i)$
			\item $h_{\ell} = \frac{i}{2}\left( \mm{0}{1}{-1}{0} + n_L(-1_J) + n_{L}^\vee(1_J)\right)$.
		\end{itemize}
		Here $r_0(i) = (1,-i 1_J, -1_J,i ) \in W_J(\C)$.  Applying \cite[section 4.2.4]{pollackQDS}, 
		\[e-f = e \otimes (1,0,e_{11},0) + f \otimes (0,e_{11},0,1) \]
		where the notations $``e, f"$ are overloaded.  
		
		For any $X \in \g(J)$, 
		\[\mathrm{pr}_{\su(2)}(X) = B(X,f_{\ell}) e_{\ell} + \frac{1}{2} B(X,h_{\ell})h_{\ell} + B(X,e_{\ell}) f_{\ell}\]
		where the pairing $B$ on $\g(J)$ is from \cite[section 4.1.2]{pollackQDS}, with $\alpha = \frac{1}{2}$.
		Thus $\mathrm{pr}_{\su(2)}(e-f) = i(e_{\ell}+f_{\ell})$.
		
		For the second part of the lemma, simply observe that, because the $\su_2$ projection of $e-f$ is nontrivial, the derived group of the centralizer of $\SO(2)$ has trivial $\su_2$ projection.
	\end{proof}
	
	Let $j_{M'}: M'(\R) \rightarrow \C^\times$ denote the function
	\[j_{M'}(g) = (g (b_1 + i (b_2+b_{-2})+ b_{-1}),b_1)_{V_7} = (g (E_{13}+v_2 \otimes e_{11} + i(v_1 \otimes e_{11}-E_{23})),E_{13})_{V_7}.\]
	\begin{lemma}\label{lem:KM'action}The function $j_{M'}$, restricted to $K_{M'}$, is a character. If $k \in K_{M'}$, then $I_{T,\infty}(gk,\phi_0) = j_{M'}(k)^{\ell - \dim(C)} I_{T,\infty}(g,\phi_0)$.
	\end{lemma}
	\begin{proof} Set $\theta= \frac{1}{2}(e-f)$.  First observe that 
		\[ad(\theta)(b_1 + b_{-1} + i (b_2+b_{-2})) = i (b_1 + b_{-1} + i (b_2+b_{-2})).\]
		Thus $\exp( t \theta)(b_1 + b_{-1} + i (b_2+b_{-2})) = e^{it} (b_1 + b_{-1} + i (b_2+b_{-2}))$.  If $k \in K_{M'}$ is in the derived group of the centralizer of $\SO(2) \approx \{\exp(t \theta): t \in \R\}$ then $k$ fixes $(b_1 + b_{-1} + i (b_2+b_{-2}))$.  Thus $j_{M'}: K_{M'} \rightarrow \C^\times$ is the unique character whose differential takes $\theta$ to $i$.
		
		By Lemma \ref{lem:KGaussianR1} and Lemma \ref{lem:KGaussianR2}, $\omega(k) \phi_0 = j_{M'}(k)^{-\dim(C)} \phi_0$.  By Lemma \ref{lem:pr(e-f)},
		\[ \langle W_{\ell,w}(gk), (x + y)^{2\ell} \rangle_{K_J} = \langle W_{\ell,w}(g), k (x+ y)^{2\ell} \rangle_{K_J} = j_{M'}(k)^{ \ell} \langle W_{\ell,w}(g), (x + y)^{2\ell} \rangle_{K_J}.\]
		This concludes the proof.
	\end{proof}
	
	We next understand $J_{T,\infty}(w,g,\phi)$ if $g \in (M_P \cap M')(\R)$.
	\begin{lemma}\label{lem:LeviM'action} Suppose $g \in (M_P \cap M')(\R)$.  Then $J_{T,\infty}(w,g,\phi) = \nu(g)^{\ell} |\nu(g)|^{-\dim(C)} J_{T,\infty}(w \cdot g, 1,\phi)$.
	\end{lemma}
	\begin{proof} For a vector space $U$ on which $(M_P \cap M')(\R)$ acts, let $|g|_U$ denote the Jacobian of the left action of $g$ on $U$.  We have
		\begin{align*} J_{T,\infty}(w,g,\phi) &= \int_{(V_7^{[0]} + V_8^{[0]})(\R)}{ W_w(g (g^{-1} \cdot s) (g^{-1} \cdot x) ) |g|_X^{-1/2}\phi(g^{-1} \cdot x)\,ds\,dx} \\ &= \nu(g)^{\ell}|\nu(g)| |g|_{V_7^{[0]}} |g|_X^{1/2} \left( \int_{(V_7^{[0]} + V_8^{[0]})(\R)}{ W_{w \cdot g}( s x ) \phi(x)\,ds\,dx}\right).\end{align*}
		
		Now, if $g$ preserves the quadratic form on $V_7$, then $|g|_{V_7^{[0]}} = |\nu(g)|^{-1}$, by using $(E_{13}, v_{2}\otimes e_{11})_{V_7} =1$.  Likewise, $|g|_{X}^{1/2} = |g|_Y^{-1/2}$, and $|g|_Y$ can be computed in terms of the similitude.  Namely, one finds $|g|_Y^{-1/2} = |\nu(g)|^{-\dim(C)}$.  The lemma follows.
	\end{proof}
	
	It remains to compute $I_{T,\infty}(w,1,\phi_{0})$ as a function of $w$.  We will assume $w \in Lie(M_R)^{[1]} \oplus V_7^{[1]}$, as this suffices for our purposes. Here is the result.
	\begin{proposition}\label{prop:ITat1} Suppose 
		\[w = (a,b,c,d) =  2\pi (a', b',c',d') \in Lie(M_R)^{[1]} \oplus V_7^{[1]}\]
		is positive-definite and $\xi_w|_{V_7^{[1]}} = \xi_T^{-1}$, with $T$ normal.  Then there is a nonzero complex number $C' \in \C^\times$, possibly depending on $T$ but otherwise independent of $w$, so that 
		\[I_{T,\infty}(w,1,\phi_0) = C'  e^{-2\pi (-b'_{11}+d')}\]
		if $ -b_{11}' + d' > 0$ and $I_{T,\infty}(w,1,\phi_0) = 0$ if $-b_{11}' + d' < 0$.
	\end{proposition}
	\begin{proof}
		We compute the integral
		\[J_{T,\infty}(w,1,\phi_0) = \int_{X(\R)}\int_{\R}{W_{\ell;w}(\exp(s v_2 \otimes e_{11}) \exp(x))\phi_0(x)\,ds\,dx}\]
		in two steps, first doing the $s$ integral then doing the $x$ integral.
		
		To do the computation, we use the explicit formula \cite{pollackQDS} for $W_{\ell,w}(g)$: for $g \in M_P(\R)$, 
		\[W_{\ell,w}(g) = \nu(g)^{\ell}|\nu(g)| \sum_{-\ell \leq v \leq \ell}{\left(\frac{|\alpha_w(g)|}{\alpha_w(g)}\right)^{v} K_v(|\alpha_w(g)|) \frac{x^{\ell+v} y^{\ell-v}}{(\ell+v)! (\ell-v)!}}\]
		where $\alpha_w(g) = \langle w, g r_0(i)\rangle$.
		
		Set $w' = \exp(-v_2 \otimes e_{11}) w$.  If $w = (a,b,c,d)$, then $w' = (0, a e_{11}, b \times e_{11}, (c,e_{11})) \in V_7^{[1]}$.  Note that if $x \in X(\R)$, then $w' \cdot \exp(x) = w'$.  Thus $w'' = \langle w', x r_0(i) \rangle = \langle w', r_0(i) \rangle$ is independent of $x$.  In coordinates, $w'' = a-c_{11} -i(b,E)$.
		
		Let $z(x) = \langle w, x r_0(i) \rangle$. 	We have 
		\[\alpha_w(\exp(s v_2 \otimes e_{11}) \exp(x)) = \langle w, \exp(s v_2 \otimes e_{11}) \exp(x) r_0(i)\rangle = z(x) + s w''.\]
		Using the explicit formula for $W_{\ell,w}(g)$, we therefore must compute the integral
		\[ I_{v}(z,w'') = \int_{\R}{ \left(\frac{z+sw''}{|z+sw''|}\right)^v K_v(|z+sw''|)\,ds}.\]
		This integral is computed in Proposition \ref{prop:defInt2} in terms of the quantity $\delta(x) = \frac{Im((w'')^* z(x))}{|w''|}$.   
		
		To make the formula explict, we compute $z(x)$ in more detail.  Suppose $x = \Phi_{E,u} + v_2 \otimes v$. One has
		\begin{align*}[\Phi_{E,u} + v_2 \otimes v, r_0(i)] &= [\Phi_{E,u}+v_2 \otimes v, E_{12} - i v_1 \otimes 1_J - \delta_3 \otimes 1_J + i E_{23}] \\ &= (-i) v_1 \otimes \Phi_{E,u}(1_J) -\delta_3 \otimes \Phi_{E,u}(1_J) - v_1 \otimes v + i \delta_3 \otimes (v \times 1_J) \\
			&= (-i)v_1 \otimes u + \delta_3 \otimes u - v_1 \otimes v + i \delta_3 \otimes (-v) \\
			&= -i v_1 \otimes (u-iv) + \delta_3 \otimes (u-iv).
		\end{align*}
		Continuing,
		\begin{align*}
			[\Phi_{E,u} + v_2 \otimes v, [\Phi_{E,u}+v_2 \otimes v,r_0(i)]] &= [\Phi_{E,u} + v_2 \otimes v, -i v_1 \otimes (u-iv) + \delta_3 \otimes (u-iv)] \\ 
			&= -iv_1 \otimes \Phi_{E,u}(u-iv) + \delta_3 \otimes \Phi_{E,u}(u-iv) \\&\,\,\,\,+ i \delta_3 \otimes (v \times (u-iv)) - (u-iv,v) E_{23} \\
			&=  -i (u,u-iv)v_1 \otimes e_{11} - \delta_3 \otimes (u \times (u-iv) + (u,u-iv)E) \\&\,\,\,\,+ i \delta_3 \otimes (v \times (u-iv)) - (u-iv,v) E_{23}.
		\end{align*}
		
		Thus
		\[
		z(x) = \langle w, r_0(i) \rangle - \frac{1}{2} a (u-iv,v) + (b,(u-iv)^\#) +\frac{1}{2}(u,u-iv)(b,E) - \frac{i}{2} (u,u-iv) c_{11}.
		\]
		
		Assume $a=0, c_{11} =0, b=b_{11} e_{11} - e_{11} \times (2\pi T')$, as we can because $T$ is assumed normal and $\xi_w(v) = \psi((T,v)_{V_7})$ for $v \in V_7^{[1]}$.  Then
		\begin{align*} \delta(x) &=\frac{Im((w'')^* \langle w, r_0(i) \rangle)}{|w''|} + \frac{1}{|w''|} (b,E)(-(b \times e_{11},u^\#) - (b,v^\#)) \\
			&= - Re( \langle w, r_0(i) \rangle) - (2\pi T',u^\#) - (2\pi T' \times e_{11},v^\#)\\
			&= 2\pi((T',E) - b'_{11}+d' - (T',u^\#) - (T' \times e_{11},v^\#)).
		\end{align*}
		Because $w > 0$, $b_{11}' < 0$ and $d' > 0$.  Thus $\delta(x) > 0$ for all $x$.  Thus
		\begin{align} \nonumber I_v(z(x),w'') &= \frac{1}{2} (-1)^v (T',E)^{-1} e^{-\delta(x)} \\ \label{eqn:Ivphi0}
			&=\frac{1}{2} (-1)^v (T',E)^{-1} e^{-2\pi ((T',E) - b'_{11}+d')} \phi_0(x)
		\end{align}
		where we are using that $\phi_0(x) = e^{2\pi ((T',u^\#)+(T'\times e_{11},v^\#))}$.
		
		The above quantity is equal to its complex conjugate, and 
		\[\left\langle \sum_{v} (-1)^v \frac{x^{\ell+v} y^{\ell-v}}{(\ell+v)! (\ell-v)!}, (x+y)^{2\ell} \right\rangle_{K_J} \neq 0,\]
		so
		\[I_{T,\infty}(w,1,\phi_0) = C'  e^{-2\pi (-b'_{11}+d')} \int_{X(\R)}{\phi_0(x)^2\,dx}.\]
		This proves the proposition.
	\end{proof}

	The following corollary will be used in section \ref{sec:converseThm}.
	\begin{corollary}\label{cor:FJotherphi} Suppose $\phi \in S(X(\R))$ satisfies $\int_{X(\R)}{\overline{\phi_0(x)} \phi(x)\,dx} = 0$.  Then $J_{T,\infty}(w,g;\phi) = 0$ on $M'(\R)^0$.
	\end{corollary} 
	\begin{proof}
		Note that, because $\phi_0$ is an eigenvector for $K_{M'}^0$ and the inner product on $S(X(\R))$ is preserved by the Weil representation, if $k \in K_{M'}^0$ then $\int_{X(\R)}{\overline{\phi_0(x)} (\omega(k) \phi)(x)\,dx} = 0$.  Thus, by $N_{P_R}^T$ and $K_{M'}^0$ equivariance, it suffices to prove the statement of the corollary for $g \in (M_P \cap M')(\R)$.  By Lemma \ref{lem:LeviM'action}, then, it suffices to prove the corollary for $g =1$.  But this follows from the calculations of Proposition \ref{prop:ITat1}, namely, equation \eqref{eqn:Ivphi0}.
	\end{proof}
	
	\subsection{Holomorphic modular forms}\label{subsec:Rhmf}
	In this subsection, we briefly describe the symmetric space for $M'(\R)$ and holomorphic modular form on $M'$.  We then piece together the work above to obtain our main theorem regarding the Fourier-Jacobi coefficient along the $R$-parabolic.
	
	Let $V_5 \subseteq V_7$ be $V_5 = \mathrm{Span}(b_2, H_2(C), b_{-2})$ and let $V_5^T$ be the subspace of $V_5$ orthogonal to $T$.  For $Y \in V_5^T(\R)$, we write $Y > 0$ if $q_{V_7}(Y) > 0$ and $(b_2+b_{-2},Y)_{V_7} > 0$.  Let 
	\[\mathcal{H}_{T} = \{X+iY \in V_5^T \otimes \C: Y > 0\}.\]
	This is the symmetric space for the identity component of $M'(\R)$.
	
	To see the action, if $Z \in \mathcal{H}_T$, set $R(Z) = -q_{V_7}(Z) b_1 + Z + b_{-1}$.  If $g \in M'(\R)^0$, set $j_{M'}(g,Z) = (g R(Z), b_1)$.  Then $g \cdot R(Z) = j_{M'}(g,Z) R(gZ)$ for a unique element $gZ \in \mathcal{H}_T$.
	
	For $u_1 \in V_5$, and an integer $\ell_1$, let $\mathcal{W}_{\ell_1, u_1}: M'(\R)^0 \rightarrow \C$ be the function
	\[\mathcal{W}_{\ell_1,u_1}(g) = j_{M'}(g, i (b_2+b_{-2}))^{-\ell_1} e^{2\pi i (u_1, Z_g)}\]
	where $Z_g = g \cdot (i(b_2+b_{-2})) \in \mathcal{H}_T$.
	
	We can now piece together the work above to prove the following proposition. 
	\begin{proposition}\label{prop:FJR2} Suppose $w = 2\pi w' = 2\pi (a',b',c',d') \in Lie(M_R^{[1]}) \oplus V_7^{[1]}$ and satisfies $\xi_{w'}(v) = \psi((T,v)_{V_7})$ for $v \in V_7^{[1]}$.  Then there is a nonzero constant $C'$, possibly depending on $T'$ but otherwise independent of $w$, so that
		\[I_{T,\infty}(w,g,\phi_0) = C' \overline{W_{\ell_1,u_1}(g)}\]
		for $g \in M'(\R)^0$, with $\ell_1 = \ell-\dim(C)$ and $u_1 = -b_{11}' b_2 - c_{23}' + d' b_{-2}$.
	\end{proposition}
	\begin{proof}
		Set $u_2=\alpha E_{12} + v_1 \otimes \beta + \gamma \delta_3 \otimes e_{11}$, where $\beta \in H_2(C)$.  Let $v_2 = \alpha b_2 + e_{11} \times \beta - \gamma b_{-2} \in V_5$.  Let $c' = c_{11}' e_{11} + c_{23}'$, with $c_{23}' \in H_2(C)$.  Then 
		\begin{align*}\xi_{w'}(u_2) &= \psi_\infty(-b_{11}'\gamma + (c_{23}',\beta)-d'\alpha) 
			\\&= \psi_\infty( (b_{11}' b_2 + c_{23}' - d' b_{-2}, v_2)_{V_7}).
		\end{align*}
		On the other hand, if $Z \in \mathcal{H}_T$, one computes $\exp(u_2) Z = Z +v_2$.  This shows that $I_{T,\infty}(w,g,\phi_0)$ and $\overline{W_{\ell_1,u_1}(g)}$ have the same equivariance condition on the left for elements of the form $\exp(u_2)$.  They have the same equivariance condition on the right by $K_{M'}^0$ by Lemma \ref{lem:KM'action}.  By Lemma \ref{lem:LeviM'action} and Proposition \ref{prop:ITat1}, the two functions agree on the identity component of $(M_P \cap M')(\R)$.  The proposition follows.
	\end{proof}

	We now come to our main theorem on the Fourier-Jacobi coefficient along the $R$-parabolic.  Suppose $T = \delta_3 \otimes T'$ is normal, with $T'$ positive-definite.   Suppose $\varphi$ is a cuspidal quaternionic modular form on $G_J$ of weight $\ell$, with Fourier expansion
	\[\varphi_Z(g) = \sum_{w\in W_J(\Q), w > 0}{a_w(g_f) W_{\ell,2\pi w}(g_\infty)}.\]
	For $u_1 =  -b_{11}' b_2 - c_{23}' + d' b_{-2} \in V_5^T(\Q)$, let 
	\[w(u_1) = b_{11}' v_1 \otimes e_{11} - v_1 \otimes (e_{11} \times T') + \delta_3 c_{23}' + d' E_{23} \in W_J(\Q).\]
	If $\phi \in S(X(\A_f))$ is a Schwartz-Bruhat function and $r_f \in M'(\A_f)$ and $g_f \in G(\A_f)$, set
	\[A^R_{\varphi,u_1}(r_f;g_f;\phi) = \int_{X(\A_f)}\int_{\A_f}{a_{w(u_1)}(\exp(s v_2 \otimes e_{11})\exp(x)r_f g_f) (\omega_{\chi_T}(r_f))\phi(x)\,ds\,dx}.\]
	
	\begin{theorem} Let the notation be as above, and set $\ell_1 = \ell-\dim(C)$.  Fix $g_f \in G(\A_f)$.  There is an automorphic form on $M'$, whose restriction to $M'(\A_f) \times M'(\R)^0$ corresponds to a holomorphic modular form of weight $\ell_1$ and Fourier expansion
		\[\sum_{u_1 > 0}{ \overline{A^R_{\varphi,u_1}(r_f;g_f;\phi)} W_{\ell_1,u_1}(r_\infty)}.\]
	\end{theorem}
	\begin{proof} This follows immediately from Proposition \ref{prop:FJT1} and Proposition \ref{prop:FJR2}.
	\end{proof}
	
	\begin{remark}\label{rmk:M'connected} In fact, for the quaternionic exceptional groups of type $F_4$ and $E_n$, the group $M'(\R)$ is always connected.  
	\end{remark}

	\section{The identity theorem}
	In this section, we state and prove an ``identity theorem" for quaternionic functions.  The identity theorem asserts that if $F(g): G(\R)^0 \rightarrow \Vell$ is a quaternionic function, and $F$ vanishes on a large enough subset, then $F$ is identitcally $0$.  We will use the identity theorem to help establish the Converse Theorem in section \ref{sec:converseThm}.
	
	We begin with a definition.
	\begin{definition} Suppose $U$ is a real vector space, with a linear action of $\SU(2)$.  Let $u_1, \ldots, u_n$ be vectors in $U$.  We say $u_1, \ldots, u_n$ are \emph{quaternionically independent} if 
		\[\dim_{\R}\mathrm{Span}_{\R}(\SU(2) \cdot u_1, \ldots, \SU(2) \cdot u_n) = 4n.\]
		The action of $\SU(2)$ gives rise to an $\mathbb{H}$-module structure on $U$, where $\mathbb{H}$ denotes Hamilton's quaternions.  The condition on the $u_1, \ldots, u_n$ is equivalent to the $u_i$ being independent for the $\mathbb{H}$-module structure.  If $U_0 \subseteq U$ is a subspace, we say that $U_0$ is \emph{quaternionically tranverse} if $u_1, \ldots, u_n$ is quaternionically independent for one (equivalently, any) basis of $U_0$.
	\end{definition}
	
	From the definition, one can prove:
	\begin{lemma}\label{lem:quatIndInj} Suppose $G(\R)^0$ is a quaternionic Lie group, and suppose $u_1, \ldots, u_n \in \p^\vee$ are quaternionically independent.  The linear map $\Vell^{n} \rightarrow S^{2\ell-1}(V_2) \otimes W$ given by $(v_1, \ldots, v_n) \mapsto \mathrm{pr}(v_1 \otimes u_1 + \cdots + v_n \otimes u_n)$ is injective.
	\end{lemma}
	\begin{proof} We have $u_j = x \otimes w_{j1} + y \otimes w_{j2}$ in $\p^\vee \otimes C \simeq V_2 \otimes W$.  The $\C$-span of $\{\SU(2) \cdot u_j\}_j$ is contained in $\sum_{j}\mathrm{Span}\{x \otimes w_{j1}, y \otimes w_{j1}, x \otimes w_{j2}, y \otimes w_{j2}\}$.  Because the $u_j$ are quaternionically independent, this $\C$-span is a full $4n$-dimensional over the complex numbers.  Thus the set $\{w_{11}, w_{12}, \ldots, w_{n1}, w_{n2}\}$ is $\C$-linearly independent in $W$.
		
		Now, suppose $\mathrm{pr}(v_1 \otimes u_1 + \cdots + v_n \otimes u_n) = 0$.  By the independence of the $w$'s, $\mathrm{pr}(v_j \otimes x) = 0$ and $\mathrm{pr}(v_j \otimes y) = 0$ for every $j$.  But then $v_j = 0$ for each $j$, as desired.
	\end{proof}
	
	Here is the identity theorem.
	\begin{theorem}\label{thm:Identity} Suppose $F: G(\R)^0 \rightarrow \Vell$ is a smooth, quaternionic function.  Let $\mathcal{X} \subseteq G(\R)^0$ be a closed submanifold of an open neighborhood of $1 \in G(\R)^0$, satisfying $xk \in \mathcal{X}$ for all $x \in \mathcal{X}$ and $k \in K^0$.   Let $U_0 \subseteq \p^\vee$ be the annihilator of $T_{1}(\mathcal{X}) \subseteq \p$.  Assume the following two conditions:
		\begin{enumerate}
			\item $U_0 \subseteq \p^\vee$ is quaternionically transverse;
			\item $F(x) = 0$ for all $x \in \mathcal{X}$.
		\end{enumerate}
		Then $F$ is identically $0$.
	\end{theorem}
	\begin{proof} First, because $F$ is quaternionic, it is real analytic.  (The idea for the proof of this fact is from \cite{ganOnlineNotes}.) Indeed, the quaternionicity of $F$ implies $\sum_{i}{X_i^2 F} -\sum_{j}{X_j^2 F} = \lambda F$ for orthonormal bases $\{X_i\}_i$ of $\p$ and $\{X_j\}_j$ of $\k$, and a constant $\lambda$ depending on $G$ and $\ell$.  Thus $F$ is annihilated by an elliptic differential operator, so is real-analytic by the elliptic regularity theorem.
		
		Now, let $V \subseteq G(\R)^0$ be the set of $g \in G(\R)^0$ for which every partial derivative of $F$ evaluated at $g$ is equal to $0$.  The set $V$ is closed: Let $\{U_\alpha\}$ be an open cover of $G(\R)^0$ so that $U_\alpha$ is diffeomorphic to an open subset of $\R^N$ for every $\alpha$.  Then $V \cap U_{\alpha}$ is closed in $U_\alpha$ for every $\alpha$, so $V$ is closed.   Because $F$ is real-analytic, $V$ is open.  Thus, if $V$ is non-empty, then $V = G(\R)^0$ and $F \equiv 0$.
		
		To see that $V$ is non-empty, we use the quaternionicity of $F$ and the assumption of the theorem to prove $1 \in V$.  Let $\{X_\beta\}$ be a basis of $T_{1}(\mathcal{X}) \subseteq \p$ and $\{X_\gamma\}$ elements so that the concatenation of the $X_\beta$'s with the $X_\gamma$'s is a basis of $\p$.  Let $X_\beta^\vee$ and $X_\gamma^\vee$ be the elements of the basis of $\p^\vee$ dual to this basis of $\p$.  Observe that the $X_\gamma^\vee$ form a basis of $U_0$, so are quaternionically independent.
		
		From $D_{\ell}F = 0$, we obtain
		\[ \mathrm{pr}\left(\sum_{\gamma}{ X_\gamma F \otimes X_\gamma^\vee}\right) = - \mathrm{pr}\left(\sum_{\beta}{ X_\beta F \otimes X_\beta^\vee}\right).\]
		By Lemma \ref{lem:quatIndInj}, every $X_\gamma F(g)$ for arbitrary $g \in G(\R)^0$ can be expressed in terms of the $X_\beta F(g)$.  By the commutativity of partial derivatives, every $X_{s_1} \cdots X_{s_M} F(g)$ can be expressed in terms of the $X_\beta F$'s where each $s_k$ is either a $\beta$ or a $\gamma$.   Because $F$ restricted to $\mathcal{X}$ is identically $0$, it follows that every higher order derivated $X_{s_1} \cdots X_{s_M} F(g)$ is $0$ at $g=1$.  Thus $1 \in V$ and the theorem is proved.
	\end{proof}
	
	\section{The converse theorem}\label{sec:converseThm}
	In this section, we state and prove the converse theorem, which says that certain absolutely convergent infinite sums define a cuspidal modular form on the exceptional group $G_J$.  We defer some of the technical details of the proof of this theorem to the next section.
	
	Fix an integer $\ell \geq 1$.  Suppose given functions $a_w: G(\A_f) \rightarrow \C$, one for each $w \in W_J(\Q)$ with $w > 0$, that satisfy
	\[a_w(ng_f) = \xi_w(n) a_w(g_f)\]
	for all $n \in N_P(\A_f)$.   We assume moreover that there is an open compact subgroup $U \subseteq G(\A_f)$ for which $a_w$ is right invariant by $U$ for all $w \in W_J(\Q)$.
	
	The numbers $a_w(1)$ are supported on a lattice in $W_J(\Q)$.
	\begin{lemma} Given $g_f \in G(\A_f)$, there is a lattice $\Lambda \subseteq W_J(\Q)$ (depending on $g_f$) so that $a_w(g_f) \neq 0$ implies $w \in \Lambda$.
	\end{lemma}
	\begin{proof} It suffices to prove the lemma when $g_f =1$.  Suppose $u \in U \cap N_P(\A_f)$.  Then $a_w(1) = a_w(u) = \xi_w(u) a_w(1)$.  Hence if $a_w(1) \neq 0$ then $\xi_w(u) =1$.  The set of $w \in W_J(\Q)$ with $\xi_w(U \cap N_P(\A_f)) = 1$ is a lattice.
	\end{proof}
	
	We define a notion of what it means for the $a_w$ to grow slowly with $w$.  Let $||\cdot ||$ be the norm on $W_J(\R)$ given by $||(a,b,c,d)||^2 = a^2 + (b,b) + (c,c) + d^2$.
	\begin{definition} We say the $a_w$ \textbf{grow polynomially} with $w$ if there are positive constants $C_{g_f}, N_{g_f} > 0$ so that $|a_w(g_f)| \leq C_{g_f}||w||^{N_{g_f}}$ for all $w \in W_J(\Q)$ and all $g_f \in G(\A_f)$.
	\end{definition}
	
	The following proposition will be proved in section \ref{sec:absConv}.
	\begin{proposition}\label{prop:absConv}  Suppose the functions $a_w$ grow polynomially with $w$.  Then the infinite sum
		\begin{align}\label{eqn:PsiDef}\Psi(g) &= \sum_{w \in W_J(\Q), a(w) = 0}{a_w(g_f) W_{\ell;w}(g_\infty)} \\&\,\,\,\,\,\, \nonumber + \sum_{\gamma \in B(\Q)\backslash \SL_2(\Q)}\sum_{w \in W_J(\Q), a(w) \neq 0}{a_w(j_{E_{12}}(\gamma_f) g_f)W_{\ell;w}(j_{E_{12}}(\gamma_\infty) g_\infty)}\end{align}
		converges absolutely.  For each fixed $g_f \in G(\A_f)$, $\Psi(g_f g_\infty)$ is a function of moderate growth in $g_\infty$, i.e., $||\Psi(g_f g_\infty)|| \leq C ||g_\infty||^N$ for some $C, N > 0$.  Moreover, it is $\mathcal{Z}(\g)$-finite, $K_J$-equivariant, and satisfies the differential equation $D_{\ell}\Psi \equiv 0$.
	\end{proposition}
	
	We recall notation from subsection \ref{subsec:Rhmf}.  Let $T = \delta_3 \otimes T' \in V_5 \subseteq V_7$ be normal, with $T' \in H_2(C)$ positive-definite.  For $u_1 \in V_5^T$, recall the element $w(u_1) \in W_J(\Q)$.  	If $\phi \in S(X(\A_f))$ is a Schwartz-Bruhat function, $g_f \in G(\A_f)$, and $r_f \in M'(\A_f)$, recall the quantity
	\[A^R_{\varphi,u_1}(r_f;g_f;\phi) = \int_{X(\A_f)}\int_{\A_f}{a_{w(u_1)}(\exp(s v_2 \otimes e_{11})\exp(x)r_fg_f) (\omega_{\chi_T}(r_f))\phi(x)\,ds\,dx}.\]
	Likewise, recall from subsection \ref{subsec:holMFQ} the quantity
	\[A^Q_{\varphi,B,d}(r_f,g_f;\phi) = \int_{J(\A_f)}{a_{(0,B,0,d)}(\exp(v_2 \otimes x) \overline{r_f}g_f) (\omega_{\psi_B}(r_f)\phi)(x)\,dx}.\]
	
	We define the $P$, $Q$ and $R$ symmetries.
	\begin{definition}\label{defn:symmetries}  We say the collection of functions $\{a_w\}_w$ satisfy the \textbf{$P$-symmetries} if $a_w(n g_f) =  \xi_w(n) a_w(g_f)$ for all $n \in N_P(\A_f)$ and
		\[ a_w(\gamma_f g_f) = \nu(\gamma)^{-\ell}|\nu(\gamma)|^{-1} a_{w \cdot \gamma}(g_f)\]
		for all $\gamma \in M_P(\Q)$.   We say the $\{a_w\}_w$ satisfy the \textbf{$Q$-symmetries} if, for all $g_f \in G(\A_f)$ and $r_f \in \widetilde{\SL}_2(\A)$
		\[ \sum_{n \in \Q_{> 0}}{A_{\varphi,B,-n}^Q(r_f,g_f;\phi) \mathcal{W}_{\SL_2,\ell',n}(r_\infty)}\]
		is the Fourier expansion of an automorphic form on $\widetilde{\SL_2}(\A)$ corresponding to a holomorphic modular form of weight $\ell' = \ell+1-\dim(J)/2$.  We say the collection of functions $\{a_w\}_w$ satisfy the \textbf{$R$-symmetries} if, for all $g_f \in G(\A_f)$, there is a cuspidal modular form on $M'$ (depending on $g_f$), whose restriction to $M'(\A_f) \times M'(\R)^0$ has Fourier expansion
		\[\sum_{u_1 > 0}{ \overline{A^R_{\varphi,u_1}(r_f;g_f;\phi)} W_{\ell_1,u_1}(r_\infty)}.\]
		Here $\ell_1 = \ell-\dim(C)$.
	\end{definition}
	Observe that if the $a_w$ satisfy the $P$-symmetries, then 
	\[ a_w(u_f g_f) W_{\ell,w}(u_\infty g_\infty) = a_{w}(g_f) W_{\ell, w}(g_\infty) \]
	for all $u \in N_P(\Q)$ and
	\[ a_w(\gamma_f g_f) W_{\ell,w}(\gamma_\infty g_\infty) = a_{w\cdot \gamma}(g_f) W_{\ell, w \cdot \gamma}(g_\infty) \]
	for all $\gamma \in M_P(\Q)$.
	
	We will prove the following theorem, after some preliminaries.
	\begin{theorem}\label{thm:Converse} Suppose the functions $a_w$ satisfy the $P$ and $R$ symmetries, and grow polynomially with $w$.  Then $\Psi(g)$ is a cuspidal quaternionic modular form on $G_J(\A)$ of weight $\ell$.\end{theorem}
	The $Q$-symmetries are not needed in Theorem \ref{thm:Converse}.
	
	\subsection{Fourier-Jacobi and automorphy}
	We will understand various Fourier-Jacobi coefficients of the function $\Psi(g)$.  To make sense of this, we start with the following lemma.
	\begin{lemma}\label{lem:NRinvariance} The function $\Psi: G(\A) \rightarrow \Vell$ is left-invariant by $N_R(\Q)$.
	\end{lemma}
	\begin{proof} Let 
		\[\Psi_1(g) = \sum_{w \in W_J(\Q), a(w) = 0}{a_w(g_f) W_{\ell;w}(g_\infty)}\]
		and 
		\[\Psi_2(g) =  \sum_{\gamma \in B(\Q)\backslash \SL_2(\Q)}\sum_{w \in W_J(\Q), a(w) \neq 0}{a_w(j_{E_{12}}(\gamma_f) g_f)W_{\ell;w}(j_{E_{12}}(\gamma_\infty) g_\infty)}.\]
		We will prove that each of $\Psi_1$ and $\Psi_2$ are left-invariant by $N_R(\Q)$.  For $\Psi_1$, observe that $N_R(\Q) = (N_R \cap N_P)(\Q) (N_R \cap M_P)(\Q)$.   Because the elements in $(N_R \cap M_P)(\Q)$ preserves the set of $w$ with $a(w) =0$, $\Psi_1$ is left-invariant by $N_R(\Q)$ because the $a_w$ satisfy the $P$-symmetries.
		
		For $\Psi_2$, observe that $j_{E_{12}}(\SL_2) \subseteq M_R$, so it normalizes $N_R$.  Because the set of $w$ with $a(w) \neq 0$ is preserved by $N_R(\Q)$, $\Psi_2$ is also preserved by $N_R(\Q)$.
	\end{proof}
	
	By virtue of Lemma \ref{lem:NRinvariance}, we can define a Fourier-Jacobi coefficient of $\Psi$.  To do so, let
	\[\Psi_T(g) = \int_{[V_7]}{\chi_T^{-1}(v)\Psi(vg)\,dv}.\]
	For $\phi \in S(X(\A))$, $r \in M'(\A)$ and $g \in G(\A)$, let
	\[\mathrm{FJ}_{T,\phi}(\Psi)(r;g) = \int_{[N_R]}{\Theta_\phi(nr) \Psi(nrg)\,dn}.\]
	Note that we can define this Fourier-Jacobi coefficient, without knowing if $\Psi$ or $\Psi_T$ has $M'(\Q)$-automorphy.  However, the $M'(\Q)$-automorphy of $\Psi_T$ can be detected by the automorphy of the Fourier-Jacobi coefficients.
	\begin{lemma}\label{lem:HilbertBasis} Suppose $\{\phi_\alpha\}_{\alpha}$ is a Hilbert basis of $S(X(\A))$ and $\{\phi_\alpha^\vee\}_{\alpha}$ is its dual basis.  Then
		\[\Psi_T(nrg) = \sum_{\alpha}{ \Theta_{\phi_{\alpha}^\vee}(nr) \mathrm{FJ}_{T,\phi_\alpha}(\Psi)(r;g)};\]
		the sum converges absolutely for fixed $n$, $r$ and $g$. If, as a function of $r \in M'(\A)$, $\mathrm{FJ}_{T,\phi_\alpha}(\Psi)(r;g)$ is automorphic, then $\Psi_T(\gamma rg) = \Psi_T(rg)$ for all $\gamma \in M'(\Q)$ and $r \in M'(\A)$.  In particular, $\Psi_T(\gamma g) = \Psi_T(g)$ for all $\gamma \in M'(\Q)$.
	\end{lemma}
	\begin{proof} This follows from \cite[section 1]{ikedaFJ}.
	\end{proof}
	
	We wish to prove the automorphy of $\Psi_T$.  We do this by proving the automorphy of its Fourier-Jacobi coefficients.  We will reduce down the set of $g$'s for which we need to prove the automorphy.
	
	\begin{lemma}\label{lem:FJequiv} If $n_1 \in N_R(\A)$, then $\mathrm{FJ}_{T,\phi}(\Psi)(r; n_1 g) = \mathrm{FJ}_{T,\omega(n_1)^{-1}\phi}(\Psi)(r;g)$.  Likewise, if $x \in M'(\A)$, then $\mathrm{FJ}_{T,\phi}(\Psi)(r;xg) = \mathrm{FJ}_{T,\omega(x)^{-1}\phi}(\Psi)(rx;g)$.  In particular, if the Fourier-Jacobi coefficient $\mathrm{FJ}_{T,\phi}(\Psi)(r;g)$ is $M'$-automorphic for all $\phi \in S(X(\A))$, then so is $\mathrm{FJ}_{T,\phi}(\Psi)(r;yg)$ for any $y \in (N_R(\A) \rtimes M'(\A))$.
	\end{lemma}
	\begin{proof} This follows from a change of variable in the integral defining $\mathrm{FJ}_{T,\phi}(\Psi)(r;xg).$
	\end{proof}
	
	The assumption that the $a_w$ satisfy the $R$-symmetries implies that the Fourier-Jacobi coefficient $\mathrm{FJ}_{T,\phi}(\Psi)(r;g_f)$ is automorphic for $g=g_f \in G(\A_f) \times \{1\}$.
	\begin{proposition}\label{prop:g=1Aut} Suppose the $a_w$ grow polynomially with $w$ and satisfy the $P$ and $R$ symmetries.  If $\phi \in S(X(\A))$ and $g_f \in G(\A_f)$, then $\mathrm{FJ}_{T,\phi}(\Psi)(r;g_f)$ is automorphic.\end{proposition}

	\begin{proof} Note that 
		\[\mathrm{FJ}_{T\,\phi}(\Psi)(r;g) = \int_{[N_R]}{\Theta_\phi(hr)\Psi(hrg)\,dh}\]
		automatically satisfies $\mathrm{FJ}_{T\,\phi}(\Psi)(\delta r;g) = \mathrm{FJ}_{T\,\phi}(\Psi)(r;g)$ for $\delta \in (M_R^T \cap N_P)(\Q)$.  Indeed, $\Theta_\phi(hr)$ is automorphic in $r$, so satisfies this invariance equation, and one sees that $\Psi(h \delta r g) = \Psi(h r g)$ for $\delta \in (M_R \cap N_P)(\Q)$.  For this latter invariance, it holds for the term $\Psi_1$ in the decomposition $\Psi = \Psi_1 + \Psi_2$.  For $\Psi_2$, one can handle it in cases: If $\delta = \exp(\alpha E_{12})$, then $\delta \in j_{E_{12}}(\SL_2)$, so the invariance is clear.  If $\delta = \exp(\gamma \delta_3 \otimes e_{11})$, then $\delta$ commutes with $j_{E_{12}}(\SL_2)$, so one again has invariance.  If $\delta \in \exp(v_1 \otimes H_2(C))$, then a $j_{E_{12}}(\SL_2(\Q))$ conjugate of $\delta$ lives in $\exp(v_1 \otimes H_2(C)+ v_2 \otimes H_2(C))$.  One obtains the invariance of $\Psi_2$ now using that the $a_w$ satisfy the $P$-symmetries.
		
		To prove the proposition, it suffices to assume $\phi = \phi_f \otimes \phi_\infty$ is a pure tensor in $S(X(\A)) = S(X(\A_f)) \otimes S(X(\R))$. Because $\mathrm{FJ}_{T\,\phi}(\Psi)(r;g)$ is invariant by $M_R^T(\Q) \cap N_P(\Q)$, it has a Fourier expansion.   In fact, the proof of Proposition \ref{prop:FJT1} goes over line-by-line to give 
		\[\mathrm{FJ}_{T,\phi}(\Psi)(r;g_f) = \sum_{u > 0}{A_{\Psi,u}^{R}(r_f;g_f;\phi_f) G_{T,u}(r_\infty,\phi_\infty)}\]
		where
		\[G_{T,u}(r_\infty,\phi_\infty) =\int_{\R \times X(\R)}{W_{\ell,w(u)}(\exp(s v_2 \otimes e_{11}) \exp(x) r_\infty)(\omega(r_\infty)\phi_\infty)(x)\,ds\,dx}.\]
		By the work of section \ref{sec:FJcomp}, $G_{T,u}(r_\infty,\phi_\infty)$ is proportional to the generalized Whittaker function $\mathcal{W}_{\ell_1,u}(r_\infty) \otimes (x-y)^{2\ell}$ on $M'(\R)$.  (See remark \ref{rmk:M'connected}.)  Thus, because the $a_w$ are assumed to satisfy the $R$-symmetries, there is an automorphic form $\alpha$ on $M'(\A)$ so that $\mathrm{FJ}_{T,\phi}(\Psi)(r;g_f) = \alpha(r) \otimes (x-y)^{2\ell}$.  This proves the proposition. 
	\end{proof}
	
	The following corollary follows from Lemma \ref{lem:HilbertBasis}, Lemma \ref{lem:FJequiv}, and Proposition \ref{prop:g=1Aut}.
	\begin{corollary}\label{cor:PsiAut1} Suppose $g \in G(\A_f) \times (N_R(\R)M'(\R)K_J)$ and $\gamma \in M'(\Q)$.  Then $\Psi_T(\gamma g) = \Psi_{T}(g)$.\end{corollary}
	
	Applying the identity theorem, Theorem \ref{thm:Identity}, we obtain the following strengthening of Corollary \ref{cor:PsiAut1}.
	\begin{corollary}\label{cor:PsiAut2} Suppose $g \in G(\A)$ and $\gamma \in M'(\Q)$.  Then $\Psi_{T}(\gamma g) = \Psi_{T}(g)$.
	\end{corollary}
	\begin{proof} Fix $g_f \in G(\A_f)$, and let $\mathcal{X} = N_R(\R)M'(\R)K_J$.  Set $F(g_\infty) = \Psi_T(\gamma g_f g_\infty) - \Psi_T(g_f g_\infty)$.  Then $F$ vanishes on $\mathcal{X}$ by Corollary \ref{cor:PsiAut1}.  Let $n= \dim(C) + 3$.  To see the necessary quaternionic transversality, we can work in $\SO(4,n+1) \supseteq \SO(3,n) \supseteq \SO(2,n)$, where it is easily verified.
	\end{proof}
	
	\subsection{Proof of the Converse theorem}
	We are now ready to prove Theorem \ref{thm:Converse}.
	\begin{proof}[Proof of Theorem \ref{thm:Converse}] Given Proposition \ref{prop:absConv}, it suffices to prove that $\Psi$ is automorphic, i.e., $\Psi(\gamma g) = \Psi(g)$ for all $\gamma \in G(\Q)$, and that $\Psi$ is cuspidal.  Once we prove that $\Psi$ is automorphic, the cuspidality follows immediately from the expansion \eqref{eqn:PsiDef} of $\Psi$, because only $w > 0$ appear in the sum.
		
		To prove the automorphy of $\Psi$, first observe that $\Psi(\gamma g) = \Psi(g)$ for all $\gamma \in M_J^1(\Q)$, because this group commutes with $j_{E_{12}}(\SL_2)$.  Because $G$ is exceptional, one sees easily that $G(\Q)$ is generated by $M_J^1(\Q)$ and $R(\Q)$.  (This property fails for the groups of type $B$ and $D$, because in that case $M_J^1$ is contained inside of $R$.)  Thus, it suffices to prove that $\Psi$ is left-invariant by $R(\Q)$.
		
		We first prove that, for $T$ normal, $\Psi_T$ is left-invariant by $M_R^T(\Q)$.  By Corollary \ref{cor:PsiAut2}, $\Psi_T$ is left-invariant by the derived group $M'(\Q)$.  Let $M_1$ denote the intersection of the Siegel Levi subgroup, inside the Heisenberg Levi subgroup, with $M_R$.  One has that $M_R^T(\Q)$ is generated by $M'(\Q)$ and $M_1(\Q)$; this follows from the Bruhat decomposition.  So, we need only check that $\Psi$ is left-invariant by $M_1(\Q)$.  This holds for $\Psi_1$, and for $\Psi_2$, it holds because $M_1(\Q)$ normalizes $j_{E_{12}}(B(\Q))$ and $j_{E_{12}}(\SL_2(\Q))$.  Thus, $\Psi_T$ is left-invariant by $M_R^T(\Q)$, for every normal $T$.
		
		Suppose $T_1 \in V_{7}(\Q)$ is arbitrary, with $q_{V_7}(T_1) > 0$.  We prove the following claim.
		\begin{claim}\label{claim:Converse1}
			There exists $\gamma_1 \in M_R(\Q)$ and $T \in V_{7}(\Q)$ normal so that $T_1 = T \cdot \gamma_1$ and $\Psi_{T\cdot \gamma_1}(g) = \Psi_{T}(\gamma_1 g)$.  
		\end{claim}
		\begin{proof}
			To deduce this statement about $T_1$, first suppose $(T_1, b_1)_{V_7} = 0$.  Let 
			\[\Psi_Z(g) = \sum_{w \in W_J(\Q), w>0}{a_w(g_f)W_{\ell,w}(g_\infty)}.\]
			Then $\Psi_{T_1} = (\Psi_Z)_{T_1}$.  Because $\Psi_Z$ is left-invariant by $N_P(\Q)$, and in particular by 
			\[M_R(\Q)^{[1]} = \exp( \mathrm{Span}_\Q(E_{12},v_1 \otimes H_2(C), \delta_3 \otimes e_{11})),\]
			one has $\Psi_{T_1 \cdot \mu}(g) = \Psi_{T_1}(\mu g)$ for any $\mu \in M_R(\Q)^{[1]}$.  We can find $\mu_1$ so that $T_2 = T_1 \cdot \mu_1$ satisfies $(T_2, b_{-1})_{V_7} = 0$.  We can find $\mu_2 \in (M_R \cap M_P)(\Q)$ so that $T_3 = T_2 \cdot \mu_2$ is normal.  Because $\Psi_{T_2} = (\Psi_Z)_{T_2}$, we have $\Psi_{T_2 \cdot \mu_2}(g) = \Psi_{T_2}(\mu_2 g)$.  Thus
			\[\Psi_{T_3}(g) = \Psi_{T_1 \cdot \mu_1 \mu_2}(g) = \Psi_{T_1}(\mu_1 \mu_2 g).\]
			This proves our claim for those $T_1$ iwht $(T_1, b_1)_{V_7} = 0$.
			
			Now suppose that $(T_1,b_{1})_{V_7} \neq 0$.  There is some $\mu \in j_{E_{12}}(\SL_2(\Q))$ with $(T_1 \cdot \mu, b_1)_{V_7} = 0$.  Then $\Psi_{T_1 \cdot \mu}(g) = \Psi_{T_1}(\mu g)$, because $\Psi_2$ is left-invariant by $j_{E_{12}}(\SL_2(\Q))$, and $\Psi_{T_1}(g) = (\Psi_2)_{T_1}(g)$.  But now, by what was just done, there is $T$ normal and $\gamma \in M_R(\Q)$ so that $T_1 \cdot \mu = T \cdot \gamma$ and $\Psi_{T_1 \cdot \mu}(g) = \Psi_{T \cdot \gamma}(g) = \Psi_{T}(\gamma g)$.  Thus $\Psi_{T_1}(\mu g) = \Psi_T(\gamma g)$ and our claim is proved.
		\end{proof}
		
		We require the following claim.
		\begin{claim}\label{claim:Convere2} Suppose $T_1 \in V_{7}(\Q)$ and $\Psi_{T} \neq 0$.  Then $q_{V_7}(T) > 0$.
		\end{claim}
		\begin{proof} First suppose $(T_1, b_1)_{V_7} =0$.  Then $\Psi_{T_1} = (\Psi_Z)_{T_1}$.  But if $w > 0$, and $w = (a,b,c,d)$, then $(b^\# - ac)_{11} > 0$.  Consequently, $q_{V_7}(T_1) > 0$, by the Fourier expansion of $\Psi_Z$.  If, on the other hand $(T_1, b_1)_{V_7} \neq 0$, then by the proof of Claim \ref{claim:Converse1}, there is $\mu \in j_{E_{12}}(\SL_2(\Q))$ with $(T_1 \cdot \mu,b_1)_{V_7} = 0$ and $\Psi_{T_1 \cdot \mu}(g) = \Psi_{T_1}(\mu g)$.   Because $\SL_2$ is its own derived group, and $j_{E_{12}}(\SL_2) \subseteq M_R$, it preserves the quadratic form on $V_7$.  Thus $q_{V_7}(T_1) = q_{V_7}(T_1 \cdot \mu) > 0$.
		\end{proof}
		
		One last claim.
		\begin{claim}\label{claim:Converse3} Suppose $T, T' \in V_{7}(\Q)$ are normal, and $q_{V_7}(T) = q_{V_7}(T') > 0$.  Then there is $\gamma' \in M_R(\Q)$ with $T' = T \cdot \gamma'$, and $\Psi_{T'}(g) = \Psi_T(\gamma' g)$.
		\end{claim}
		\begin{proof} The function $\Psi$ is left-invariant under $(M_J^1 \cap M_R)(\Q)$.  Thus $\Psi_{T}(\gamma' g) = \Psi_{T \cdot \gamma'}(g)$ for any $\gamma' \in (M_J^1 \cap M_R)(\Q)$.  This group acts transitively on the $S \in H_2(C)$ positive-definite with the same norm $n_{H_2(C)}(S)$.
			
			To handle the distinction between $S >0$ and $S < 0$, one uses the element $\diag(-1,-1)$ in the $\SL_2$ whose Lie algebra is generated by $v_2 \otimes e_{22}$ and $\delta_2 \otimes e_{22}$.
		\end{proof}

		Now, we have $\Psi(g) = \sum_{T_1 \in V_{7}(\Q)}{\Psi_{T_1}(g)}$.  By Claim \ref{claim:Convere2}, the sum can be taken over $T$ with $q_{V_7}(T) > 0$.  Suppose $\gamma \in M_R(\Q)$.  We claim $\Psi_{T_1}(\gamma g) = \Psi_{T_1 \cdot \gamma}(g)$.  To see this, let $\gamma_1 \in M_R(\Q)$ be as in Claim \ref{claim:Converse1}, and $T \in V_7(\Q)$ normal so that $T_1  = T \cdot \gamma_1$ and $\Psi_{T_1}(g) = \Psi_{T}(\gamma_1 g)$.  By Claim \ref{claim:Converse3}, we can assume $T$ is also positive-definite.  Now, again by Claim \ref{claim:Converse1} and Claim \ref{claim:Converse3}, there is $\delta \in M_R(\Q)$ so that $T_1 \cdot \gamma = T \cdot \delta$, and $\Psi_{T_1 \cdot \gamma}(g) = \Psi_{T}(\delta g)$.  Then
		\[ \Psi_{T_1 \cdot \gamma}(g) = \Psi_T(\delta g) = \Psi_T(\gamma_1 \gamma g) = \Psi_{T_1}(\gamma g)\]
		because $T \cdot \delta = T \cdot (\gamma_1 \gamma)$ and so $\gamma_1 \gamma \delta^{-1} \in M_R^T(\Q)$.  This completes the proof.
	\end{proof}

	\section{Absolute convergence}\label{sec:absConv}
	The purpose of this section is to prove Proposition \ref{prop:absConv}.  
	
	\subsection{Preliminaries} We begin by defining various norms we will use.  On $\Vell$, let $J_{2} = \mm{0}{1}{-1}{0}$ be the map defined by $x \mapsto -y$, $y \mapsto x$.  If $u = \sum_{v}{u_v x^{\ell+v}y^{\ell-v}} \in \Vell$, define $\overline{u} = \sum_{v}{\overline{u_v} x^{\ell+v}y^{\ell-v}} \in \Vell$.  We set $u^* = \overline{-J_2 u}$ if $u \in \Vell$.  The pairing $(u_1, u_2) \mapsto \langle u_1, u_2^* \rangle_{K_J}$ is $K_J$-equivariant and positive-defininte.  In fact, if $u = \sum_{v}{u_v x^{\ell+v}y^{\ell-v}} \in \Vell$, then
	\[ \langle u, u^* \rangle_{K_J} = \sum_{v} (\ell+v)! (\ell-v)! |u_v|^2.\]
	Define $||u|| = (\langle u, u^* \rangle_{K_J})^{1/2}$.
	
	We now define a norm on $G(\R)$.  Recall the positive-definite form $B_{\theta}(\cdot, \cdot): \g \times \g \rightarrow \R$ from \cite[section 4.1.3]{pollackQDS}.  If $\{X_\alpha\}$ is a basis of $\g$, and $X_\alpha^\vee$ is the dual basis with respect to the pairing $B_{\theta}$, then $\sum_{\alpha}{B_\theta(g X_\alpha, g X_\alpha^\vee)}$ is independent of the choice of basis.  Define $||g|| = \left(\sum_{\alpha}{B_\theta(g X_\alpha, g X_\alpha^\vee)}\right)^{1/2}$.   
	
	Because $B_\theta$ is $K_J$-invariant, $||k_1 g k_2 || = ||g||$ for any $k_1, k_2 \in K_J$.  Applying the Cartan decomposition, one deduces $||g|| = ||g^{-1}||$.  If $m \in M_P(\R)$ and $n \in N_P(\R)$, then $||nm|| \geq ||m||$.  One verifies this inequality by choosing a basis $X_\alpha$ compatible with $n, m$.  Letting some $X_\alpha = E_{13}$, one sees $||m|| \geq |\nu(m)|$.  
	
	For $w \in W_J(\R)$, recall $||w||_W = (\langle w, J_2 w \rangle)^{1/2}$.  One has $B_\theta(w,w) = ||w||^2$.  For $x \in M_P(\R)$, define $||x||_W$ via $||x||_W^2 = \sum_{\alpha}{B_\theta(x X_\alpha, x X_\alpha)}$, where $X_\alpha$ is an orthonormal basis of $W$ with respect to $B_\theta$.  Then $||x w||_W \leq ||x||_W ||w||_W$ for all $x \in M_P(\R)$ and $w \in W_J(\R)$.  Moreover, $||x||_W \leq ||x||$.
	
	The following lemma is crucial.
	\begin{lemma}\label{lem:PosDefLem} Suppose $u = (a,b,c,d) \in W_J(\R)$.  Then
		\[|\langle u, r_0(i)\rangle|^2 = ||u||^2 + 2(b^\#-ac,1_J) + 2(c^\#-db,1_J).\]
		In particular, if $u > 0$, then $|\langle u, r_0(i)\rangle| \geq ||u||.$ 
	\end{lemma}
	\begin{proof}
		Observe that $\langle u, r_0(i) \rangle = ((b,1)-d) + i(a-(c,1))$, so 
		\begin{align*} | \langle u, r_0(i) \rangle|^2 &= ((b,1)-d)^2 + (a-(c,1))^2 \\ 
			&= a^2 + (b,b) + (c,c) + d^2 + 2(b^\#-ac,1) + 2 (c^\#-db,1)\end{align*}
		using that $(x,1)^2 = (x,x) + 2(x^\#,1)$. This proves the lemma.
	\end{proof}

	We use these norms to prove the following lemma.
	\begin{lemma} There is a positive constant $C_{\ell}$, depending on $\ell$ and $G$, so that for $w \in W_J(\R)$ satisfying $w > 0$, 
		\[ ||W_{\ell,w}(g)|| \leq C_{\ell} ||g||^{\ell+1}K_{\ell}(||w|| \cdot ||g||^{-1}).\]
	\end{lemma}
	\begin{proof} Let $g = nmk$.  Then 
		\begin{align*}||W_{\ell,w}(g)|| &= ||W_{\ell,w}(m)|| = |\nu(m)|^{\ell+1} ||W_{\ell, w \cdot m}(1)|| \leq ||g||^{\ell+1} W_{\ell, w\cdot m}(1)|| \\ &\leq C_{\ell} ||g||^{\ell+1} K_{\ell}(|\langle w \cdot m, r_0(i)\rangle|).\end{align*}
		Here we have used that the functions $K_v$ satisfy $K_r(x) \leq K_s(x)$ if $0 \leq r \leq s$, which is verified immediately using the integral expression for $K_v(x)$ as $\frac{1}{2}\int_{1}^{\infty}{(t^v+t^{-v})e^{-x(t+t^{-1})/2}\,\frac{dt}{t}}$.
		
		Because $w \cdot m > 0$, by Lemma \ref{lem:PosDefLem},
		\[ |\langle w \cdot m, r_0(i) \rangle| \geq ||w \cdot m||_W \geq ||w|| \cdot ||m^{-1}||_W^{-1} \geq ||w|| \cdot ||g||^{-1}
		\]
		using that $||m^{-1}||_W \leq ||m^{-1}|| = ||m|| \leq ||g||$.  Because $K_{\ell}$ is a decreasing function, the lemma follows.
	\end{proof}
	
	\subsection{Moderate growth}\label{subsec:moderate}
	In this subsection, we prove that, for each fixed $g_f \in G(\A_f)$, the sum defining $\Psi$ converges absolutely to a function of moderate growth.
	
	Because $B(\Z)\backslash \SL_2(\Z) \rightarrow B(\Q)\backslash \SL_2(\Q)$ is a bijection, in the definition of $\Psi$ we only need to sum over elements of $\SL_2(\Z)$.  We have $\Psi(g) = \Psi_Z(g) + \Psi_2'(g)$, where
	\[\Psi_Z(g) = \sum_{w \in W_J(\Q), w > 0}{a_w(g_f) W_{\ell,w}(g_\infty)}\]
	and
	\[\Psi_2'(g) = \sum_{\gamma \in R}\sum_{w \in W_J(\Q), a(w) \neq 0, w > 0}{a_w(\gamma_f g_f) W_{\ell,w}(\gamma_\infty g_\infty)}\]
	where $R$ denotes the subset of $j_{E_{12}}(B(\Z)\backslash \SL_2(\Z))$ consisting of the non-identity cosets.
	
	\begin{lemma}\label{lem:PsiLambda} Fix $g_f$.  There is a lattice $\Lambda \subseteq W_J(\Q)$ so that, if $\gamma \in j_{E_{12}}(\SL_2(\Z))$ and $a_w(\gamma_f g_f) \neq 0$, then $w \in \Lambda$.  Moreover, there are constants $C,N >0$, possibly depending on $g_f$ but independent of $\gamma$ so that $|a_w(\gamma_f g_f)| \leq C || w ||^N$ for all $w$.\end{lemma}
	\begin{proof} Without loss of generality, we can assume $g_f  = 1$.  Assume the $a_w$'s are right-invariant by the open compact subgroup $U$ of $G(\A_f)$.  Let $U' \subseteq \SL_2(\A_f)$ be an open compact such that $j_{E_{12}}(U') \subseteq U$.  There are finitely many $k_j \in \SL_2(\widehat{\Z})$ so that $\SL_2(\Z) \subseteq \bigcup_{j} k_j U'$.  For each $k_j$, there is a lattice $\Lambda_j$ so that $a_w(k_j) \neq 0$ implies $w \in \Lambda_j$.  The lemma follows easily.
	\end{proof}
	
	Let $\Lambda$ be as in Lemma \ref{lem:PsiLambda}.  For a real number $\alpha > 0$, set
	\[F_Z(\alpha) = \sum_{w \in \Lambda, w > 0}{ ||w||^{N} K_{\ell}(||w|| \cdot \alpha^{-1})}.\]
	Suppose $\gamma = \mm{a}{b}{c}{d} \in \SL_2(\R)$.  Let $z_\gamma = \gamma \cdot i = x_\gamma + i y_\gamma$. Then $y_\gamma = |ci+d|^{-2}$.  Define 
	\[t_\gamma = \diag(y_\gamma^{1/2},y_\gamma^{-1/2}) = \diag(|ci+d|^{-1},|ci+d|).\]
	
	Observe that for $\gamma \in R$ and $g \in G(\R)$,
	\begin{align*} ||W_{\ell,w}(\gamma_\infty g)|| &=  ||W_{\ell,w}(j_{E_{12}}(t_\gamma k) g)|| = |\nu(j_{E_{12}}(t_\gamma))|^{\ell+1} ||W_{\ell,w \cdot j_{E_{12}}(t_\gamma)}(j_{E_{12}}(k) g)|| \\ &
		\leq C_{\ell} (c^2+d^2)^{-(\ell+1)/2} ||g||^{\ell+1} K_{\ell}(||w \cdot j_{E_{12}}(t_\gamma)|| \cdot ||g||^{-1})\end{align*}
	for some $k \in \SO(2) \subseteq \SL_2(\R)$.  Here we are using that $\nu(j_{E_{12}}(\diag(t,t^{-1}))) = t$.  One has
	\[ (a,b,c,d) \cdot j_{E_{12}}(\diag(t,t^{-1})) = (t^{-1}a,b, tc, t^2 d).\]
	For $\alpha > 0$, set
	\[F_{2}(\alpha) = \sum_{w \in \Lambda, w> 0, a(w) \neq 0}\sum_{\gamma = \mm{*}{*}{c}{d} \in R} ||w||^{N} (c^2+d^2)^{-(\ell+1)/2} K_{\ell}(||w \cdot j_{E_{12}}(t_\gamma)|| \cdot \alpha^{-1}).\]
	
	We wish to bound $F_Z(\alpha)$ and $F_2(\alpha)$ as functions of $\alpha \geq 1$.  Here is a useful lemma.
	\begin{lemma} The function $e^y K_v(y)$ is decreasing on $(0, \infty)$.   \end{lemma}
	\begin{proof} From the integral representation of $K_v(y)$ we have
		\[2 K_v(y) = \int_{0}^{\infty}{t^v e^{-y(t+t^{-1})/2} \, \frac{dt}{t}} = \int_{1}^{\infty}{ (t^v + t^{-v}) e^{-y(t+t^{-1})/2} \, \frac{dt}{t}}.\]
		Consequently $2 e^y K_v(y) = \int_{1}^{\infty}{ (t^v + t^{-v}) e^{-y(t-2+t^{-1})/2} \, \frac{dt}{t}}.$  Differentiating under the integral sign gives
		\[ -2 \frac{d}{dy} (e^y K_v(y)) = \int_{1}^{\infty}{ (t^v + t^{-v}) (t^{1/2} - t^{-1/2})^2 e^{-y(t^{1/2} - t^{-1/2})^2}\,\frac{dt}{t}}.\]
		This is positive, proving the lemma. 
	\end{proof}
	
	We begin by bounding $F_2(\alpha)$.
	\begin{proposition}\label{prop:F2alpha} For $\alpha \geq 1$, there are constants $R,S > 0$ so that $F_2(\alpha) \leq R \alpha^S$.
	\end{proposition}
	\begin{proof}
		Suppose $w \in \Lambda$, $a(w) \neq 0$.  Then $|a(w)| \geq \epsilon > 0$ for some $\epsilon$, independent of $w$.  We have
		\[  ||w \cdot j_{E_{12}}(t_\gamma)|| \geq  |a(w)| (c^2+d^2)^{1/2} \geq  \epsilon.\]
		The function $y^{\ell+1}K_{\ell}(y)$ is bounded on $(0,\infty)$.  Thus 
		\[K_v(||w \cdot j_{E_{12}}(t_\gamma)|| \cdot \alpha^{-1}) \leq (e^y K_{\ell}(y))|_{y = \alpha^{-1} \epsilon} e^{-||w \cdot j_{E_{12}}(t_\gamma)|| \cdot \alpha^{-1}} \leq C_{\ell,\Lambda} \alpha^{\ell+1} e^{-||w \cdot j_{E_{12}}(t_\gamma)|| \cdot \alpha^{-1}}\]
		for some positive constant $C_{\ell,\Lambda}$ independent of $\alpha$.
		
		To get rid of the term $||w||^N$, we first observe the following lemma.
		\begin{lemma}\label{lem:expBound} One has $v^N e^{-r v} \leq (N/r)^N e^{-N}$ for all $v \geq 0$.
		\end{lemma}
		\begin{proof} Setting $f(v) = v^N e^{-rv}$, one computes $f'(v) = e^{-rv}(N-rv)v^{N-1}$ and the lemma follows.
		\end{proof}
		We have $||w \cdot  j_{E_{12}}(t_\gamma)|| \geq (c^2+d^2)^{-1} ||w||$ and so
		\[ ||w||^N e^{-||w \cdot j_{E_{12}}(t_\gamma)|| \cdot \alpha^{-1}/2} \leq ||w||^N e^{-(c^2+d^2)^{-1} ||w|| \cdot \alpha^{-1}/2} \leq (2 (c^2+d^2)\alpha)^{N} e^{-N}.\]
		Thus, to bound $F_2(\alpha)$, it suffices to bound
		\[F_3(\alpha) =  \sum_{w \in \Lambda, w> 0, a(w) \neq 0}\sum_{\gamma = \mm{*}{*}{c}{d} \in R} (c^2+d^2)^{M} e^{-(2\alpha)^{-1} \cdot (||w \cdot j_{E_{12}}(t_\gamma)||)}.\]
		
		We will use the following elementary lemma.
		\begin{lemma}\label{lem:elementary} One has $\frac{1}{1-e^{-r}} \leq 1 + r^{-1}$ for all $r > 0$.
		\end{lemma}
		\begin{proof} For $r \geq 0$ we have $1 + r \leq e^r$, so $r \leq e^r -1$, so $\frac{e^{-r}}{1-e^{-r}} = \frac{1}{e^r-1} \leq r^{-1}$.  Thus $\frac{1}{1-e^{-r}} = 1 + \frac{e^{-r}}{1-e^{-r}} \leq 1 + r^{-1}$.
		\end{proof}
		
		Assume without loss of generality that $\Lambda = A^{-1} \Lambda_0$, where $A > 0$ and $\Lambda_0 = \Z \oplus J_0 \oplus J_0 \oplus \Z = \Z \oplus \Lambda_0^1$.  Choose a basis of $\Lambda_0$ subordinate to this decomposition, and use that basis to define a taxicab norm $|| \cdot ||_1$ on $\Lambda \otimes \R$.  Summing up a geometric series, we then have 
		\[\sum_{v\in \Lambda_0^1}{\exp(-(R_\Lambda A\alpha (c^2+d^2))^{-1} ||v||_1)} \leq (1-e^{-(R_\Lambda A \alpha (c^2+d^2))^{-1}})^{-\dim(\Lambda_0^1)} \leq D_{\Lambda} (\alpha (c^2+d^2))^{\dim \Lambda_0^1}\]
		for some constant $D_\Lambda$ that only depends on $\Lambda$.  Here we have applied Lemma \ref{lem:elementary}.
		
		Thus, we are left to bound
		\[ \sum_{\gamma \in R}\sum_{n \geq 1} (c^2+d^2)^M e^{-S_\Lambda \alpha^{-1}  (c^2+d^2)^{1/2} n}\]
		for some constant $S_\Lambda >0$ that only depends on $\Lambda$.  Applying the same techniques as above, we can sum the geometric series, and bound this in terms of a power of $\alpha$.
	\end{proof}
	
	The bounding of $F_Z(\alpha)$ is easier. 
	\begin{proposition}\label{prop:FZalpha} There are constants $R, S >0$ so that $F_Z(\alpha) \leq R \alpha^S$.
	\end{proposition}
	\begin{proof} The proposition can be proved using the same techniques as used in the proof of Proposition \ref{prop:F2alpha}.
	\end{proof}
	We have now proved that the sum defining $\Psi$ converges absolutely to a function of moderate growth.
	
	\subsection{Derivatives} In this subsection, we prove that $\Psi$ is $\mathcal{Z}(\g)$-finite, and satisfies $D_{\ell}\Psi \equiv 0$.  The idea of the proof is simple.  One has $D_{\ell}W_{\ell,w}(\gamma g) = 0$ for any $\gamma$.  Thus, $D_{\ell} \Psi \equiv 0$, if differentiation term-by-term can be justified.  Likewise, suppose $Z \in \mathcal{Z}(\g)$.  By the uniqueness theorem regarding the generalized Whittaker functions $W_{w}(g)$ \cite{wallach, pollackQDS}, one sees easily that $Z W_{\ell,w}(g) = \lambda_{\ell} W_{\ell,w}(g)$ for some constant $\lambda_\ell$ that is independent of $w$.  (To see that $\lambda$ is independent of $w$, one uses the relation $W_{\ell, w}(mg) = \nu(m)^{\ell}|\nu(m)| W_{\ell,w \cdot m}(g)$ for $m \in M_P(\R)$.) Thus $Z \Psi = \lambda_\ell \Psi$, if differentiation term-by-term can be justified.
	
	To justify the term-by-term differentiation, we prove the following proposition. For $m \in M_P(\R)$, recall that $\alpha_w(m) = \langle w \cdot m, r_0(i) \rangle$.  Fix $w > 0$.  Let $\mathcal{F}_w$ denote the set of smooth functions $f: P(\R) \rightarrow \C$ that satisfy $f(np) = e^{i \langle w, \overline{n}\rangle} f(p)$ and $f(m)$ is a finite sum of functions of the form $P_v(w \cdot m) |\alpha_v(m)|^{-v} K_v(|\alpha_w(m)|)$ for integers $v$ and polynomials $P_v$ on $W_J(\C)$.
	\begin{proposition} Let the notation be as above.  The space $\mathcal{F}_w$ is closed under the right differentiation by $Lie(P(\R))$.
	\end{proposition}
	\begin{proof}
		First suppose that we differentiate with respect to $X \in Lie(N_P(\R))$, and then evaluate at $m \in M_P(\R)$.  If $f \in \mathcal{F}_w$, then 
		\[X(f(m)) = \frac{d}{dt}|_{t=0}( f(m e^{tX})) = \frac{d}{dt}|_{t=0}( e^{i \langle w, m \cdot X \rangle} f(m)) = i \langle w m, X \rangle f(m).\]
		So, the form is preserved with the degree of polynomial increasing by $1$.
		
		Now suppose we differentiate with respect to $X \in Lie(M_P(\R))$.  For ease of notation, let $\alpha = \alpha_w(m)$. First observe 
		\[X P_v(w \cdot m) = \frac{d}{dt}|_{t=0} P(w\cdot m+ t(w \cdot m)X) = \langle P'(w \cdot m), (w\cdot m) X\rangle\]
		is still a polynomial of the same degree.  Moreover, one has
		\begin{align*} X(|\alpha)|) &= \frac{1}{2 |\alpha|} X(|\alpha|^2) \\ &= \frac{1}{2|\alpha|} ( \langle w \cdot m, X r_0(i) \rangle \langle w \cdot m, r_0(-i) \rangle + \langle w \cdot m, r_0(i) \rangle \langle w \cdot m,X r_0(-i) \rangle ) \\&= \frac{1}{2|\alpha|} Q(w \cdot m)\end{align*}
		where $Q$ is a quadratic polynomial.  
		
		Finally, recall the formula $\frac{d}{du} (u^{-v} K_v(u)) = - u^{-v} K_{v+1}(u)$.  Combining, we obtain
		\[X (|\alpha|^{-v} K_v(|\alpha|) )= - Q(w \cdot m) |\alpha|^{-v-1} K_{v+1}(|\alpha|)\]
		so the form is still preserved.
	\end{proof}
	
	The estimates of subsection \ref{subsec:moderate} work just as well with the generalized Whittaker functions $W_{\ell,w}(g)$ replaced by $K_J$-equivariant functions on $G(\R)$ whose restriction to $P(\R)$ has components in $\mathcal{F}_w$.  Thus, if $Z$ is either $D_{\ell}$ or in $\mathcal{Z}(\g)$, term-by-term differentiation by $Z$ holds for $\Psi(g)$.  This completes the proof of Proposition \ref{prop:absConv}.
	
	\section{Reduction theory}\label{sec:reduction}
	One of the key tools to prove the automatic convergence theorem is reduction theory.  In this section, we collect together and prove the results we need in this direction.
	
	\subsection{Orthogonal groups}
	We begin by discussing reduction theory for orthogonal groups. 
	
	For a quadratic form $g: V \rightarrow \R$ on a vector space $V$, let $\langle x, y \rangle_g = g(x+y) - g(x) - g(y)$ be the bilinear form associated to $g$.  If $\Lambda' \subseteq V$ is a lattice, not necessarily of full rank, let $\det(\Lambda';g) = \det(\langle b_i, b_j \rangle_g)$ where $\{b_i\}$ is a $\Z$ basis of $\Lambda'$.
	
	Note that if $f, g$ are two quadratic forms on a vector space $V$, then they give maps $V \rightarrow V^\vee$.  If $g$ is non-degenerate, then $g^{-1} \circ f$ can be considered a linear map from $V$ to $V$.
	
	The following result of \cite{schlickewei} is crucial; see also \cite[Theorem 10.2]{blevins}.
	\begin{theorem} \label{thm:Schlickiwei} Let $\Lambda$ be a lattice, $f: \Lambda \rightarrow \Z$ a non-degenerate quadratic form, and $g: \Lambda \otimes \R \rightarrow \R$ a positive-definite quadratic form.  Assume $f$ has Witt rank $r \geq 1$.  There is a universal constant $C_{n}$ that only depends on $n = \dim(\Lambda \otimes \R)$ so that there exists a totally isotropic rank $r$ sublattice $\Lambda' \subseteq \Lambda$ with
		\[\det(\Lambda';g) \leq C_{n} \det(\Lambda;g) \tr((g^{-1} \circ f)^2)^{(n-r)/2}.\]
	\end{theorem}
	
	We recall the definition of a \emph{majorant} of a quadratic form.  Suppose $f$ is a non-degenerate quadratic form on a real vector space $V$.  Suppose $V = V' \oplus V''$ and $f$ is positive definite on $V'$, negative definite on $V''$, and $V', V''$ are orthogonal with respect to $f$.  Define a new quadratic form, $g$, on $V$ by flipping the sign on $V''$, so $g$ is positive definite on $V$.  The form $g$ is called a majorant of $f$.
	
	We will apply Theorem \ref{thm:Schlickiwei} when $g$ is a majorant of $f$, and use the following lemma.
	\begin{lemma}\label{lem:S21} Suppose $g$ is a majorant of the non-degenerate quadratic form $f$. 
		\begin{enumerate}
			\item Understand $f,g$ to be isomorphisms $V \rightarrow V^\vee$.    Then $g^{-1} \circ f: V\rightarrow V$ satisfies $(g^{-1} \circ f)^2 = Id_V$ is the indentity on $V$.
			\item If $\Lambda \subseteq V$ is a full rank lattice, then $|\det(\Lambda;f)| = \det(\Lambda;g)$.
		\end{enumerate}
	\end{lemma}
	\begin{proof} The first part is clear by considering what happens on $V'$ and $V''$.  For the second part, let $\lambda_1, \ldots, \lambda_n$ be a basis of $\Lambda$, and $\delta_1, \ldots, \delta_n$ the dual basis of $V^\vee$.   Let $F$ be the matrix for $f: V \rightarrow V^\vee$ with respect to these bases, and likewise let $G$ be the matrix for $g: V \rightarrow V^\vee$.  Then $F$ has matrix entries $F_{ij} = \langle \lambda_i, f(\lambda_j) \rangle = (\lambda_i,\lambda_j)_f$ and $G$ has entries $G_{ij} = \langle \lambda_i, g(\lambda_j) \rangle = (\lambda_i,\lambda_j)_g$.   
		
		Let $S$ be the matrix for  $g^{-1} \circ f$ with respect to the basis $\lambda_1, \ldots, \lambda_n$ of $V$.	We have $F = GS$ and $S^2 =1$, so the lemma follows by taking determinants.
	\end{proof}
	
	We need the following elementary lemma.
	\begin{lemma} Suppose $\Lambda \subseteq V$ is a lattice in a quadratic space, with integral quadratic form $q$.  Let $T \in \Lambda$ have $q(T) \neq 0$.  Let $V_T$ denote the orthogonal complement to $T$ and $\Lambda_T = \Lambda \cap V_T$.  Then $\det(\Lambda_T;q)$ divides $(T,T) \det(\Lambda;q)$.
	\end{lemma}
	\begin{proof}  The set $\{(b,T): b \in \Lambda\}$ is a nonzero ideal in $\Z$, equal to, say $r_0\Z$.  Let $b_0 \in \Lambda$ satisfy $(b_0,T) = r_0$.  The vector $b_0$ is primitive, so extends to a basis $b_0, b_1, \ldots, b_n$ of $\Lambda$.  Set $b_j' = b_j - r_0^{-1}(b_j,T) b_0$.  Then $b_j' \in \Lambda$ and $(b_j',T) = 0$.  We have $b_0, b_1', \ldots, b_n'$ is a basis of $\Lambda$, and $b_1',\ldots, b_n'$ is a basis of $\Lambda_T$.   Now, $(T,T) b_0 - r_0 T \in \Lambda_T$.  Thus
		\[ \mathrm{Span}_\Z(r_0 T, b_1',\ldots, b_n') =
		\mathrm{Span}_\Z((T,T)b_0, b_1',\ldots, b_n') \subseteq \mathrm{Span}_{\Z}(T,b_1',\ldots, b_n').\]
		Taking determinants gives 
		\[(T,T) \det(\Lambda_T) r_0^2 = (T,T)^2 \det(\Lambda).\]
		This gives the lemma.
	\end{proof}
	
	Suppose now $(S,q_S)$ is a rational quadratic space of Witt rank exactly $1$, and that $S \otimes \R$ has signature $(1, n_1)$ with $n_1 \geq 1$.  In our case of interest, $S = H_2(C)$ with quadratic form given by $q_S = n_{H_2(C)}$, but we work more generally for now.  Assume given a lattice $\Lambda_S \subseteq S$ on which $q_S$ in integral.  Fix $T \in \Lambda_S$ with $q_S(T) > 0$.  Let $S_T \subseteq S$ be the perpendicular space to $T$.  Let $\Lambda_S^\vee$ be the dual lattice to $\Lambda_S$.  Let $\Lambda_{S,T} = \Lambda_S \cap S_T$ and $\Lambda_{S,T}^\vee$ its dual lattice.  One can identify $\Lambda_{S,T}^\vee$ with $\Lambda_{S}^\vee/(\Lambda_{S}^\vee \cap \Q T)$.
	
	We let 
	\[\Lambda_T = \Z b_1 \oplus \Z b_2 \oplus \Lambda_{S,T} \oplus \Z b_{-2} \oplus \Z b_{-1}\]
	with quadratic form 
	\[q((\alpha_1, \alpha_2, \lambda, \alpha_{-2}, \alpha_{-1}) = \alpha_1 \alpha_{-1} + \alpha_2 \alpha_{-2} + q_S(\lambda).\]
	Let $V_T = \Lambda_T \otimes \R$, and $V_T^1$ the orthogonal complement of $\mathrm{Span}(b_1,b_{-1})$ in $V_T$.  Let $\Lambda_T^1 = V_T^1 \cap \Lambda_T$.
	
	Let $\Gamma_{T}$ be the arithmetic group $\SO(\Lambda_T,q) \cap SO(\Lambda_T)(\R)^0$.  We understand the reduction theory of $\Gamma_T$ acting on 
	\[\mathcal{H}_T = \{x+iy: x,y \in V_T^1, y>0\}.\]
	Here $y > 0$ means $q(y) > 0 $ and $(y, b_2+b_{-2}) > 0$.
	
	Let $\mathcal{C}_T$ be a compact subset of $V_{S,T}:=\Lambda_{S,T} \otimes \R $ so that if $v \in V_{S,T}$ there exists $\lambda \in \Lambda_{S,T}$ so that $v - \lambda \in \mathcal{C}_T$.   Let $M_T \in \R_{>0}$ be chosen so that $v \in \mathcal{C}_T$ implies $|(v,v)| \leq M_T$.  We will use the following bound on $M_T$.
	\begin{proposition}\label{prop:MT} Let the notation be as above.  There is a positive constant $C_n$, only depeding on $n$, so that there is $\mathcal{C}_T$, $M_T$ satisfying $M_T \leq C_n (T,T)^2$.\end{proposition}
	\begin{proof} The proof follows from the following more general reformulation.\end{proof}
	
	\begin{proposition} Suppose $L$ is a lattice, and $R$ is a positive-definite quadratic form on $L$. Let $n = \dim(L \otimes \R)$. Assume $R$ is integrally-valued on $L$.   Then, given $v \in L \otimes \R$, there is a fundamental domain $P$ for $L$ in $L \otimes \R$ so that if $v \in P$ then $(v,v)_R \leq C_n \det(L;R)^2$.
	\end{proposition}
	\begin{proof} The proof is essentially taken from \cite{micciancio}.  We give some details for the convenience of the reader.  We begin with the following claim.
		
		\begin{claim} Suppose $b_1, \ldots, b_n$ is a basis of $L$.  Let $b_1^*, \ldots, b_n^*$ be the basis of $L \otimes \R$ obtained from $b_1, \ldots, b_n$ by Gram-Schmidt orthogonalization.  That is, $b_1^* = b_1$, $b_2^* = b_2 - \frac{(b_1,b_2)}{(b_1, b_1)} b_1$, and one iteratively defines
			\[b_{j+1}^* = b_{j+1} - \mathrm{proj}_{b_j^*}(b_{j+1}) - \ldots -\mathrm{proj}_{b_1^*}(b_{j+1}) \]
			where $\mathrm{proj}_{y}(x) = x - \frac{(x,y)}{(y,y)}y$.  Set
			\[P = \{\alpha_1 b_1^* + \cdots + \alpha_n b_n^*: \alpha_j \in [-1/2,1/2]\}.\]
			Then $P$ is a fundamental region for $L$ in $L \otimes \R$.
		\end{claim}
		\begin{proof} Suppose $v \in L \otimes \R$, $v = \beta_1 b_1^* + \cdots + \beta_n b_n^*$.  Say $\beta_n - r_n \in [-1/2,1,2]$.  Then we subtract off $r_n b_n$ from $v$.  The coefficients $\beta_1, \ldots, \beta_{n-1}$ may change.  Then, we repeat with $b_{n-1}, b_{n-1}^*$ in place of $b_n$, $b_n^*$.  Iterating gives the claim.
		\end{proof}
		
		Let $\lambda_1, \ldots, \lambda_n$ be the successive minimal of $L$ with the quadratic form $R$.  Let now $v_1, \ldots, v_n$ be linearly independent with all $(v_j, v_j) \leq \lambda_n^2$.  Define $P$ as above from the Gram-Schmidt orthogonalizations $v_1^*, \ldots, v_n^*$.
		\begin{claim} If $v \in P$, then 
			\[(v,v)_R \leq \frac{1}{4}\sum_j{(v_j^*,v_j^*)} \leq \frac{n}{4}\lambda_n^2.\]
		\end{claim}
		\begin{proof} This is clear, as $(v_j^*, v_j^*) \leq (v_j,v_j) \leq \lambda_n^2$.
		\end{proof}
		
		The successive minima of the pair $L,R$ can be related to the determinant $\det(L;R)$.
		\begin{claim} One has $\prod_{j}{\lambda_j} \leq \gamma_n^{n/2} \det(L;R)$, where $\gamma_n$ is Hermite's constant.
		\end{claim}
		\begin{proof} This is Theorem 12 in \cite{micciancio}.
		\end{proof}
		
		Finally, because $R$ is integral on $L$, $\lambda_{n-1} \geq \cdots \geq \lambda_1 \geq 1$.  Thus, $\lambda_n \leq \gamma_n^{n/2} \det(L;R)$.  The proposition follows.
	\end{proof}	
	
	We now present reduction theory for $\Gamma_T$ acting on $\mathcal{H}_T$, and phrase the results partially adelically.  Let $G_T$ denote the algebraic group $\SO(\Lambda_T \otimes \Q)$.  Let $Q_T$ denote the parabolic subgroup of $G_T$ stabilizing $\mathrm{Span}_\Q(b_1,b_2)$.  Let $\mathcal{R}_{Q,T}$ denote a finte set of representatives for $\Gamma_T \backslash G_T(\Q)/Q_T(\Q)$.  
	\begin{claim}\label{claim:muIntBasis} The representatives $\mu \in \mathcal{R}_{Q,T}$ can be chosen so that $\mu b_1, \mu b_2$ are an integral basis of  $\mathrm{Span}_{\Q}(\mu b_1, \mu b_2) \cap \Lambda_T$.  
	\end{claim}
	\begin{proof} Indeed, to see that this can be done, suppose $\delta \in G_T(\Q)$.  Let $x_1', x_2'$ be an integral basis for $\mathrm{Span}_{\Q}(\delta b_1, \delta b_2) \cap \Lambda_T$.  Then $\delta^{-1} x_1', \delta^{-1} x_2' \in \mathrm{Span}_{\Q}(b_1, b_2)$ is a basis, so $\delta^{-1} x_1' = q b_1$, $\delta^{-1} x_2' = q b_2$ for some $q \in Q_T(\Q)$.  Thus $\delta q b_1 = x_1'$, $\delta q b_2 = x_2'$, so by right multiplying $\delta$ by some $q \in Q_T(\Q)$, we can assume that $\delta b_1, \delta b_2$ is an integral basis of $(\delta \mathrm{Span}_{\Q}(b_1,b_2)) \cap \Lambda_T$.  Now, if $\delta$ satisfies this property, then so does $\gamma \delta$ for any $\gamma \in \Gamma_T$.  Consequently, $\mu$'s can be chosen as claimed.\end{proof}
	
	We assume from now on that the $\mu$'s in $\mathcal{R}_{Q,T}$ satisfy the conclusion of Claim \ref{claim:muIntBasis}.
	
	Recall that $\mathcal{C}_T$ is a compact fundamental region for $\Lambda_{S,T}$ acting by translation on $V_{S,T}$.  For a positive number $\epsilon_n$ that only depends upon $n = \dim(V_{S,T})$, let
	\[\mathcal{S}_{B,T}(\epsilon_n) = \{Y = \left(y_1 + \frac{1}{2}|(v,v)|y_3\right) b_2 + v y_3 + y_3 b_{-2}: v \in \mathcal{C}_T, y_3 \geq \epsilon_n, y_1 \geq \epsilon_n (T,T)^{-1/2}\}.\]
	Observe that if $Y \in \mathcal{S}_{B,d}(\epsilon_n)$, then 
	\[(Y,Y) = 2 y_1 y_3 \geq \epsilon_n^2 (T,T)^{-1/2}.\]
	
	Let $G_T(\R)^{0}$ denote the identity component of $G_T(\R)$.  Set $\Gamma_{Q,T} = Q_T(\Q) \cap \Gamma_T$.
	\begin{theorem}[Classical reduction theory for orthogonal groups] There exists $\epsilon_n > 0$, independent of $T$ and only depending on $n$, so that the following statement holds: Suppose $g \in G_T(\R)^{0}$.  There is $\mu \in \mathcal{R}_Q$, $\gamma \in \Gamma_T$, and $\gamma_1 \in \Gamma_{Q,T}$ so that if $g' = \gamma_1\mu^{-1} \gamma^{-1} g$ and $g' \cdot i = X' + iY'$, then $Y' \in \mathcal{S}_{B,d}(\epsilon_n)$. 
	\end{theorem}
	\begin{proof} By Theorem \ref{thm:Schlickiwei} and Lemma \ref{lem:S21}, there are $x_1, x_2 \in \Lambda_T$ spanning a two-dimensional isotropic subspace so that $\det( \langle x_i, x_j \rangle_g) \leq C_n' (T,T)$.  Here $C_n'$ is a positive constant that only depends upon $n$.
		
		Without loss of generality we can assume that $x_1, x_2$ are an integral basis for their $\Q$-span intersect $\Lambda_T$.  We have $\Q x_1 + \Q x_2 = g_{\Q} (\Q b_1 + \Q b_2)$ for some $g_\Q \in G_T(\Q)$.  Thus we have $\Q x_1 + \Q x_2 = \gamma \mu (\Q b_1 + \Q b_2)$ for some $\gamma \in \Gamma_T$ and $\mu \in \mathcal{R}_Q$.   Intersecting with $\Lambda_T$ gives $\Z x_1 + \Z x_2 = \Z (\gamma \mu b_1) + \Z (\gamma \mu b_2)$.  Let $x_1' = \gamma \mu b_1$ and $x_2' = \gamma \mu b_2$, then 
		\[ \det( \langle x_i', x_j' \rangle_g) =\det( \langle x_i, x_j \rangle_g) \leq C_n' (T,T).\]
		Now
		\[ \langle x_i', x_j' \rangle_g = \langle \gamma \mu b_i, \gamma \mu b_j \rangle_g = \langle g^{-1}\gamma \mu b_i, g^{-1} \gamma \mu b_j \rangle_1 = \langle g_1^{-1} b_i, g_1^{-1} b_j \rangle_1\]
		where $g_1 = \mu^{-1} \gamma^{-1} g$.
		
		Let $K_{G_T}$ denote the stabilizer in $G_T(\R)^0$ of $i(b_2+b_{-2}) \in \mathcal{H}_T$.  Let $Q_T = N_{Q_T} M_{Q_T}$ denote the standard Levi decomposition of $Q_T$, so that $M_{Q_T}$ also stabilizes $\mathrm{Span}(b_{-2},b_{-1})$.  We have the Iwasawa decomposition $G_T(\R)^0 = N_{Q_T}(\R)(M_{Q_T}(\R) \cap G_T(\R)^0) K_{G_T}$.  Write $r_2: M_{Q_T} \rightarrow \GL_2$ for the homomorphism that describes the actionn of $m \in M_{Q_T}$ on $\mathrm{Span}(b_1,b_2)$.
		
		Now, write $g_1 = n m k$ in terms of this decomposition. Let $m_2 = r_2(m) \in \GL_2(\R)$. In fact, because $m\in M_{Q_T}(\R) \cap G_T(\R)^0$ and $V_T$ has Witt rank two, the matrix $m_2$ has positive determinant.  One has
		\[ \det( \langle g_1^{-1} b_i, g_1^{-1}b_j \rangle_1) =\det( \langle m^{-1} b_i, m^{-1}b_j \rangle_1) = |\det(m_2)|^{-2} \det( \langle  b_i, b_j \rangle_1) \leq C_n' (T,T).\]
		Thus $\det(m_2) \geq (C_n')^{-1} (T,T)^{-1/2}$.
		
		Because $\Gamma_{Q,T}$ contains a copy of $\SL_2(\Z) \subseteq M_{Q_T}(\Q)$, there is $\gamma' \in \SL_2(\Z) \subseteq \Gamma_{Q,T}$ and $k_1' \in K_{G_T}$ so that $m = \gamma' b k_1'$, where $b \in M_{Q_T}(\R)$ has $r_2(b) = \mm{t_1}{xt_2}{}{t_2}$ with $t_1, t_2 > 0$, $t_1/t_2 > \frac{\sqrt{3}}{2}$.   Moreover, there exists $\gamma'' \in \Gamma_{Q,T}$ so that $(\gamma'' n) b_{-2} = v + b_{-2}$ with $v \in \mathcal{C}_T$.
		
		Thus, there is $\gamma_1 \in \Gamma_{Q,T}$ so that if $g' = \gamma_1 g_1 = \gamma_1 \mu^{-1} \gamma^{-1} g$, then $g' = n' m' k'$ with $n' \in N_{Q_T}(\R)$ satisfying $n' b_{-2} \in b_{-2} + \mathcal{C}_T$, $m' \in M_{Q_T}(\R)$ having $r_2(b) =   \mm{t_1}{xt_2}{}{t_2}$ with $t_1, t_2 > 0$, $t_1/t_2 \geq \frac{\sqrt{3}}{2}$  and $t_1 t_2 \geq  (C_n')^{-1} (T,T)^{-1/2}$, and $k' \in K_{G_T}$.
		
		Now one applies $g'$ to $i(b_2 + b_{-2}) \in \mathcal{H}_T$ to obtain $X' + Y' i$ with 
		\[Y'  = \left(t_1 t_2 + t_2^{-1}t_1 \frac{|(v,v)|}{2}\right)b_{2} +  t_2^{-1} t_1 v  +t_2^{-1} t_1 b_{-2}\] 
		for some $v \in \mathcal{C}_T$.  The theorem follows.
	\end{proof}
	
	Let $K_{T,p}$ be the open compact subgroup of $G_{T}(\Q_p)$ stabilizing $\Lambda_T \otimes \Z_p$ and set $K_{T,f} = \prod_{p}{K_{T,p}} \subseteq G_{T}(\A_f)$.  Set $G_T(\Q)^{+} = G_T(\Q) \cap G_T(\R)^{0}$.
	\begin{lemma} One has $G_T(\A_f) = G_T(\Q)^{+} K_{T,f}$.
	\end{lemma}
	\begin{proof} For this proof only, let $V = \Lambda_T \otimes \Q$.  The idea is to reduce from $G_T$ to $\Spin(V)$, and apply strong approximation to the spinor group.
		
		First note that if $t \in \mathbf{G}_m$, then $r(t):=t b_1 b_{-1} + b_{-1} b_1 \in \mathrm{Clif}^{+}(V)$ has $r(t) r(t)^* = t$ and $r(t) \in \GSpin(V)$.  On $V$, $r(t)$ acts as $\diag(t,1,1,1,t^{-1})$.  It follows, in particular, that the spinor norm on $K_p$ fills up all of (the image of) $\Z_p^\times$ inside of $\Q_p^\times/(\Q_p^\times)^2$.
		
		Now suppose $g \in G(\A_f)$.  Then $g_p \in K_p$ for almost every $p$.  By the remarks above, we see that there exists $\gamma \in G_T(\Q)^{+}$ and $k \in K_{T,f}$ so that $h:=\gamma g k$ has spinor norm $1$.  Indeed, $\gamma$ and $k$ can be taken to be of the form $diag(t,1,1,1,t^{-1})$.  It follows that there exists $h' \in \Spin(V)(\A_f)$ so that $h' \mapsto h$ under the surjection $\GSpin(V) \rightarrow G=\SO(V)$.  By strong approximation for the Spin group, there exists $\gamma' \in G_T(\Q)$ and $k' \in K_f$ so that $h = \gamma' k'$.  The result follows.
	\end{proof}
	
	Set 
	\[\mathcal{S}_{Siegel,2} = \{g' \in G_d(\R)^{0}: g' \cdot i = X' + iY', Y' \in \mathcal{S}_{B,d} \}.\]
	Let $\mathcal{S}_{T,2} =  \bigcup_{\mu \in \mathcal{R}_Q} \Gamma_{Q,T} \mu^{-1} K_{T,f}$.  Note that $\mathcal{S}_{T,2}$ is compact and a finite union of $K_{T,f}$ cosets.
	\begin{corollary}[Adelic reduction theory for orthogonal groups] Suppose $g \in G_T(\A)$.  Then there is $\alpha \in G_T(\Q)$ so that $\alpha g  \in \mathcal{S}_{Siegel,2} \mathcal{S}_{T,2}$.
	\end{corollary}
	\begin{proof}  Let $g = g_\infty g_f$.  There is $\alpha_1 \in G_T(\Q)$ so that $\alpha_1 g = g_1  k$ with $k \in K_{T,f}$ and $g_1 \in G_T(\R)^{0}$, because $G_T(\A_f) = G_T(\Q)^{+} K_{T,f}$.  Now $g_1 = \gamma_1^{\infty} \mu^{\infty} \gamma_2^{\infty} g'$ with $g' \in \mathcal{S}_{Siegel,2}$, $\gamma_1^{\infty} \in \Gamma_T$, $\gamma_2^\infty \in \Gamma_{Q,T}$ and $\mu \in \mathcal{R}_{Q,T}$.  Thus $g_1 = (\gamma_1 \mu \gamma_2)_{\Q} (\gamma_1 \mu \gamma_2)_f^{-1} g'$.  The corollary is proved.
	\end{proof}
	
	We will use the following lemma in our proof of the Quantitative Sturm Bound.
	\begin{lemma}\label{lem:Cn''} Suppose $Y \in \mathcal{S}_{B,T}(\epsilon_n)$.  Then there is a positive constant $C_n''$, depending only on $n$ and not on $T$, so that
		\[(Y,b_2+b_{-2}) \leq C_n'' (T,T)^{5/2} (Y,Y).\]
	\end{lemma}
	\begin{proof} Let $Y = \left(y_1 + \frac{1}{2}|(v,v)|y_3\right) b_2 + v y_3 + y_3 b_{-2}$.  Then 
		\[(Y,b_2+b_{-2}) = y_1 + (1 + \frac{1}{2}|(v,v)|) y_3\]
		with $y_1 \geq \epsilon_n (T,T)^{-1/2}$, $y_3 \geq \epsilon_n$ and $v \in \mathcal{C}_T$.   Temporarily let $A = \frac{1}{2}\epsilon_n (T,T)^{-1/2}$.  Then $y_1 \geq 2A$ and (trivially) $(1 + \frac{1}{2}|(v,v)|) y_3 \geq 2A$.
		
		If $\alpha, \beta$ are real numbers, each at least $2A$, then
		\[1 \leq (A^{-1}\alpha-1)(A^{-1}\beta - 1)\]
		implies $\alpha+\beta \leq A^{-1}\alpha \beta$.  Applying this inequality for $\alpha = y_1$, $\beta = (1 + \frac{1}{2}|(v,v)|) y_3 $ gives
		\[y_1 + (1 + \frac{1}{2}|(v,v)|) y_3 \leq \frac{1}{A} (1 + \frac{1}{2}|(v,v)|) y_1 y_3 \leq 2(1+M_T) \epsilon_n^{-1} (T,T)^{1/2} y_1 y_3.\]
		The lemma now follows from Proposition \ref{prop:MT}.
	\end{proof}
	
	\subsection{Preparation for Sturm bound}
	The primary purpose of this subsection is to prove the following proposition, which will be used to help prove the quantitative Sturm bound.  Recall $\Lambda_T^1 = \Z b_{2} \oplus \Lambda_{S,T} \oplus \Z b_{-2}$. 
	\begin{proposition}\label{prop:SturmHelper} Suppose $M \geq 1$ is a positive integer, and $X > 0$ is a real number.  Let $Y \in S_{B,T}(\epsilon_n)$.  There is a positive constant $D_{n,S}'$, depending on $n$ and $\Lambda_S$ but not on $T$, so that the number of $\lambda \in M^{-1} (\Lambda_T^1)^\vee$ with $(\lambda,Y) \leq X$ is bounded above by $D_{n,S}' (T,T)^{(7n+10)/2} (MX)^{n+2}$.
	\end{proposition}
	
	We break the proof of Proposition \ref{prop:SturmHelper} into several lemmas.
	
	\begin{lemma}\label{lem:inequality1} If $\lambda, Y \in V_{T}^1$ with $(\lambda,\lambda) > 0$ and $(Y,Y) > 0$ then $|(\lambda,y)| \geq (\lambda,\lambda)^{1/2} (Y,Y)^{1/2}$.
	\end{lemma}
	\begin{proof} Because $V_{T}^1$ has signature $(1,n+1)$, the restriction of $q$ to the span of $\lambda, Y$ must have signature $(1,1)$ or be degenerate.  Consequently, $(\lambda,\lambda)(y,y) - (\lambda,y)^2 \leq 0$.
	\end{proof}
	
	\begin{lemma}\label{lem:Yy0} Suppose $Y, y_0 \in V_{T}^1$ satisfy $Y > 0, y_0 > 0$.  Let $\epsilon_{Y,y_0} = \frac{(Y,Y)}{2(Y,y_0)}$, which is positive.  Then $Y  > \epsilon_{Y,y_0}y_0$.
	\end{lemma}
	\begin{proof} We have 
		\[(Y-\epsilon_{Y,y_0}y_0, y_0) = \frac{1}{2(Y,y_0)}(2(Y,y_0)^2 - (Y,Y)(y_0,y_0)) \geq \frac{(Y,y_0)}{2} > 0\]
		and
		\[ (Y-\epsilon_{Y,y_0}y_0, Y-\epsilon_{Y,y_0}y_0) = (Y,Y) - (Y,Y) + \epsilon_{Y,y_0}^2(y_0,y_0) > 0.\]
	\end{proof}

	\begin{lemma}\label{lem:1T}  Suppose $Y \in \mathcal{S}_{B,T}(\epsilon_n)$ and $N >0$ is a real number.  If $\lambda \in V_{T}^1$, $\lambda >0$ and $(\lambda,Y) \leq N$, then 
		\[(b_2+b_{-2},\lambda) \leq 2 C_n'' (T,T)^{5/2}N.\]
	\end{lemma}
	\begin{proof} Let $1_T = b_2 + b_{-2}$.  Let $\epsilon_{Y,1} = \frac{(Y,Y)}{2 (Y,1_T)}$.  By Lemma \ref{lem:Yy0}, one has $Y > \epsilon_{Y,1} 1_T$.  Thus $(\lambda,Y) \leq N$ implies $\epsilon_{Y,1} (1_T,\lambda) < (Y,\lambda) \leq N$ so $(1_T,\lambda) < 2(Y,1_T) (Y,Y)^{-1} N$.  Because $Y \in \mathcal{S}_{B,T}(\epsilon_n)$, by Lemma \ref{lem:Cn''}, $2(Y,1_T) (Y,Y)^{-1} \leq 2 C_n'' (T,T)^{5/2}$.  Thus $(1_T,\lambda) \leq 2 C_n'' (T,T)^{5/2}N$.
	\end{proof}
	
	We will use the following bound.
	\begin{lemma}\label{lem:DnR} Suppose $R >0$ is a real number.  There is a positive constant $D_n$ that depends on $n$ but is idependent of $T$, so that the number of $v \in \Lambda_{S,T}$ with $|(v,v)| \leq R$ is bounded above by $D_n R^{n/2}$, where $n = \dim(V_{S,T})$.
	\end{lemma}
	\begin{proof} Let $G_T$ be the negative of the Gram matrix of $q_T$ on $\Lambda_{S,T}$.  We choose a basis of $\Lambda_{S,T}$ so that $G_T$ is Minkowski reduced.  By the Minkowski inequality for reduced matrices (See \cite[equation (1.23)]{andrianovBook}), there is a positive contant $\gamma_n$ that only depends upon $n$ so that 
		\[G_T \geq \gamma_n \diag(g_{11}, \ldots, g_{nn}) \geq \gamma_n 1_n.\]
		Here $(g_{ij})$ are the matrix entries of $G_T$ and they are at least one because $q_T$ is integral.  
		
		Let $v \in \Z^n$.  Then $v^t G v \geq \gamma_n v^t v$.  So if $v^t G v \leq R$ then $v^t v \leq \gamma_n^{-1} R$.  The lemma follows.
	\end{proof}
	As a consequence of Lemma \ref{lem:DnR}, the number of $v \in \Lambda_{S,T}^\vee$ with $|(v,v)| \leq R$ is bounded above by $D_{n,S} (T,T)^n R^{n/2}$, for a positive constant $D_{n,S}$ that depends upon $\Lambda_{S}$ but is independent of $T$.  Indeed, if $v \in \Lambda_{S,T}^\vee$, then $(T,T) \det(\Lambda_S) v \in \Lambda_{S,T}$.  If $|(v,v)| \leq R$, then $(T,T)^2 |(v,v)| \leq (T,T)^2 R$, so we may apply Lemma \ref{lem:DnR}.
	
	\begin{lemma}\label{lem:1TX} Let $1_T = b_2+b_{-2} \in \Lambda_T$ and let $X$ be a positive real number.  The number of $\lambda \in (\Lambda_T^1)^\vee$ with $\lambda > 0$ and $(1_T,\lambda) \leq X$ is bounded above by $D_{n,S}(T,T)^n X^{n+2}$.
	\end{lemma}
	\begin{proof} We have $(\Lambda_T^1)^\vee = \Z b_2 \oplus \Lambda_{S,T}^\vee \oplus \Z b_{-2}$.  If $\lambda = a b_{2} + v + b b_{-2}$, with $v \in \Lambda_{S,T}^\vee$, then $a+b = (1_T,\lambda) \leq X$.  Hence, since $(\lambda,\lambda) > 0$, $2ab - |(v,v)| > 0$, so $|(v,v)| \leq 2ab \leq (a+b)^2 \leq X^2.$  By the remark above, the number of such $v$ is bounded above by $D_{n,S} (T,T)^n X^n$.  The lemma follows.
	\end{proof}
	
	\begin{proof}[Proof of Proposition \ref{prop:SturmHelper}]  We have $M\lambda \in (\Lambda_T^1)^\vee$, so it suffices to prove the result for $M =1$.  In this case, we have $(\lambda,Y) \leq X$ so by Lemma \ref{lem:1T}, $(1_T,\lambda) \leq 2 C_n'' (T,T)^{5/2}X$.  By Lemma \ref{lem:1TX}, the number of such $\lambda$ is bounded above by $D_{n,S}' (T,T)^{(7n+10)/2} X^{n+2}$, for $D_{n,S}' = D_{n,S} (2C_n'')^{n+2}$.
	\end{proof}
	
	\subsection{Exceptional groups I}\label{subsec:redExc1}
	In this subsection, we handle some reduction theory for the groups $M_P$.  We will need this as an input to help prove the automatic convergence theorem.
	
	Let $H^1$ denote the simply-connected cover of the derived group of $M_P$.  The group $H^1$ acts on $W_J$, preserving the similitude.  Let $K_H^1$ denote the subgroup of $H^1(\R)$ that fixes the line $\C r_0(i) = \C(1,-i,-1,i) \subseteq W_J(\C)$.  The group $K_H^1$ is a maximal compact subgroup of $H^1(\R)$.  Let $U_H$ be an open compact subgroup of $H^1(\A_f)$, and let $\Gamma_{H,U} = H^1(\Q) \cap U_H$. 
	
	We state a lemma regarding the subgroup $K_H^1$.  Recall from \cite[section 3.4]{pollackQDS} the element $S_{w_1,w_2} \in \h(J)^0$ associated to element $w_1, w_2 \in W_J$.
	\begin{lemma}\label{lem:w1K} Set $w_1 = (-1,0,1,0)$.  If $k \in K_{H}^1$, then $k \cdot S_{w_1, w_1} = k S_{w_1,w_1} k^{-1} = S_{w_1,w_1}$.
	\end{lemma}
	\begin{proof} Let $w_1' = (0,1,0,-1)$.  We have $-r_0(i) = w_1 + i w_1'$.  From $S_{r_0(i),r_0(i)} = 0$, because $r_0(i)$ is rank one, we obtain $S_{w_1,w_1} = S_{w_1',w_1'}$ and $S_{w_1,w_1'} = 0$.  Now if $k \in K_{H}^1$, then $k r_0(i) = (a+ib) r_0(i)$, for $a+ib \in S^1$.  The lemma now follows by direct calculation.
	\end{proof}
	
	There is an map $\Sp_6 \rightarrow M_P$ and thus $\Sp_6 \rightarrow H_J^1$.  See \cite[section 2]{pollackETF} for our choice of this map.  Let $T_{\Sp_6}$ be the diagonal torus of $\Sp_6$ and $B_{\Sp_6}$ the standard Borel of $\Sp_6$.  The choice of $T_{\Sp_6}$ and $B_{\Sp_6}$, with the map $\Sp_6 \rightarrow H_J^1$, endows $H_J^1$ with a $C_3$ root system and a choice of positive roots.  Let $B_{H}$ be the associated minimal parabolic of $H_J^1$.  If $\epsilon > 0$, let $T_{\Sp_6}(\epsilon)$ be the set of $t \in T_{\Sp_6}(\R)$ so that $|\alpha(t)| \geq \epsilon_U$ for every positive simple root $\alpha$ for $T_{\Sp_6}$ with respect to $B_{\Sp_6}$. The general reduction theory of Borel and Harish-Chandra has the following implication.
	
	\begin{theorem}\label{thm:ExcRed} There is a finite set $\mathcal{R}_{H,U} \subseteq H^1(\Q)$, a positive constant $\epsilon_U$, and compact subset $\mathcal{C}_{B,U} \subseteq B_H(\R)$, all that may depend upon $U$, so that if $g \in H^1(\R)$, then $g = \gamma \gamma_j c t k$ where
		\begin{enumerate}
			\item $\gamma \in \Gamma_U$;
			\item $\gamma_j \in \mathcal{R}_{H,U}$;
			\item $c \in \mathcal{C}_{B,U}$;
			\item $t \in T_{\Sp_6}(\epsilon_U)$, with ;
			\item $k \in K_H^1$.
		\end{enumerate}
	\end{theorem}
	
	Recall the subspace $V_5 \subseteq V_7$, defined to be $V_5 = \mathrm{Span}(b_2, H_2(C), b_{-2})$, so $V_5 = V_7^{[1]}$.  Define a map $T_R: W_J \rightarrow V_5 \subseteq V_7$ as $\langle w, v \rangle = (T_R(w),v)_{V_7}$ for $v \in V_5 = W_J \cap V_7$.  Note that, if $w = (a,b,c,d) \in W_J(\R)$, then $q_{V_7}(T_R(w)) = (b^\#-ac)_{11}$, the $(11)$ component of $b^\#-ac$.  Consequently, if $w,w' \in W_J(\R)$ and $S_{w,w} = S_{w',w'}$, then $q_{V_7}(T_R(w)) = q_{V_7}(T_R(w'))$.
	
	Let $\mathrm{pr}_{V_7}: W_J \rightarrow V_5 \subseteq V_7$ be the projection to $V_5 \subseteq V_7$ along the decomposition $W_J = Lie(M_R)^{[1]} \oplus V_8^{[1]} \oplus V_7^{[1]}$.
	
	We will use Theorem \ref{thm:ExcRed} in conjuction with the following lemma.  
	\begin{lemma}\label{lem:RedLemExc} Let the notation be as in Theorem \ref{thm:ExcRed}.  There is a positive constant $M_U$ with the following property: Suppose $g' = ctk$ with $c \in \mathcal{C}_{B,U}$, $t\in T_{\Sp_6}(\epsilon_U)$ and $k \in K_H^1$, and $w' = g' \cdot (0,1,0,-1)$.  Then $|q_{V_7}(T_R(w')) \cdot q_{V_7}(\mathrm{pr}_{V_7}(w'))| \leq M_U$.
	\end{lemma}
	\begin{proof} Let $w_1' = (0,1,0,-1)$ and
		\[w'' = t \cdot (0,1,0,-1) = (0, \diag(t_1 t_2^{-1} t_3^{-1}, t_2 t_3^{-1} t_1^{-1}, t_3 t_1^{-1} t_2^{-1}, 0, t_1t_2 t_3).\]
		By Lemma \ref{lem:w1K},
		\[S_{t k w_1', tk w_1'} = S_{t w_1', t w_1'}	= n_{L}(\diag(t_1^{-2},t_2^{-2},t_3^{-2})) + n_{L}^\vee(\diag(t_1^2,t_2^2,t_3^2)).\] 
		
		Now, observe that, for general $w = (a,b,c,d)\in W_J(\R)$, if 
		\[S_{w,w} = (m, x,\gamma) \in \m(J) \oplus J \oplus J^\vee \simeq \h(J)^0,\]
		then $q_{V_7}(T_R(w)) = (b^\#-ac)_{11} = x_{11}$ and $q_{V_7}(\mathrm{pr}_{V_7}(w)) = (c^\#-db)_{11} = \gamma_{11}$.  Here the subscript $(11)$ denotes the $(11)$ component of the element of $H_3(C)$.  
		
		Let $R_H^1 = H^1 \cap R$ denote its Klingen parabolic subgroup.  We have a character $\lambda: R \rightarrow \GL_1$ satisfying $(r v_1, r v_2)_{V_7} = \lambda(r) (v_1, v_2)_{V_7}$ for all $v_1, v_2 \in V_7$.  Note that the modulus character $\delta_{R_H^1}$ of $R_H^1 \subseteq H$ satisfies $\delta_{R_H^1}(r) = |\lambda(r)|^{n_R}$ for some positive number $n_R$.  Also note that, if $r \in R_H^1$, then $T_R(r w) = \lambda(r)^{-1} r \cdot T_R(w)$, so $q_{V_7}(T_R(rw)) = \lambda(r)^{-1} q_{V_7}(T_R(w))$.  If $t \in T_{\Sp_6}$, then $\lambda(t) = t_1^2$.   Thus
		\begin{equation}\label{eqn:Boundt1} q_{V_7}(T_R(w')) = q_{V_7}(T_R(c t w_1')) = \lambda(c t)^{-1} q_{V_7}(w_1') = \lambda(c)^{-1} t_1^{-2}.\end{equation}
		Here we have used, in the first equality, that $S_{ct k w_1', ctk w_1'} = S_{ct w_1', ct w_1'}$ and that $q_{V_7}(T_R(w))$ can be read off from $S_{w,w}$.
		
		For ease of notation, let $L = S_{tk w_1', tk w_1'} \in \mathrm{Lie}(H^1(\R)) = \h_J^0$.  Let $B_\theta$ be the positive-definite quadratic form on $\h_J^0$ from \cite[section 3.4.5]{pollackQDS}.  In the notation of \cite[section 3]{pollackQDS}, we have $B_{\theta}((m,x,\gamma), (m,x,\gamma)) \geq (x,\iota(x)) + (\gamma, \iota(\gamma))$. 
		
		Let $|| \cdot ||$ denote an operator norm on $H^1(\R)$ so that $B_\theta( g \cdot Y, g \cdot Y) \leq ||g||^2 B_\theta(Y,Y)$ for all $g \in H^1(\R)$ and $Y \in \h_J^0$.  We have
		\[ q_{V_7}(\mathrm{pr}(w'))^2 \leq B_{\theta}(c L, cL) \leq ||c||^2 B_{\theta}(L,L) = ||c||^2 (t_1^4+t_2^4+t_3^4 + t_1^{-4} + t_2^{-4} + t_3^{-4}) \leq M_1 t_1^4\]
		for some positive constant $M_1$, using that $\mathcal{C}_U$ is compact and $t \in T_{\Sp_6}(\epsilon_U)$.  Thus $|q_{V_7}(\mathrm{pr}(w'))|$ is bounded by $t_1^2$.  Conbined with the bound of inequality \eqref{eqn:Boundt1}, the lemma is proved.
	\end{proof}
	
	As a corollary of Theorem \ref{thm:ExcRed} and Lemma \ref{lem:RedLemExc}, we obtain:
	\begin{corollary}\label{cor:ExcRed} Suppose $w \in W_J(\R)$ is positive-definite, i.e., $w > 0$.  Let $\Gamma_U$ and $\mathcal{R}_{H,U}$ be as in Theorem \ref{thm:ExcRed}.  Then there is a positive constant $M_U$, so that the following holds: there exist $\gamma \in \Gamma_U$ and $\gamma_j \in \mathcal{R}_{H,U}$ so that if $w' = w \cdot (\gamma \gamma_j)$, then 
		\[|q_{V_7}(T_R(w')) \cdot q_{V_7}(\mathrm{pr}_{V_7}(w'))| \leq M_U |q(w)|.\]
	\end{corollary}
	\begin{proof} We have $w = |q(w)| (g \cdot w_1')$ for some $g \in H^1(\R)$.  We then apply Theorem \ref{thm:ExcRed} and Lemma \ref{lem:RedLemExc}.
	\end{proof}

	\subsection{Exceptional and orthogonal groups II}\label{subsec:redExc2}
	The purpose of this section is to give a second type of reduction theory for the elements of $\Lambda_0 = \Z \oplus J_0 \oplus J_0 \oplus \Z \subseteq W_J(\Q)$.  If $x = x_{11} e_{11} + V(0,x_2, x_3) + x'$ with $x' \in H_2(C)$, let $Res_J(x) = x_{11} e_{11} + V(0,x_2, x_3) + x'$.  If $w \in W_J$, $w = (a,b,c,d)$, let $Res_W(w) = (a,Res_J(b), Res_J(c),d)$.
	\begin{theorem}\label{thm:SL2reductionCubes}Let $\Gamma_U \subseteq H^1(\Q)$ be a fixed arithmetic subgroup.  There is a finite set $\mathcal{T}_{U} \subseteq H^1(\Q)$, and a positive constant $Y_U > 0$ with the following property: Suppose $w \in \Lambda_0 \subseteq W_J(\Q)$.  Then there is $\gamma \in \Gamma_U$ and $\delta \in \mathcal{T}_U$ so that if $w_1 = w \cdot \gamma \delta$, then $w_1= (0,b_1, c_1, d_1)$ with $b_1 = b_{11}e_{11} + V(0,u_2,u_3) + T$ with $T \in H_2(C)$  and $|b_{11}| \leq Y_U \mathrm{cont}(T;\Lambda_0)^{-1} |q(Res_W(w_1))|^{1/2}$.
	\end{theorem}
	
	\begin{remark} Note that Theorem \ref{thm:SL2reductionCubes} holds for any $w \in \Lambda_0 \otimes \Q$.  Indeed, this follows from the theorem for $w \in \Lambda_0$, because both sides of the inequality 
		\[|b_{11}| \leq Y_U \mathrm{cont}(T;\Lambda_0)^{-1} |q(Res_W(w_1))|^{1/2}\]
	scale in the same way if one multiples $w$ by a positive integer.
	\end{remark}
	
	Theorem \ref{thm:SL2reductionCubes} follows from the following more general result entirely on orthogonal groups.
	\begin{proposition}\label{prop:bvRed} Suppose $V$ is a rational quadratic space with Witt rank two, and signature $(2,n)$, and $\Lambda = \Z b_2 \oplus \Z b_3 \oplus \Lambda_C \oplus \Z b_{-3} \oplus \Z b_{-2}$ is an integral lattice in $V$.  There is a constant $Y_{\Lambda} >0$ so that the following holds: Suppose $T,v \in \Lambda$ span a positive-definite two-plane in $V$.  Then, there is an isotropic $b \in \Lambda$ with $(b,T) = 0$ and $|(b,v)| \leq Y \mathrm{cont}(T;\Lambda)^{-1} |Q(T,v)|^{1/2}$, where
		\[Q(T,v) = \det(S(T,v)) = \det\left(\mb{(T,T)}{(T,v)}{(T,v)}{(v,v)}\right).\]
	\end{proposition}
	\begin{proof}
		Consider the projection $pr: V_\R \rightarrow \mathrm{Span}_{\R}(T,v) =: V_{+}$.  Let $g$ be the majorant of $(\,,\,)_V$ with respect to $V_+$. Thus $\langle u_1, u_2 \rangle_g = 2 (pr(u_1),pr(u_2)) - (u_1,u_2)$.  For $u \in V$, one has $pr(u) = \alpha T + \beta v$, where 
		\[(\alpha,\beta)^t = S(T,v)^{-1} ((u,T), (u,v))^t.\]
		One obtains
		\[(pr(u_1),pr(u_2)) = ((u_1,T),(u_1,v)) S(T,v)^{-1} ((u_2,T),(u_2,v))^t.\]
		Suppose now $x_1, x_2 \in \Lambda$ span an isotropic two-plane in $V$.  Then, on the one hand, $(\langle x_i, x_j \rangle_g) = B S(T,v)^{-1} B^t$ where $B = \mb{(x_1, T)}{(x_1,v)}{(x_2, T)}{(x_2,v)}$, so $\det((\langle x_i, x_j \rangle_g)) = B^2 Q(T,v)^{-1}$.  On the other hand, set $b = \mathrm{cont}(T,\Lambda)^{-1}( (x_2, T) x_1 - (x_1, T) x_2)$.  Then $b \in \Lambda$, $b$ is isotropic, and $(b,T) = 0$.  We have $(b,v) = \det(B)$.
		
		By Theorem \ref{thm:Schlickiwei} and Lemma \ref{lem:S21}, there is $Y_{\Lambda}$, independent of $T,v$, so that $\det((\langle x_i, x_j \rangle_g)) \leq Y_{\Lambda}$.  We obtain $(b,v)^2  = \det(B)^2 \leq Y_{\Lambda} Q(T,v)$.  This gives the proposition.
	\end{proof}
	
	Theorem \ref{thm:SL2reductionCubes} follows from Proposition \ref{prop:bvRed} and the following lemma.   Let $\Lambda = \Z^2 \oplus \Lambda_C \oplus \Z^2$ and $V = \Lambda \otimes \Q$.  We assume $q_0: \Lambda_C \rightarrow \Z$ is a negative definite quadratic form, and define $q: \Lambda \rightarrow \Z$ as $q(a_1,a_2,\lambda,d_2,d_1) = a_1 d_1 + a_2 d_2 +q_0(\lambda)$.
	
	\begin{definition} Say that $\Lambda_C$ is \emph{maximal} for $q_0$ if the following condition is satisfied: $\Lambda_1 \supseteq \Lambda_C$ a lattice in $\Lambda_C \otimes \Q$ and $q(\lambda) \in \Z$ for all $\lambda \in \Lambda_1$ implies $\Lambda_1 = \Lambda_C$.
	\end{definition}
	
	Let $G_V$ denote the algebraic group $\SO(V,q)$.
	\begin{lemma} Suppose $\Lambda_C$ is maximal for $q_0$, and set $\Gamma = G_V(\Q) \cap \GL(\Lambda)$.  Then $\Gamma$ acts transitively on the primitive isotropic vectors of $\Lambda$.\end{lemma}
	\begin{proof} Suppose $e = (a_1, a_2, \lambda_0,d_2, d_1) \in \Lambda$ is primitive and isotropic.  Let $m = gcd(a_1,a_2,d_2,d_2)$.  Because $e$ is isotropic, $q(\lambda_0)$ is divisible by $m^2$.   Thus, if $p|m$, there exits $\lambda_p \in \Lambda_C$ so that $(\lambda_p,\lambda_0)$ is not divisible by $p$.  Indeed, if not, then $q_0$ would be integral on $\Lambda_C + \Z \frac{\lambda_0}{p}$, contradicting either the maximality of $\Lambda_C$ or the primitivity of $e$.  If $m = p_1^{a_1} \cdots p_r^{a_r}$ is its prime factorization, we see that 
		\[gcd(m, (\lambda_{p_1},\lambda_0), \ldots, (\lambda_{p_r},\lambda_0)) = 1.\]
		It follows that there exists $\lambda \in \Lambda_C$ so that $(\lambda, \lambda_0)$ is relatively prime to $m$.
		
		Now, by using the $\SL_2(\Z) \times \SL_2(\Z)$ inside of $\Gamma$ which acts trivially on $\Lambda_C$, we may assume $(a_1,a_2,d_2,d_1) = (m,0,0,mr)$ for some integer $r$.  Applying an appropriate unipotent transformation in $\Gamma$, we obtain $e' = (m,(\lambda,\lambda_0),\lambda_0,0,mr)$.  Using the $\SL_2(\Z) \times \SL_2(\Z)$ action again, we can move $e'$ to $e'' = (1,0,\lambda_0,0,d)$ for some integer $d$.  Applying another unipotent element 
		$\Gamma$ gives $(1,0,0,0,0)$.  This proves the $\Gamma$ action is transitive on primitive isotropic elements of $\Lambda$.
	\end{proof}
	
	\section{Quantitative Sturm Bound}
	In this section, we prove two quantitative Sturm bounds, first for $\SL_2$ and then for certain groups of type $\SO(2,n)$. 
	
	\subsection{The group $\SL_2$}
	We start with the following lemma.  Let 
	\[\mathcal{S}_{\SL_2,Siegel} = \{g \in \SL_2(\R): g \cdot i = x + i y \text{ with } y \geq \sqrt{3}/2\}.\]
	\begin{lemma}\label{lem:redSL2} (Adelic reduction theory for $\SL_2$) Given $g \in \SL_2(\A)$, there is $\gamma \in \SL_2(\Q)$ and $k \in \SL_2(\widehat{\Z})$ so that $g=\gamma g_1 k$ with $g_1 \in \mathcal{S}_{\SL_2,Siegel}$.
	\end{lemma}
	\begin{proof} Write $g = g_f g_\infty$.  By strong approximation, there is $\gamma_1 \in \SL_2(\Q)$ so that $g_f = \gamma_{1,f} k$, so $g = \gamma_1 (\gamma_{1,\infty})^{-1} g_\infty k$.  Let $g_2 =  (\gamma_{1,\infty})^{-1} g_\infty$.  By the well-known fundamental domain for $\SL_2(\Z)$ on the upper-half complex plane, there is $\gamma_2 \in \SL_2(\Z)$ so that $g_2 = \gamma_{2,\infty} g_1$ with $g_1 \in \mathcal{S}_{\SL_2,Siegel}$.  Thus $g = (\gamma_1 \gamma_2) g_1 ((\gamma_{2,f})^{-1} k)$ is the desired decomposition.
	\end{proof}
	Replacing $\SL_2(\A)$ with $\widetilde{\SL_2}(\A)$, $\SL_2(\widehat{\Z})$ with its inverse image in $\widetilde{\SL_2}(\A_f)$, and $\mathcal{S}_{\SL_2,Siegel}$ with its inverse image in $\widetilde{\SL_2}(\R)$, we get an identical statement of reduction theory on $\widetilde{\SL_2}(\A)$.
	
	We now state and prove a quantitative Sturm bound on $\SL_2$.  Let $\widetilde{K} \subseteq \widetilde{\SL_2}(\A_f)$ be the inverse image of $\SL_2(\widehat{\Z})$.
	\begin{theorem}[Quantitative Sturm bound for $\SL_2$] \label{thm:SturmSL2} Suppose $\varphi$ is a cuspidal automorphic form on $\SL_2(\A)$ or $\widetilde{\SL_2}(\A)$, that corresponds to a holomorphic modular form of weight $\ell' \in 2^{-1} \Z$.  Assume $\varphi(g)$ has a Fourier expansion of the form 
		\[\varphi(g) = \sum_{d \in \Q_{> 0}} b_{d}(g_f) W_{\ell',\SL_2}(g_\infty).\]
		Suppose that $M \in \Z_{\geq 1}$ is a positive integer with the property that $b_d(k) \neq 0$ for $k \in \widetilde{K}$ implies $d \in M^{-1} \Z$.   Let $\beta_{d}(g_f) = d^{-\ell'/2} b_d(g_f)$ be the normalized Fourier coefficients.  There are positive constants $A_{\ell'}, B_{\ell'}$, that only depend upon $\ell'$, so that the following holds: Assume $|\beta_d(k)| \leq \epsilon$ for all $d <  R := \log(M)/\pi + A_{\ell'}$ and all $k \in \widetilde{K}$.  Then $|\beta_d(g_f)| \leq \epsilon B_{\ell'} M$ for all $d$ and all $g_f$.
	\end{theorem}
	\begin{proof} As $\varphi$ is cuspidal, $|\varphi(g)|$ achieves its maximum, which we denote by $L$, $|\varphi(g_{*})| = L$.  By Lemma \ref{lem:redSL2}, we can assume $g_{*}= g_1 k$ with $g_1 \in \mathcal{S}_{\SL_2,Siegel}$ and $k \in \widetilde{K}$.  One has $|\beta_d(g_f)| \leq e^{2\pi} L$ for all $d, g_f$.  Let $g_1 \cdot i = x + iy$, so $y \geq \sqrt{3}/2$.  Then
		\[
		L = |\varphi(g_{*})| \leq \epsilon\left(\sum_{d \in M^{-1}\Z, d > 0}{ (dy)^{\ell'/2} e^{-2\pi d y}}\right) + e^{2\pi}L\left(\sum_{d \in M^{-1}\Z, d \geq R}{ (dy)^{\ell'/2} e^{-2\pi d y}}\right).\]
		By Lemma \ref{lem:expBound}, $(dy)^{\ell'/2} e^{-\pi dy} \leq C_{\ell'}$, for some constant $C_{\ell'}$ that only depends on $\ell'$.  Thus
		\[
		\sum_{d \in M^{-1}\Z, d > 0}{ (dy)^{\ell'/2} e^{-2\pi d y}} \leq C_{\ell'} \sum_{n \geq 1}{e^{-\pi n/M}} = \frac{C_{\ell'}}{e^{\pi/M}-1} \leq C_{\ell'} \frac{M}{\pi}.\]
		Similarly, 
		\[
		\sum_{d \in M^{-1}\Z, d \geq R}{ (dy)^{\ell'/2} e^{-2\pi d y}} \leq C_{\ell'} \sum_{n \geq MR}{e^{-\pi n/M}} = C_{\ell'} \frac{e^{-\pi R}}{1-e^{-\pi/M}} \leq C_{\ell'} e^{-\pi R}(1+ M/\pi).\]
		Thus
		\[L \leq \epsilon C_{\ell'} M + 2 e^{2\pi} L C_{\ell'} M e^{-\pi R}.\]
		The theorem now follows by rearranging the inequality.
	\end{proof}
	
	\subsection{Orthogonal groups} Recall from section \ref{sec:reduction} the lattice $\Lambda_T$.  Let $V_T(\Q) = \Lambda_T \otimes \Q$; there is an associated special orthogonal group $G_T =\SO(V_T(\Q))$.  We let $M' \rightarrow G_T$ be an isogenous algebraic group with the property $M'(\R)$ preserves the connected symmetric space $\mathcal{H}_T$.  Equivalently, we assume that the image of $M'(\R)$ in $G_T(\R)$ lands in the identity component.  We will apply the results of this section to the group $M'$, as defined in section \ref{sec:splittings}, hence the overload in notation should not cause confusion.
	
	The group $M'$ supports automorphic forms that correspond to holomorphic modular forms on the symmetric space $\mathcal{H}_T$.  Suppose $F$ is such an automorphic form, corresponding to a holomorphic modular form of weight $\ell_1$.  The Fourier coefficients of $F$ are parametrized by elements $\lambda \in M^{-1} (\Lambda_T^1)^\vee$ for some positive integer $M$.  The quantitative Sturm bound says that if all the normalized Fourier coefficients 
	\[\beta_F(\lambda):=q(\lambda)^{-\ell_1/2}a_F(\lambda)\]
	of $F$ with $q(\lambda)$ small are bounded by some nonnegative constant $\epsilon$, then every $\beta_F(\lambda)$ is bounded by some explicit constant, proportional to $\epsilon$.  (The case $\epsilon=0$ would then be a classical Sturm bound.)
	
	To prove the Sturm bound, we will use the results on reduction theory for orthogonal groups in section \ref{sec:reduction}.  To review, we assume that $\mathcal{S}_{T}$ is a compact open subset of $M'(\A_f)$ so that $M'(\A) = \mathcal{S}_{Siegel} \mathcal{S}_T$, and where 
	\[\mathcal{S}_{Siegel} = \{g \in M'(\R): g \cdot (i1_T) = X'+iY', Y' \in \mathcal{S}_{B,T}\}.\]

	Here is the Sturm Bound. 
	\begin{theorem}[Quantitative Sturm Bound] \label{thm:SturmBound} There are positive costants $E_1, \alpha, d$, independent of $T$ and $M$ so that the following statement holds.  Suppose $M'(\A) = M'(\Q) \mathcal{S}_{Siegel} \mathcal{S}_T$, with $\mathcal{S}_T$ a compact open subset of $M'(\A_f)$.	Let  $F$ be a cuspidal automorphic form on $M'$ corresponding to a holomorphic modular form of weight $\ell_1 \geq 0$, and let 
		\[F(g) = \sum_{\lambda > 0}{a_F(\lambda)(g_f) \mathcal{W}_{\ell_1,\lambda}(g_\infty)}\]
		be its Fourier expansion.  Suppose $M \geq 1$ is a positive integer so that $a_F(\lambda)(s) \neq 0$ for $s \in \mathcal{S}_T$ implies $\lambda \in M^{-1} (\Lambda_T^1)^\vee$.  Denote $\beta_F(\lambda)(g_f) = q(\lambda)^{-\ell_1/2}a_F(\lambda)(g_f)$, the normalized Fourier coefficients.  Suppose $|\beta_F(\lambda)(s)| \leq \epsilon$ for all $s\in \mathcal{S}_T$ and all $\lambda \in M^{-1} (\Lambda_T^1)^\vee$ with $\lambda > 0$ and 
		
		\begin{equation}\label{eqn:ineqLambda}(\lambda, \lambda)^{1/2} \leq A_0:=2 \pi^{-1} \epsilon_n^{-1}(T,T)^{1/4} \log(E_1 M^d (T,T)^{\alpha}).\end{equation}
		Then 
		\[|\beta_F(\lambda)(g_f)| \leq \epsilon E_1 (T,T)^{\alpha} M^d\]
		for all $\lambda \in (\Lambda_T^1)^\vee \otimes \Q$ and all $g_f \in M'(\A_f)$.
	\end{theorem}
	
	\begin{remark}	The key feature of this result is that the dependence of the right-hand side of inequality \ref{eqn:ineqLambda} on $T$ is relatively explicit, and the exponent $1/4$ on the $(T,T)$ in inequality \eqref{eqn:ineqLambda} is relatively small.  In particular, this exponent is strictly less that $1/2$, which is the exponent one might produce with more naive reduction theory. \end{remark}
	
	\begin{proof}[Proof of Theorem \ref{thm:SturmBound}]
		The proof is simple given the reduction theory we have already developed.  We break the proof into a couple claims.
		
		\begin{claim}\label{claim:TotalSum} Let $\ell_1 \geq 0$ be an integer.  Suppose $M \geq 1$ is a positive integer.  There are positive constants $E,\alpha,d$, independent of $T$, but dependent on $\ell_1$, so that 
			\begin{align}\label{eqn:SumToBound} TotSum(Y,M,\ell_1)&:=\sum_{\lambda > 0, \lambda \in M^{-1}(\Lambda_T^1)^\vee}{q(\lambda)^{\ell_1/2} q(Y)^{\ell_1/2} e^{-2\pi (\lambda, Y)}} \\ & \nonumber \leq E(T,T)^{\alpha} M^d\end{align}
			for all $Y \in \mathcal{S}_{B,T}$.
		\end{claim}
		\begin{proof} By Lemma \ref{lem:inequality1}, the sum in question is less than or equal to
			\[\sum_{\lambda > 0, \lambda \in M^{-1}(\Lambda_T^1)^\vee}{(\lambda,Y)^{\ell_1} e^{-2\pi (\lambda, Y)}}.\]
			By Lemma \ref{lem:expBound}, this is bounded by a constant (only depending on $\ell_1$) times 
			\[\sum_{\lambda > 0, \lambda \in M^{-1}(\Lambda_T^1)^\vee}{e^{-\pi (\lambda, Y)}}.\]
			Applying Proposition \ref{prop:SturmHelper}, we must bound $\sum_{n \geq 0}{n^d e^{-\pi n}}$, and thus are finished.
		\end{proof}
		
		We also will bound the tail of the sum in the inequality \eqref{eqn:SumToBound}.
		\begin{claim}\label{claim:TailSum} Let the notation be as in Claim \ref{claim:TotalSum}.  Suppose $Y \in \mathcal{S}_{B,T}$, so that $(Y,Y) \geq \epsilon_n^2 (T,T)^{-1/2}$.  Let $A >0$.  Then
			\begin{align} TailSum(Y,M,\ell_1,A)&:=\label{eqn:TailToBound} 
				\sum_{\lambda > 0, \lambda \in M^{-1}(\Lambda_T^1)^\vee, (\lambda,\lambda)^{1/2} \geq A}{q(\lambda)^{\ell_1/2} q(Y)^{\ell_1/2} e^{-2\pi (\lambda, Y)}} \\
				& \nonumber \leq E' M^d (T,T)^{\alpha} \exp(-\pi A \epsilon_n (T,T)^{-1/4}/2)\end{align}
			for some positive constant $E'$ that is independent of $T, M, Y,A$.
		\end{claim}
		\begin{proof} Let $B = \floor{A \epsilon_n (T,T)^{-1/4}}$.  Arguing as in the proof of Claim \ref{claim:TotalSum}, we see that the sum in \eqref{eqn:TailToBound} is bounded by a constant (independent of $T,Y,M, A$) times 
			\[M^d (T,T)^{\alpha} \sum_{n \geq B}{n^d e^{-\pi n}}.\]
			But $\sum_{n \geq B}{n^d e^{-\pi n}}$ is bounded by a constant (that only depends on $d$) times $\sum_{n \geq B}{e^{-\pi n/2}}$, which in turn is bounded by a constant times $e^{-\pi B/2}$.  This completes the proof.
		\end{proof}
		
		Now, suppose $F$ is a cuspidal modular form on $M'$, corresponding to a holomorphic modular form of weight $\ell_1$.  Then $F$ has a Fourier expansion $F(g) = \sum_{\lambda > 0}{a_F(\lambda)(g_f) \mathcal{W}_{\ell_1,\lambda}(g_\infty)}$.  Define $\beta_F(\lambda)(g_f) = q(\lambda)^{-\ell_1/2}a_F(\lambda)(g_f)$, the normalized Fourier coefficients.  If $|F(g)| \leq L$ for some $L >0$, then $|\beta_F(\lambda)(g_f)| \leq e^{4\pi} L$ for all $\lambda$, all $g_f \in M(\A_f)$.
		
		Let $g \in M'(\A)$ be such that $|F(g)|$ is at its maximum, say $L$.  We can assume $g = g_\infty s \in \mathcal{S}_{Siegel} \mathcal{S}_{T}$.  Let $g_\infty (i 1_T) = X + iY$, so that $Y \in \mathcal{S}_{B,T}$. 
		Then
		\begin{align*}L = |F(g)| &\leq \sum_{\lambda \in M^{-1}(\Lambda_T^1)^\vee}{|\beta_F(\lambda)(s)| e^{-2\pi (\lambda,Y)}}\\
			&\leq \epsilon \cdot \mathrm{TotSum}(Y,M,\ell_1) + e^{4\pi} L \cdot \mathrm{TailSum}(Y,M,\ell_1,A_0) \\
			&\leq  \epsilon (T,T)^\alpha M^d E + e^{4\pi} L M^d (T,T)^\alpha \exp(-\pi A_0 \epsilon_n (T,T)^{-1/4}/2).
		\end{align*}
		Our choice of $A_0$ implies 
		\[ e^{4\pi}  M^d (T,T)^\alpha \exp(-\pi A_0 \epsilon_n (T,T)^{-1/4}/2) \leq 1/2.\]
		Thus we have the inequality $L \leq \epsilon (T,T)^\alpha M^d E + L/2.$  Rearranging gives $e^{4\pi}L \leq \epsilon (T,T)^\alpha M^d (2e^{4\pi} E)$, which proves the theorem.
	\end{proof}
	
	\section{Automatic convergence}
	The purpose of this section is to prove the automatic convergence theorem:
	\begin{theorem}[Automatic Convergence] \label{thm:automatic} Suppose $\ell \geq 1$ is a fixed integer.  For each $w \in W_J(\Q)$, $w > 0$, let $a_w: G(\A_f) \rightarrow \C$ be a function satisfying $a_w(n g_f) = \xi_w(n) a_w(g_f)$ for all $n \in N_P(\A_f)$.  Assume
		\begin{enumerate}
			\item there is a compact open subset $U \subseteq G(\A_f)$ for which $a_w(\cdot)$ is right $U$-invariant for every $w$;
			\item the $a_w$ satisfy the $P$, $Q$, and $R$ symmetries;
		\end{enumerate}
		Then, the $a_w$ grow polynomially with $w$.
	\end{theorem}
	
	We begin with a lemma.
	\begin{lemma}\label{lem:PRBounded} Suppose the functions $a_w$ are right $U$-invariant for some compact open subgroup $U$ of $G(\A_f)$.  Fix $L > 0$.  Then the numbers $|a_w(1)|$ are bounded if $|q(w)| < L$.
	\end{lemma}
	\begin{proof} Let $\Gamma_U = M_P^1(\Q) \cap U$.  There is a $\Gamma_U$-invariant lattice $\Lambda_U$ so that if $w \notin \Lambda_U$ then $a_w(1) = 0$.  So, we only must bound the $a_w(1)$ for $w \in \Lambda_U$ and $|q(w)| < L$.  But $|a_{w \cdot \gamma}(1)| = a_w(1)$ for $\gamma \in \Gamma_U$. There are finitely many $\Gamma_U$ orbits on the elements $w \in \Lambda_U$ with $|q(w)| < L$ by \cite[Theorem 4.9]{platonovRapinchuk}.  This gives the lemma.
	\end{proof}
	
	We introduce notation we will use in the proof of Theorem \ref{thm:automatic}.  For $g \in G(\A_f)$, let 
	\[\overline{a}_w(g) = \int_{\A_f}{a_w(\exp(s v_2 \otimes e_{11})g)\,ds}.\]
	For $w \in \widetilde{\mathcal{R}}$ corresponding to some normal $T \in V_7$, $\phi \in S(X(\A_f))$, $r \in M'(\A_f)$, and $g \in G(\A_f)$, set
	\[ a_w(r;g,\phi) = \int_{X(\A_f)}{ \overline{a}_w(x r g) (\omega(r)\phi)(x)\,dx}.\]
	If $w = (a,b,c,d) \in W_J(\Q)$, we write $w = (w_1, x,w_2)$ where the $w_1, x, w_2$ are the components of $w$ in the decomposition $W_J = Lie(M_R)^{[1]} \oplus V_8^{[1]} \oplus V_7^{[1]}$.  Thus, if $w_1$ is fixed and $T_R(w) = T_R(w_1)$ is normal, then the $a_{(w_1,0,u)}(r;g,\phi)$ form the Fourier coefficients of a modular form on $M'$.
	
	Let $K_G = \prod_{p}{K_{G,p}}$ be a fixed open compact subgroup of $G(\A_f)$, which is sufficiently large in a sense to be made precise.  Specifically, we assume $K_{G,p}$ is the stabilizer of the lattice $\Lambda_\g \otimes \Z_p \subseteq \g(J) \otimes \Q_p$ for every prime $p$.  Let $K_{R} = K_G \cap M_R(\A_f)$ and $K_{P} = K_G \cap M_P(\A_f)$.  We choose once and for all a finite set of elements of $G(\A_f)$ in the following way:
	\begin{itemize}
		\item We have $G(\A_f) = \bigsqcup_{\ell} N_P(\A_f) M_P(\Q) \delta_{\ell} K_G$, for a finite set of elements $\delta_{\ell} \in G(\A_f)$.  Indeed, this follows from the fact that $G(\A_f)$ is a finite union of sets of the form $P(\A_f) g_r K_G$, $P(\A_f) = N_P(\A_f) M_P(\A_f)$, and $M_P(\A_f)$ is a finite union of sets $M_P(\Q) g'_s K_P$.
		\item We have a finite subset $\{\gamma_j\}_j = \mathcal{R}_{H,K_P} \subseteq H^1(\Q)$ so that the conclusion of Corollary \ref{cor:ExcRed} holds for this set, with $\Gamma = H^1(\Q) \cap K_P$.
		\item We have $M_R(\A_f) = \bigsqcup_{k}{ (M_R \cap P)(\Q) v_k K_{R}}$ for a finite set $\{v_k\}_k$ of $M_R(\A_f)$.
	\end{itemize}
	Set $B_G = \cup_{j,\ell}{ K_G \gamma_j^{-1} K_G \delta_\ell K_G}$ and $A_G = \cup_{k}{v_k K_R}.$ Note that $B_G$ is a compact open subset of $G(\A_f)$.   
	
	For a positive number $D$, let $V(D)$ be a function of $D$, to be made explicit below.  Let
	\[N(D) = \{n \in Lie(M_R)^{[1]}(\A_f): \exists r \in \Z: rn \in \Lambda_0, r \leq V(D)\}.
	\]
	
	\subsection{Prepatory lemmas}
	We will need numerous lemmas to prove the Automatic Convergence theorem.  We begin by bounding the numbers $\overline{a}_w(g)$ and the coefficients $a_{(T,0,u)}(r,g,\phi)$.
	
	\begin{lemma}\label{lem:boundawbar} Suppose $g \in G(\A_f)$ is fixed.  Let $\Lambda_g \subseteq W_J(\Q)$ be a lattice so that $a_w(g) \neq 0$ implies $w \in \Lambda_g$.  Let $\mathrm{cont}(T;\Lambda_g)$ be the content of $(e_{11} \otimes v_{2}) \cdot (T,0,0)$ with respect to $\Lambda_g$.  Suppose $B_{T,D,g} > 0$ is a real number such that if $w$ is of the form $w=(T,x,u)$ and $|q(w)| = D$ then $|a_w(g)|  \leq B_{T,D,g}$.   Then, for $w$ of this form, $|\overline{a}_w(g)| \leq B_{T,D,g} \cdot \mathrm{cont}(T;\Lambda_g)$.
	\end{lemma}
	\begin{proof} We have
		\[ \overline{a}_{w}(g) = \int_{\A_f}{a_{w}(\exp(s e_{11} \otimes v_2) g)\,ds}.
		\]
		If the integrand is $0$ for all $s \in \A_f$, there is nothing to prove.  So, we can assume without loss of generality that the integrand is nonzero for $s =0$, in which case $w \in \Lambda_g$.  We see that then, for $s \in \Q$, $a_{w}(\exp(s e_{11} \otimes v_2) g) \neq 0$ implies $s \in \mathrm{cont}(T;\Lambda_g)^{-1} \Z$.  Rewriting the integral as a sum of the integrand evaluated at values $s \in \Q$, we get
		\[\overline{a}_w(g) = \frac{1}{M} \sum_{s \in c^{-1} \Z/M\Z}{ a_{w \cdot \exp(s e_{11} \otimes v_2)}(g)}\]
		where $c = \mathrm{cont}(T;\Lambda_g)$ and $M$ is a sufficiently large positive integer.  Each term in the sum is bounded by $B_{T,D,g}$, so the lemma follows.
	\end{proof}
	
	\begin{lemma}\label{lem:volBound1} Let $U_{r,g} \subseteq X(\A_f)$ be an open compact subset such that $\overline{a}_{(T,0,u)}(xrg) \neq 0$ impliies $x \in U_{r,g}$.  Let $B$ be a positive real number so that $|\overline{a}_{(T,0,u)}(xrg)| \leq B$ for all $x \in X(\A_f)$.  Then $|a_{(T,0,u)}(r,g,\phi)| \leq ||\phi||_{L^2} \cdot B \cdot \mathrm{vol}(U_{r,g})^{1/2}$.
	\end{lemma}
	\begin{proof}
		We have
		\[a_{(T,0,u)}(r,g,\phi) = \int_{X(\A_f)}{\overline{a}_{(T,0,u)}(xrg) (\omega(r)\phi)(x)\,dx}.\]
		Consequently,
		\[ |a_{(T,0,u)}(r,g,\phi)| \leq  B \cdot \int_{U_{r,g}}{|\omega(r)\phi(x)| \,dx}.\]
		By Cauchy-Schwarz,
		\[ \int_{U_{r,g}}{|\omega(r)\phi(x)| \,dx} \leq \left(\int_{U_{r,g}}{|\omega(r)\phi(x)|^2\,dx}\right)^{1/2} \left(\int_{U_{r,g}}{1\,dx}\right)^{1/2} \leq ||\omega(r)\phi||_{L^2} \cdot \mathrm{vol}(U_{r,g})^{1/2}.\]
		The Weil representation preserves the $L^2$ norm, so the lemma follows.
	\end{proof}
	
	Conversely, we can bound the $\overline{a}_{(T,0,u) \cdot x}(g)$ in terms of the $a_{(T,0,u)}(r,g,\phi)$.  To do this, one starts with the identity
	\begin{equation}\label{eqn:abaraFJ}\overline{a}_{(T,0,u)}(xrg) = \sum_{\alpha}{ a(r,g,\phi_\alpha^\vee) (\omega(r)\phi_\alpha)(x)}.\end{equation}
	Here $\{\phi_\alpha\}$ is a basis of $S(X(\A_f))$ and $\phi_\alpha^\vee$ is the dual basis.
	\begin{lemma}\label{lem:abarfromFJ} Suppose $B_{T,u,g}'> 0$ is a constant so that $|a_{(T,0,u)}(1,g,\phi)| \leq B_{T,u,g}' \cdot ||\phi||_{L^2}$ for all $\phi \in S(X(\A_f))$.  Suppose $V_{T,u,g} \subseteq X(\A_f)$ is a compact open subgroup with the property that $\overline{a}_{(T,0,u)}(x v g)  = \overline{a}_{(T,0,u)}(xg)$ if $v \in V_{T,u,g}$ and $x \in X(\A_f)$.  Then $|\overline{a}_{(T,0,u)}(xg)| \leq B_{T,u,g}' \cdot \mathrm{vol}(V_{T,u,g})^{-1/2}$.
	\end{lemma}
	\begin{proof}  Let $V' \supseteq V_{T,u,g}$ be a compact open subgroup of $X(\A_f)$, and let $S(X(\A_f))_{V', V_{T,u,g}}$ be the functions $\phi$ that are supported in $V'$ and satisfy $\phi(x+v) = \phi(x)$ for all $v \in V_{T,u,g}$.  The space $S(X(\A_f))_{V', V_{T,u,g}}$ is finite-dimensional.  Let $\phi_\alpha$ be the basis of characteristic functions for the cosets of $V_{T,u,g}$ in $V'$.  We have $\phi_\alpha^\vee = \mathrm{vol}(V_{T,u,g})^{-1}\phi_\alpha$, and $||\phi_\alpha^\vee|| = \mathrm{vol}(V_{T,u,g})^{-1/2}$.  Consequently,
		\[ \sum_{\alpha}{ |a(1,g,\phi_\alpha^\vee)| \cdot |\phi_\alpha(x)|} \leq B_{T,u,g}' \mathrm{vol}(V_{T,u,g})^{-1/2},\]
		as only one of the terms $|\phi_\alpha(x)|$ can be nonzero for a fixed $x$.   The assumption of the lemma that $\overline{a}_{(T,0,u)}(x v g)  = \overline{a}_{(T,0,u)}(xg)$ if $v \in V_{T,u,g}$ and $x \in X(\A_f)$ implies that, for $V'$ sufficiently large, the $\phi_\alpha$ are an acceptable set of functions to use in the right-hand side of equation \eqref{eqn:abaraFJ}.  This gives the lemma.
	\end{proof}
	
	We can also bound the $a_w(g)$ in terms of the values $\overline{a}_w(ng)$ with $n \in M_R^{[1]}(\A_f)$.  This uses the identity	
	\begin{equation}\label{eqn:awg} a_w(g) = \int_{\mathrm{Stab}_T\backslash M_R^{[1]}(\A_f) \simeq \A_f}{\xi_w^{-1}(n) \overline{a}_{w}(n g)\,dn}.\end{equation}
	(The integral is convergent, because the integrand is $0$ outside of a compact set of the domain of integration.)  We identify $\mathrm{Stab}_T\backslash M_R^{[1]}(\A_f) \simeq \A_f$ via the map $\exp(v) \mapsto \langle T, [e_{11} \otimes v_2, v] \rangle$; this puts a measure $dn$ on $\mathrm{Stab}_T\backslash M_R^{[1]}(\A_f)$.
	\begin{lemma}\label{lem:awbaraw} Suppose $g \in G(\A_f)$ is fixed, and $M_g \in \Z$ is such that $a_w(sug) = a_w(sg)$ if $u \in \exp(M_g \widehat{\Z} v_2 \otimes e_{11})$.  Let
		\[V_{T,g} = \{v \in Lie(M_R)^{[1]}: \psi(M_g \widehat{\Z} \langle T, [e_{11} \otimes v_2, v]\rangle) = 1\}.\]
		Then $\overline{a}_w(\exp(v)g) \neq 0$ implies $v \in V_{T,g}$.  Moreover, the identity \eqref{eqn:awg} holds.  In particular, if $B''_{w,g} \geq 0$ is such that $|\overline{a}_w(\exp(v)g)| \leq B''_{w,g}$ for all $v \in Lie(M_R)^{[1]}$, then $|a_w(g)| \leq M_{g} B''_{w,g}$.
	\end{lemma}
	\begin{proof} We have 
		\[\overline{a}_w(ng) = \int_{\A_f}{a_w(sng)\,ds} = \int_{\A_f}{\xi_w(s \cdot n) a_w(sg)\,ds}.\]
		Suppose $w = (T,x,u) \in W_J(\Q)$ and $n = \exp(v)$, $v \in Lie(M_{R})^{[1]}$.  Then 
		\[\xi_w(s \cdot n)  = \psi(\langle w, (v,0, 0) + s (e_{11} \otimes v_2)(v,0,0) \rangle) = \psi(\langle u, v \rangle)\psi(s\langle T, [e_{11} \otimes v_2, v] \rangle).\]
		Thus $\xi_w(s \cdot n) = \xi_w(n) \psi(s \langle T, [e_{11} \otimes v_2, v] \rangle)$.  Suppose $a_w(sg) = a_w(s u g)$ if $u \in \exp(M_g \widehat{\Z} e_{11} \otimes v_2)$.  Then, for $\overline{a}_{w}(ng)$ to be nonzero, we would need $\psi(M_g \widehat{\Z} \langle T, v' \rangle) = 1$, where $v' = [e_{11} \otimes v_2, v]$. 
		
		Let $\overline{V}_{T,g} = V_{T,g}/\mathrm{Stab}_T(\A_f)$.
		Then
		\begin{align*}
			\int_{\mathrm{Stab}_T\backslash M_R^{[1]}(\A_f) \simeq \A_f}{\xi_w^{-1}(n) \overline{a}_{w}(n g)\,dn} &= \int_{\overline{V}_{T,g}}{\xi_w^{-1}(n) \overline{a}_{w}(n g)\,dn} \\
			&=\int_{\overline{V}_{T,g}}\int_{\A_f} \psi(s \langle T, v' \rangle) a_w(sg)\,ds\,dv \\
			&= \frac{1}{M_g} \sum_{s \in A^{-1} \Z/M_g \Z} a_w(sg) \left( \int_{\overline{V}_{T,g}}{ \psi(s \langle T,v' \rangle)\,dv}\right).
		\end{align*}
		This gives $a_w(g)$, as claimed.
	\end{proof}
	
	\begin{corollary}\label{cor:Mg} Let the notation be as in Lemma \ref{lem:awbaraw}.  Suppose $w = (T,x,u)$, $v \in Lie(M_R)^{[1]}$ is such that $\langle T, v' \rangle = 1$.  Suppose moreover that $|\overline{a}_w(\exp( a M_g^{-1} v) g)| \leq B''_{w,g}$ for all $a \in \widehat{\Z}$.  Then $|a_w(g)| \leq M_{g} B''_{w,g}$.
	\end{corollary}
	\begin{proof} If $v_1 \in V_{T,g}$, $a:=M_g \langle T, v_1' \rangle \widehat{\Z}$, so $\langle T, v_1' - a M_g^{-1} v \rangle = 0$.  Thus 
		\[\overline{a}_w(\exp(v_1)g) = \overline{a}_w(\exp(a M_g^{-1} v)g),\]
		and the corollary follows from Lemma \ref{lem:awbaraw}.
	\end{proof}
	
	Suppose $n \in N(D)$.  We now wish find a lattice $\Lambda_{n} \subseteq W_J(\Q)$ so that $a_w(g) \neq 0$ and $g \in U_1 n V_1$ for fixed open compact $U_1, V_1$ implies $w \in \Lambda_n$.  We begin very simply:
	\begin{lemma}\label{lem:simpleLattice} Suppose the $a_w(g)$ are right-invariant by an open compact subgroup $U$ of $G(\A_f)$.  
		\begin{enumerate}
			\item If $n \in N_P(\A_f)$, and there exists $z \in Z(\A_f)$ (the center of $N_P)(\A_f)$) so that $zn \in gUg^{-1}$, then $a_w(g) = a_w(ng) = \xi_w(n) a_w(g)$.  Thus, if $a_w(g) \neq 0$, $\xi_w(n) = 1$.  
			\item If $X \subseteq G(\A_f)$ is a set, let $U_X = \{u \in G(\A_f): x^{-1}ux \in U \forall x \in X\}.$  Then $a_w(x'u x) = a_w(x'x)$ for all $u \in U_X$ and $x \in X$, arbitrary $x' \in G(\A_f)$.
			\item Let $N_X = N_P(\A_f) \cap (Z(\A_f) \cdot U_X)$.   Then $a_w(x) \neq 0$ for $x \in X$ implies $\langle w, n \rangle \in \widehat{\Z}$ for all $n \in N_X$.
		\end{enumerate}
	\end{lemma}
	
	Let $\Lambda_{\g} \subseteq \g(\Q)$ be our fixed lattice.  If $R = p_1^{n_1} \cdots p_k^{n_k}$ is positive integer, we denote
	\[\exp(R \Lambda_\g(\widehat{\Z})) = \prod_{p} K_{G,p,R}\]
	where $K_{G,p,R} = K_{G,p}$ if $p \nmid R$ and $K_{G,p,R}  = \exp( p_i^{n_i} \Lambda_\g(\Z_p))$ if $p = p_i$ divides $R$.
	\begin{lemma} Let the notation be as in Lemma \ref{lem:simpleLattice}.  Suppose $R_U > 0$ is a positive integer so that $U \supseteq \exp(R_U \Lambda_{\g}(\widehat{\Z}))$.  Let $X \subseteq G(\A_f)$ be a set, and suppose $B_X \in \Z_{\geq 1}$ satisfies $B_X \cdot Ad(x)^{-1}(\lambda) \in \Lambda_{\g}(\widehat{\Z})$ for all $\lambda \in \Lambda_{\g}(\widehat{\Z})$ and all $x \in X$.  Then $U_X \supseteq \exp(B_X R_U \Lambda_{\g}(\widehat{\Z})).$
	\end{lemma}
	\begin{proof} The proof is immediate from the definitions.
	\end{proof}
	
	\begin{lemma}\label{lem:U1V1Rn} Suppose $U_1, V_1$ are compact subsets of $G(\A_f)$, and $A_1, B_1 \in \Z_{\geq 1}$ satisfy 
		\begin{align*}
			A_1 Ad(U_1)^{-1} \Lambda_\g(\widehat{\Z}) &\subseteq \Lambda_{\g}(\widehat{\Z})\\
			B_1 Ad(V_1)^{-1} \Lambda_\g(\widehat{\Z}) &\subseteq \Lambda_{\g}(\widehat{\Z}).
		\end{align*}
		If $n \in N(D)$, let $X_n = U_1 n V_1$. 
		\begin{enumerate}
			\item 	Let $R_n \in \Z_{\geq 1}$ satisfy $R_n Ad(n)^{-1} \Lambda_\g(\widehat{\Z}) \subseteq \Lambda_{\g}(\widehat{\Z})$.  Then $(A_1 R_n B_1) Ad(X)^{-1} \Lambda_\g(\widehat{\Z}) \subseteq \Lambda_{\g}(\widehat{\Z})$.
			\item Suppose $n = \exp(r^{-1}v)$ with $v \in Lie(M_R)^{[1]} \cap \Lambda_\g$.  For $R_n$, one can take $M_1 r^{T_1}$, for a positive integers $T_1, M_1$ that only depend upon $G$.
			\item One has $U_{X_n} \supseteq \exp(A_1 B_1 R_n R_U \Lambda_{\g}(\widehat{\Z}))$.
		\end{enumerate}	
	\end{lemma}
	\begin{proof} The first part of the lemma is clear.  For the second part, observe that the elements of $N(D)$ are all unipotent.  Thus, there is an integer $T_1$ for which $Ad(\log(n))^{j} = 0$ if $j > T_1$, for all $n \in N(D)$.  The third part is also clear.
	\end{proof}
	
	The following corollary follows directly from the above work.
	\begin{corollary}\label{cor:Lambdag} Let the notation be as above, with $X_n = U_1 nV_1$.  Let $\Lambda_0^\vee = (\Lambda_\g \cap W_J(\Q))^\vee$.  If $a_w(x) \neq 0$ for some $x \in X_n$, then $w \in (A_1 B_1 R_n R_U)^{-1} \Lambda_0^\vee$.
	\end{corollary}
	
	Recall that $\{\delta_\ell\}_\ell$ is a finite set so that $G(\A_f) = \bigsqcup_{\ell} N_P(\A_f) M_P(\Q) \delta_\ell K_G$.  Let $V_2 = \cup_{\ell} \delta_\ell K_G$.  Recall also that $\Lambda_\g \subseteq \g(\Q)$ is a lattice.  We assume that $K_G$ fixes $\Lambda_\g(\widehat{\Z})$.  Let $U_1, V_1$ be fixed open subsets of $G(\A_f)$.  If $n \in N(D)$, we now wish to bound $a_w(U_1 n V_1 V_2)$ given bounds on $a_w(V_2)$.
	
	We have $\Lambda_\g(\A_f) = \prod_{v < \infty}' \Lambda_\g \otimes \Q_v$, a restricted product.  Say an element $\lambda \in \Lambda_\g(\Z_p)$ is \emph{primitive} if $n \in \Z_{\geq 0}$ and $p^{-n} \lambda \in \Lambda_\g(\Z_p)$ implies $n =0$.  We define a norm on $\Lambda_\g \otimes \Q_p$ as $||\lambda||_p = |p^n|_p$ if $\lambda = p^{n} \lambda_0$ with $\lambda_0$ primitive.  The norm $||\cdot||_p$ is $K_{G,p}$-invariant for every $p < \infty$.  If $\lambda \in \Lambda_\g(\A_f)$, we write $||\lambda||_p := ||\lambda_p||_p$, where $\lambda_p$ is the $p$-component of $\lambda$, and set $||\lambda||_f = \prod_{p}||\lambda||_p$.  Contrary to the notation, $||\cdot ||_f$ is not a norm on $\Lambda_\g(\A_f)$.
	
	Let $X_n = U_1 n V_1$, if $n \in N(D)$.  Observe that if $x \in X$, then $x \delta_\ell k_1 = nm\delta_j k_2$ for some $n \in N_P(\A_f)$, $m \in M_P(\Q)$, $k_2 \in K_G$.  Consequently, 
	\[ ||(x \delta_\ell k_1)^{-1} E_{13}||_p = || k_2^{-1} \delta_j^{-1} m^{-1} n^{-1} E_{13}||_p = |\nu(m)|_{p}^{-1} \cdot ||\delta_j^{-1} E_{13}||_p.\]
	Thus,
	\[|(x \delta_\ell k_1)^{-1} E_{13}||_f  = |\nu(m)|_\infty \cdot ||\delta_j^{-1} E_{13}||_f.\]
	Therefore, we can read off bounds on $|\nu(m)|_\infty$ if we can bound $|| (x \delta_\ell k_1)^{-1} E_{13}||_f$.  
	
	For a finite place $p$, and $g \in G(\Q_p)$, let $||g||_p$ be the operator norm with respect to our norm on $\Lambda_\g$.  That is, $||g||_p = \sup_{\lambda \in \Lambda_\g(\Z_p)}{||g \lambda||_p}$.  Let $E > 0$ be such that $\prod_p{||\delta_\ell||_p} \leq E$ and $\prod_{p}{||\delta_\ell^{-1}||_p} \leq E$ for every $\ell$.  Note that $E$ exists, because $(\delta_\ell)_p \in K_{G,p}$ for almost every $p$.  Moreover, 
	\[E^{-1}\cdot  ||\lambda_0||_f \leq ||\delta_\ell^{-1} \lambda_0||_f \leq E \cdot||\lambda_0||_f\]
	for every $\lambda_0 \in \Lambda_\g(\A_f)$.
	
	Putting things together, we have
	\[
	E^{-1}\cdot ||x^{-1}E_{13}||_f \leq |(x \delta_\ell k_1)^{-1} E_{13}||_f  \leq E \cdot ||x^{-1} E_{13}||_f
	\]
	and
	\[E^{-1} \cdot |\nu(m)|_\infty \leq |(x \delta_\ell k_1)^{-1} E_{13}||_f \leq E \cdot |\nu(m)|_\infty. \]
	Thus 
	\[
	E^{-2} \cdot ||x^{-1} E_{13}|| \leq |\nu(m)|_\infty \leq E^2 \cdot ||x^{-1} E_{13}||_f.
	\]
	
	Summarizing:
	\begin{lemma}  Let $E > 0$ be such that $\prod_p{||\delta_\ell||_p} \leq E$ and $\prod_{p}{||\delta_\ell^{-1}||_p} \leq E$ for every $\ell$.  If $x \delta_\ell k_1 = nm\delta_j k_2$ for some $n \in N_P(\A_f)$, $m \in M_P(\Q)$, $k_2 \in K_G$, then 
		\[
		E^{-2} \cdot ||x^{-1} E_{13}|| \leq |\nu(m)|_\infty \leq E^2 \cdot ||x^{-1} E_{13}||_f.
		\]
	\end{lemma}
	
	Continuing, we have:
	\begin{lemma}\label{lem:continuing} Suppose $E_{N(D)}$ is such that $\prod_p{||n||_p} \leq E$ and $\prod_{p}{||n^{-1}||_p} \leq E_{N(D)}$ for every $n \in N(D)$ and every $p$.  There is a constant $E_1$, independent of $D$, so that if $x \in U_1 N(D) V_1$, and $x \delta_\ell k_1 = nm\delta_j k_2$ for some $n \in N_P(\A_f)$, $m \in M_P(\Q)$, $k_2 \in K_G$, then 
		\[
		E_1^{-1} E_{N(D)}^{-1} \leq |\nu(m)|_\infty \leq E_1 E_{N(D)}.\]
		Moreover, one can take $E_{N(D)} =M_1 V(D)^{T_1}$, for some absolute constants $M_1, T_1$.
	\end{lemma}
	\begin{proof} The first part follows as above.  For the second, if $n \in N(D)$, then $n = \exp(r^{-1} v)$, $v \in Lie(M_R)^{[1]} \cap \Lambda_\g$ and $r \in \Z$ with $r \leq V(D)$.  If $\lambda \in \Lambda_g$, then 
		\[n \cdot \lambda = \sum_{0 \leq j \leq T_1}{\frac{r^{-j}}{j!} ad(v)^j(\lambda)}.\]
		Thus $M_1 r^{T_1} (n \cdot \lambda) \in \Lambda_\g$.  It follows that $||(M_1 r^{T_1}) n \cdot \lambda||_f \leq 1$ so $||n \cdot \lambda||_f \leq M_1 r^{T_1}$.  As the set $N(D)$ is closed under taking inverses, the lemma follows.
	\end{proof}
	
	We now prepare some lemmas to understand the Fourier-Jacobi expansion along the $Q$-parabolic.   If $\phi \in S(J(\A_f))$, $B \in J(\Q)$, $d \in \Q$, $r \in \widetilde{\SL_2}(\A_f)$ and $g \in G(\A_f)$, we write
	\[b_{(B,d)}(r,g;\phi) = \int_{J(\A_f)}{ \omega_{\psi_B}(r)\phi(x)a_{(0,B,0,d)}(xrg)\,dx}.
	\]
	That the $\{a_w\}_w$ satisfy the $Q$-symmetries mean that the $b_{(B,d)}(r,g,\phi)$ are the Fourier coefficients of a holomorphic modular form of weight $\ell' = \ell+1 - \dim(J)/2$ on $\widetilde{\SL_2}$ as $d$ varies.
	
	\begin{lemma} \label{lem:volBound2} Suppose $B \in J(\Q)$ is positive-definite, and $C$ satisfies $|a_{(0,B,c,d)}(rg)| \leq C$ for all $w = (0,B,c,d)$ with $|q(w)| = D$.  Let $d' = \frac{D}{4n(B)}$.  Suppose $V_{r,g} \subseteq J(\A_f)$ is an open compact set so that $a_{(0,B,0,d')}(xrg) \neq 0$ implies $x \in V_{r,g}$. Then $|b_{(B,d')}(r,g;\phi)| \leq ||\phi||_{L^2} \cdot C \cdot \mathrm{vol}(V_{r,g})^{1/2} $.
	\end{lemma}
	\begin{proof} We have
		\[
		|b_{(B,d)}(r,g;\phi)| \leq C \cdot \int_{V_{r,g}}{ |\omega_{\psi_B}(r)\phi(x)| \,dx}.
		\]
		The lemma follows by Cauchy-Schwarz, as in the proof of Lemma \ref{lem:volBound1}.
	\end{proof}
	
	We can also bound the $a_{w}(g)$ in terms of the $b_{B,d}(r,g,\phi)$.  The details are very similar to the proof of Lemma \ref{lem:abarfromFJ}. One starts with
	\begin{equation}\label{eqn:afromFJ}a_{(0,B,0,d)}(xrg) = \sum_{\alpha}{ b_{B,d}(r,g,\phi_\alpha^\vee) (\omega(r)\phi_\alpha)(x)}.\end{equation}
	Here $\{\phi_\alpha\}$ is a basis of $S(J(\A_f))$ and $\phi_\alpha^\vee$ is the dual basis.
	\begin{lemma}\label{lem:afromFJQ} Suppose $C_{B,d,g}'> 0$ is a constant so that $|b_{B,d}(1,g,\phi)| \leq C_{B,d,g}' \cdot ||\phi||_{L^2}$ for all $\phi \in S(X(\A_f))$.  Suppose $V_{B,d,g} \subseteq J(\A_f)$ is a compact open subgroup with the property that $a_{(0,B,0,d)}(x v g)  = a_{(0,B,0,d)}(x g) $ if $v \in V_{B,d,g}$ and $x \in J(\A_f)$.  Then $|a_{(0,B,0,d)}(xg)| \leq C_{B,d,g}' \cdot \mathrm{vol}(V_{B,d,g})^{-1/2}$.
	\end{lemma}
	\begin{proof}  Let $V' \supseteq V_{B,d,g}$ be a compact open subgroup of $J(\A_f)$, and let $S(J(\A_f))_{V', V_{B,d,g}}$ be the functions $\phi$ that are supported in $V'$ and satisfy $\phi(x+v) = \phi(x)$ for all $v \in V_{B,d,g}$.  The space $S(J(\A_f))_{V', V_{B,d,g}}$ is finite-dimensional.  Let $\phi_\alpha$ be the basis of characteristic functions for the cosets of $V_{B,d,g}$ in $V'$.  We have $\phi_\alpha^\vee = \mathrm{vol}(V_{B,d,g})^{-1}\phi_\alpha$, and $||\phi_\alpha^\vee|| = \mathrm{vol}(V_{B,d,g})^{-1/2}$.  Consequently,
		\[ \sum_{\alpha}{ |b_{B,d}(1,g,\phi_\alpha^\vee)| \cdot |\phi_\alpha(x)|} \leq C_{B,d,g}' \mathrm{vol}(V_{B,d,g})^{-1/2},\]
		as only one of the terms $|\phi_\alpha(x)|$ can be nonzero for a fixed $x$.   The assumption of the lemma that $a_{(0,B,0,d)}(x v g)  = a_{(0,B,0,d)}(x g)$ if $v \in V_{B,d,g}$ and $x \in J(\A_f)$ implies that, for $V'$ sufficiently large, the $\phi_\alpha$ are an acceptable set of functions to use in the right-hand side of equation \eqref{eqn:afromFJ}.  This gives the lemma.
	\end{proof}
	
	We now bound the volume of the sets $U_{r,g}$ and $V_{r,g}$ of Lemmas \ref{lem:volBound1} and \ref{lem:volBound2}.  We begin with a simple calculation.  Let $\mathrm{pr}_{C,2}: J=H_3(C) \rightarrow C^2$ be the projection that reads off the $x_2$ and $x_3$ components.
	\begin{lemma} \label{lem:Txsimple} Suppose $w = (a, b,c,d)$ with $b = \diag(b_{11},B_{2,3})$ and $c = \diag(c_{11},C_{2,3})$, where $b_{11}, c_{11} \in \Q$ and $B_{2,3}, C_{2,3} \in H_2(C)$.  Let $x = \Phi_{E, V(0,u_2,u_3)} + v_2 \otimes V(0,v_2,v_3)$.  Then $w \cdot \exp(x) = (a',b',c',d')$, where
		\[\mathrm{pr}_{C^2}(b') = u \times B_{2,3} + (E,B_{2,3})u + av
		\]
		and
		\[\mathrm{pr}_{C^2}(c') = -c_{11} u + B_{2,3} \times v.
		\]
	\end{lemma}
	\begin{proof} One has
		\[ \Phi_{E,u}(b) = - E \times (u \times b) + (E,B_{2,3}) u  = u \times B_{2,3} + (E,B_{2,3})u.
		\]
		Additionally,
		\[ \Phi_{E,u}(c) = u \times (E \times c) = u \times (E \times (c_{11} e_{11} + C_{2,3})) = u \times (c_{11} \times E + E \times C_{2,3}) = - c_{11} u.
		\]
		Thus $\exp(\Phi_{E,u})(b) = (*, u \times B_{2,3} + (E,B_{2,3})u,B_{2,3})$ in components for $J = H_3(C) = \Q \oplus C^2 \oplus H_2(C)$.  Similarly, $\exp(\Phi_{E,u})(c) = (c_{11},-c_{11}u,*)$.  The lemma follows by applying $\exp(v_2 \otimes v)$.
	\end{proof}
	
	To bound the volume of $U_{r=1,g}$, we will apply the computation of Lemma \ref{lem:Txsimple} and a corollary of the following lemma.
	\begin{lemma}\label{lem:V5intorb} Let $V_5(\Z) = \Z b_2 \oplus \Z b_3 \oplus \mathcal{O}_C \oplus \Z b_{-3} \oplus \Z b_{-2}$.  Let $V_5(\Z)^\vee$ be the dual lattice, so that $V_5(\Z)^\vee = \Z b_2 \oplus \Z b_3 \oplus \mathcal{O}_C^\vee \oplus \Z b_{-3} \oplus \Z b_{-2}$.  Let $\Gamma_{V_5} = \GL(V_5(\Z)) \cap \SO(V_5(\Q))$. Let $\mathcal{C} \subseteq C$ be any set such that if $v \in \mathcal{O}_C^\vee$, then there exists $x \in \mathcal{O}_C$ so that $v-x \in \mathcal{C}$.  Suppose $\lambda \in V_5(\Z)^\vee$ is primitive.  Then there is $\gamma \in \Gamma_{V_5}$ so that $\gamma \cdot \lambda = b_2+ v + s b_{-2}$, where $s \in \Z$ and $v \in \mathcal{C}$.
	\end{lemma}
	\begin{proof} If $\lambda = (p,q,v,r,s)$, let $c'(\lambda) = gcd(p,q,r,s)$.  We claim that there exists $\gamma \in \Gamma_{V_5}$ so that $c'(\gamma \cdot \lambda) = 1$.  To see this, assume $\lambda$ is such that $c'(\gamma \cdot \lambda) \geq c'(\lambda)$ for all $\gamma \in \Gamma_{V_5}$.  Write $a :=c'(\lambda)$. Then, by using $\SL_2(\Z) \times \SL_2(\Z) \in \Gamma_{V_5}$, we can assume $p = a$, $q,r=0$, and $a|s$.
		
		Now note that if $x \in \mathcal{O}_C$, then there is $n(x) \in \Gamma_{V_5}$ so that $n(x) \cdot (p,q,v,r,s) = (p,q, v + q x,r + (x,v) + q n_C(x),s)$.  Thus, $gcd(a,s,(x,v)) \geq a$ for all $x \in \mathcal{O}_C$.  We obtain $a| (x,v)$, so $c(\lambda) = a$.  But $c(\lambda) = 1$, so $a=1$.
	\end{proof}
	
	\begin{corollary}\label{cor:V5red} Let the notation be as in Lemma \ref{lem:V5intorb}. There are a finite set of elements $\{\tau_i\}_i \in \SO(V_5(\Q))$, so that if $\lambda \in V_5(\Q)$, then there is $\gamma \in \Gamma_{V_5}$ and some $\tau_i$ so that $\tau_i \gamma \lambda \in \mathrm{Span}(b_2,b_{-2})$.
	\end{corollary}
	\begin{proof} The set $\mathcal{O}_C^\vee/\mathcal{O}_C$ is finite, so we $\tau_i = n(x_i)$ where $x_i$ are representatives in $\mathcal{O}_C^\vee$ for $\mathcal{O}_C^\vee/\mathcal{O}_C$.  
	\end{proof}
	
	\begin{lemma}\label{lem:volU1} Suppose $T \in Lie(M_R)^{[1]}$ satisfies $q_{V_7}(T_R(T)) \neq 0$.  Let $U_T \subseteq X(\A_f)$ consist of those $x$ so that if $(T,0,u) \cdot \exp(x) = (T,x',u')$, then $x' \in \mathcal{O}_C^4$.  Then there is a constant $A' > 0$, independent of $T$, so that $\mathrm{vol}(U_T) \leq A' |q(T)|^{2 \dim(C)}.$
	\end{lemma}
	\begin{proof} By the reduction theory of Corollary \ref{cor:V5red}, and the invariance of the volume under $(M_R \cap M_P(\Q)) \cap K_G$, we can assume $B_{2,3} = 0$.  Then, by Lemma \ref{lem:Txsimple}, the volume of $U_T$ is $(|ac_{11}|_f^{-1})^{2\dim(C)}$.
	\end{proof}
	
	Recall that $\mathcal{S}_T \subseteq M_T'(\A_f)$ is defined as $\mathcal{S}_T = \cup_{\mu \in \mathcal{R}_{Q,T}} \Gamma_{Q,T} \mu^{-1} K_T$, where $\mathcal{R}_{Q,T}$ is a set of representatives for $\Gamma_T \backslash M'_T(\Q)/ Q_T(\Q)$, where $Q_T$ is the parabolic stabilizing $\mathrm{Span}_\Q(b_1,b_2)$.  Here $\Gamma_{Q,T} = \Gamma_T \cap Q_{T}(\Q)$. Moreover, the representatives $\mu$ are chosen so that $\mu b_1, \mu b_2$ are an integral basis of $\mathrm{Span}_\Q(b_1, b_2) \cap \Lambda_T$.
	
	We also recall that $A_G = \cup_k v_k K_R$ where $M_R(\A_f) = \bigsqcup_{k} (M_R \cap P)(\Q) v_k K_R$.
	\begin{lemma}\label{lem:Z1Z2} Suppose $\gamma \mu^{-1} k_0\in \mathcal{S}_{T}$, and $\gamma \mu^{-1} k_0 = b v_k k_1$, with $k_1 \in K_R$ and $b \in (M_R \cap P)(\Q)$.  Then
		\begin{enumerate}
			\item There is an absolute constant $Z_1 > 0$ so that $Z_1^{-1} \leq |\lambda(b)|_\infty \leq Z_1$.
			\item There is an absolute constant $Z_2 > 0$ so that $Z_2^{-1} \leq |\nu(b)|_\infty \leq Z_2$.
			\item There is an absolute constant $Z_3 > 0$ so that 
			\[Z_3^{-1} \mathrm{cont}(T;\Lambda_0) \leq \mathrm{cont}(T \cdot b,\Lambda_0) \leq Z_3 \mathrm{cont}(T;\Lambda_0).\]
		\end{enumerate}
	\end{lemma}
	\begin{proof} We have $b = \gamma \mu^{-1} k_0 k_1^{-1} v_k^{-1}$.  The left-hand side has $\lambda$ in $\Q^\times$, while the right-hand side has $\lambda$ in a fixed open compact subset of $\A_f^\times$, as $\lambda(M_T') = 1$.  This proves the first statement.  
		
		For the second statement, let $|| \cdot ||_p$ be a $p$-adic norm on $V_7 \otimes \Q_p$, defined in terms of our fixed lattice in $V_7$.  Then, for $p_1 \in (M_R \cap P)(\Q_p)$, $|| p_1 b_1||_p =|\nu(p)|_p$.  Thus $||b b_1||_f = |\nu(b)|_f = |\nu(b)|_\infty^{-1}$.  Now 
		\[ |\nu(b)|_\infty = |\nu(b)|_f^{-1} = ||b^{-1} b_1||_f = ||v_k k_1 k_0^{-1} \mu \gamma^{-1} b_1||_f. \]
		However, by our assumption on $\mu$, $\mu \gamma^{-1} b_1 \in \Lambda_\g$ is primitive, so $||\mu \gamma^{-1} b_1||_f = 1$.  This proves the second statement.
		
		For the third statement, we observe that if $T_1 \in \Lambda_\g$, $T_1 = h T_2$ with $h \in \Z_{\geq 1}$ and $T_2$ primitive, then $||T_1||_p = |h|_p ||T_2||_p = |h|_p$.  Thus $||T_1||_f = |h|_f = |h|_\infty^{-1}$.  Now, with $T_1 = T \cdot b$, we have 
		\[|h|_\infty^{-1} = ||T_1||_f = ||T \cdot (\gamma \mu^{-1} k_0 k_1^{-1} v_k^{-1})||_f = ||T k_1^{-1} v_k||_f.\]
		The latter term is bounded between $Z_3^{-1} ||T||_f$ and $Z_3 ||T||_f$.  This gives the lemma.
	\end{proof}
	
	\subsection{Proof of automatic convergence, I}
	We now give the proof of the automatic convergence theorem. 
	
	\begin{proof}[Proof of Theorem \ref{thm:automatic}]
		We begin by recalling and setting some notation.

		\textbf{Notation}
		\begin{enumerate}
			\item Recall the finite set of elements $\delta_{\ell} \in G(\A_f)$.  We set $V_2 = \cup_{\ell} \delta_\ell K_G$.
			\item The reduction theory of subsection \ref{subsec:redExc2} gives a finite set of elements $\{\sigma_r\}_r =\mathcal{T}_{K} \subseteq H^1(\Q)$.  We set $V_1'' = \cup_{r} K_H \sigma_r$.
			\item The reduction theory of subsection \ref{subsec:redExc1} gives a finite set of elements $\{\gamma_j\}_j = \mathcal{R}_{H,K} \subseteq H^1(\Q)$.  We set $V_1' = \cup_j K_G \gamma_j^{-1} K_G$.
			\item We set $V_1 = V_1' V_1''$.
			\item Let $V(D)$ be a function of $D$, as yet to be specified.  Let $N(D)$ be as defined above, in terms of $V(D)$.
			\item We have a finite set of elements $v_k \in M_R(\A_f)$.  We let $U_1 = \cup_{k} v_k K_R$.
			\item If $n \in N(D)$, we let $X_n = U_1 n V_1$.
		\end{enumerate}
		
		Let $f(D)$ be a function of $D$, as yet to be specified.  We will prove that if $v \in V_2$, $|a_w(v)| \leq Q |q(w)|^{(\ell+1)/2} f(|q(w)|)$, for some $Q >0$.  For an appropriate choice of $f(D)$, this will imply that the $a_w$ grow polynomially with $w$.
		
		Let $\delta > 1$ be a real number, as yet to be specified.  Let $D_0 > 0$ be a large positive number.  If $w \in W_J(\Q)$, then $|q(w)| \leq D_0^{\delta^n}$ 
		for some positive integer $n$.  We will prove $|a_w(v)| \leq Q f(|q(w)|)$ by induction on $n$.
		
		For any $D_0$, there is positive number $Q$ (depending on $D_0$), so that $|q(w)| \leq D_0$ implies $|a_w(v)| \leq Q$.  This follows from Lemma \ref{lem:PRBounded}.  Thus, the base case $n=1$ of the induction can be satisfied for any $D_0$ and positive, increasing function $f(D)$.
		
		We now do the inductive step.  Suppose then that $|a_w(v)| \leq Q f(|q(w)|)$ if $|q(w)| \leq D:=D_0^{\delta^n}$.  Let $E_{N(D)}  = M_1 V(D)^{T_1}$, $E_1 > 0$ from Lemma \ref{lem:continuing}.  Let $E_D = E_1 E_{N(D)}$.  Then if $n \in N(D)$, $z \in X_n$, $v \in V_2$, and $|q(w)| \leq DE_D^{-2}$, we can bound $|a_w(z v)|$.  Specifically, 
		\[ |a_w(zv)| \leq Q |q(w)|^{(\ell+1)/2} f(E_{D}^2 |q(w)|) \,\,\,\,\,\text{    if }\,\,\,\,\, |q(w)| \leq DE_D^{-2}.\]
		This follows from Lemma \ref{lem:continuing}.
		
		We next bound the $\overline{a}_{w = (T,*,*)}(zv)$ if $|q(w)| \leq D E_D^{-2}$.  We will apply Lemma \ref{lem:boundawbar}.  In the context of this lemma, with $g = zv$, we can take $\Lambda_g = (C_1 R_n)^{-1} \Lambda_0$, where $C_1 = A_1 B_1 R_U$ depends only on $U$ and $R_n = M_1 r^{T_1}$ with $r \leq V(D)$.  Here the notation and the proof of this claim follows from Corollary \ref{cor:Lambdag}.  We thus have
		\[ |\overline{a}_{w=(T,*,*)}| \leq Q |q(w)|^{(\ell+1)/2} f(E_{D}^2 |q(w)|) \mathrm{cont}(T;\Lambda_g) \,\,\,\,\,\text{    if }\,\,\,\,\, |q(w)| \leq DE_D^{-2}.\]
		Rewriting in terms of $\Lambda_0$ gives
		\[ |\overline{a}_{w=(T,*,*)}| \leq Q |q(w)|^{(\ell+1)/2} f(E_{D}^2 |q(w)|) \cdot  \mathrm{cont}(T;\Lambda_0) \cdot C_1 E_{N(D)} \,\,\,\,\,\text{    if }\,\,\,\,\, |q(w)| \leq DE_D^{-2}.\]
		
		With an eye toward applying the Quantitative Sturm bound for orthogonal groups, we now bound the $a_{w=(T,0,u)}(r,zv,\phi)$ if $T$ is normal and $|q(w)| \leq D E_D^{-2}$. Here $r \in \mathcal{S}_T$. We will apply Lemma \ref{lem:volBound1}.  We have
		\[ |a_{w= (T,0,u)}(r,zv,\phi)| \leq Q |q(w)|^{(\ell+1)/2} f(E_{D}^2 |q(w)|) \cdot  \mathrm{cont}(T;\Lambda_0) \cdot C_1 E_{N(D)} \cdot ||\phi||_{L^2} \cdot \mathrm{vol}(U_{r,g=zv})^{1/2}\]
		if $|q(w)| \leq D E_D^{-2}$.  To make this explicit, we bound $\mathrm{vol}(U_{r,g=zv})$.  Recall that $U_{r,g} \subseteq X(\A_f)$ is an open compact subset so that $\overline{a}_{(T,0,u)}(xrg) \neq 0$ implies $x \in U_{r,g}$.
		
		We first consider the case $r=1$, but $T$ not necessarily normal.  To get a handle on $U_{r=1,g=zv}$, we will use Lemma \ref{lem:volU1}. Approximating $x \in X(\A_f)$ by an element of $X(\Q)$, $\overline{a}_{(T,0,u)}(x g) \neq 0$, $g = zv$, implies $M_1 r^{T_1} x \in U_T$, in the notation of Lemma \ref{lem:volU1}.  Thus $\mathrm{vol}(U_{r=1,g=zv})$ is bounded above by $A' |E_{N(D)}^2 q(T)|^{2\dim(C)}$.  Setting $A'' = C_1 (A')^{1/2}$, we have checked that
		\begin{align}
			\nonumber |a_{w= (T,0,u)}(1,zv,\phi)| & \leq Q |q(w)|^{(\ell+1)/2} f(E_{D}^2 |q(w)|) \cdot  \mathrm{cont}(T;\Lambda_0) \\
			&\,\,\,\,\,\, \times  A''  E_{N(D)}^{1+2\dim(C)} \cdot ||\phi||_{L^2} \cdot |q(T)|^{\dim(C)} \label{eqn:IneqFJr=1} \end{align}
		if $|q(w)| \leq D E_D^{-2}$.  
		
		Now we consider the case of general $r \in \mathcal{S}_T$, but $g \in K_G n V_1 V_2$, $n \in N(D)$.  Embedding $M_T'(\A_f) \subseteq M_R(\A_f)$, we can write $r = \gamma \mu^{-1}k_0 = b v_k k_1$, in the notation of Lemma \ref{lem:Z1Z2}.  We have $b = n_b m_b$ with $n_b \in (M_R \cap N_P)(\Q)$ and $m_b \in (M_R \cap M_P)(\Q)$.  Then, if $g_1 = v_k k_1 g \in X_n$,
		\begin{align}
			\nonumber a_{w=(T,0,u)}(r,g,\phi) &= \int_{(x,s) \in X(\A_f) \times A_f(v_2 \otimes e_{11})} a_w(\exp(s) \exp(x) r g)(\omega(r)\phi)(x)\,ds\,dx \\
			\nonumber  &= \int_{(x,s) \in X(\A_f) \times A_f(v_2 \otimes e_{11})} a_w(\exp(s) \exp(x) b g_1)(\omega(r)\phi)(x)\,ds\,dx \\
			\nonumber  &= \zeta \int_{(x,s) \in X(\A_f) \times A_f(v_2 \otimes e_{11})} a_w(\exp(s) \exp(x) m_b g_1)(\omega(r)\phi)(x)\,ds\,dx \\
			\nonumber   &= \zeta \nu(m_b)^{-\ell} |\nu(m_b)|^{-1} \\
			\label{eqn:mbchange}  &\,\,\,\,\,\,\,\, \times \int_{(x,s) \in X(\A_f) \times A_f(v_2 \otimes e_{11})} a_{w \cdot m_b}(m_b^{-1} \exp(s) \exp(x) m_b g_1)(\omega(r)\phi)(x)\,ds\,dx
		\end{align}
		for some $\zeta \in S^1 \subseteq \C^\times$.
		
		We now take absolute values, change variables, and apply Lemma \ref{lem:Z1Z2}.  For some absolute constant $Z_4 >0$, we obtain
		\[
		| a_{w=(T,0,u)}(r,g,\phi)| \leq Z_4 \int_{(x,s) \in X(\A_f) \times A_f(v_2 \otimes e_{11})} |a_{w \cdot m_b}( \exp(s) \exp(x)  g_1)|\cdot |(\omega(r)\phi)(x)|\,ds\,dx.
		\]
		But the right-hand side can be bounded using \eqref{eqn:IneqFJr=1}.  We obtain, if $r \in \mathcal{S}_T$ and $g \in K_G n V_1 V_2$, $n \in N(D)$,
		\begin{align}
			\label{eqn:forQSB}	| a_{w=(T,0,u)}(r,g,\phi)|  \leq Q |q(w)|^{(\ell+1)/2} f((E_{D}Z_2)^2 |q(w)|) \cdot  \mathrm{cont}(T;\Lambda_0) \\
			\nonumber	\,\,\,\,\,\, \times  A'''  E_{N(D)}^{1+2\dim(C)} \cdot ||\phi||_{L^2} \cdot |q(T)|^{\dim(C)} \label{eqn:IneqFJr=1}
		\end{align}
		if $|q(w)| \leq D (Z_2E_D)^{-2}$.  Here $A'''$ is another constant.
		
		We now assume $T$ is normal, and use the fact that the $a_{(T,0,u)}(r,g,\phi)$ are Fourier coefficients of a cusp form on $M_T'$.  To apply the Quantitative Sturm bound for orthogonal groups, we need to bound the lattice in which $u$ can live when $a_{(T,0,u)}(r,g,\phi) \neq 0$, where $g \in K_G n V_1 V_2$, $n \in N(D)$.  By equation \eqref{eqn:mbchange}, and the argument of Lemma \ref{lem:Z1Z2} (observe $\gamma, \mu, v_k, k_0, k_1 \in M_R(\A_f)$ and thus all preserve $V_8 = V_8^{[0]} + V_8^{[1]}$), it suffices to consider the case when $r=1$ but $g \rightarrow g_1 \in A_G K_G n V_1 V_2$.  In this case, given $T\in Lie(M_R)^{[1]}$, we must bound the $u \in V_5(\Q)$ for which there exists $x \in X(\Q)$ and $s \in \Q$ so that $a_{(T,0,u) \cdot \exp(x)\exp(s v_2 \otimes e_{11})}(g_1) \neq 0$.
		
		Let $w = (T,0,u) \cdot \exp(x)\exp(s v_2 \otimes e_{11})$. By Corollary \ref{cor:Lambdag}, $M_3:=M_2 r^{T_1} w \in \Lambda_0$, for some $r \leq V(D)$ and some positive integer $M_2$, independent of all choices.  In particular, $M_3 T \in \Lambda_0$.  By Corollary \ref{cor:V5red}, $M_3[(T,0,0),x] \in \Lambda_0$ implies $x \in q(M_3 T)^{-1} \mathcal{O}_C^4$, up to a fixed absolute constant.  Thus $M_3^2 q(T) x \in \mathcal{O}_C^4$.  Therefore, $M_3^3 q(T) [[(T,0,0),x],x] \in \Lambda_0$.  We conclude $M_4 r^{3 T_1} q(T) u \in \Lambda_0$, for some fixed positive integer $M_4$ and an $r \leq V(D)$ that only depends upon $n \in N(D)$.  Thus, if $M$ is the positive integer of Theorem \ref{thm:SturmBound}, then $M \leq M_4 q(T) E_{N(D)}^3$.
		
		Applying the Quantitative Sturm bound for orthogonal groups, Theorem \ref{thm:SturmBound}, we arrive at the following fact, which we single out as a proposition.  Assume from now on that $V(D) = D^{s}$ for some small positive number $s \leq 1$, to be determined.  
		\begin{proposition}\label{prop:applySturm} There are positive constants $A,\alpha_1, \alpha'',\alpha'''$ so that the following statement holds for $D$ sufficiently large: Suppose $T$ is normal with $|q(T)|^{3/2} \leq \alpha''' \frac{D^{1-2sT_1}}{\log(D)}$, and $g \in K_G n V_1 V_2$ with $n \in N(D)$.   Then 
			\[|a_{(T,0,u_1)}(r,g,\phi)| \leq |q(u_1)|^{\ell_1/2} Q D^{\alpha''} f(D) \cdot A \cdot \mathrm{cont}(T;\Lambda_0) |q(T)|^{\alpha_1} \cdot ||\phi||_{L^2}.\]
			Here $r \in M_T'(\A_f)$.
		\end{proposition}
		\begin{proof} Suppose $|q(T)|^{3/2} \leq \alpha''' \frac{D^{1-2sT_1}}{\log(D)}$.  We have $M \leq M_4 D^{1+3s T_1} \leq M_5 D^{T_2}$.  Likewise, 
			\[E_1(T,T)^\alpha M^d \leq M_6 D^{T_3} \leq D^{T_3'}.\]
			If $|q(u)| \leq |q(T)|^{1/2} \log(E_1 M^d (T,T)^\alpha)$, then
			\[ |q(w)| = |q(T) q(u)| \leq |q(T)|^{3/2} T_3' \log(D) \leq T_3' \alpha''' D^{1-2sT_1} \leq D(Z_2 E_D)^{-2}.  \]
			Thus we can apply inequality \eqref{eqn:forQSB} to give an $\epsilon$, in the notation of Theorem \ref{thm:SturmBound}.  The proposition follows.
		\end{proof}
		
		We now will bound the $\overline{a}_{(T,0,u_1)}(xg)$ using Proposition \ref{prop:applySturm} and Lemma \ref{lem:abarfromFJ}. Let $\epsilon > 0$ be quite small.  For $D$ sufficiently large depending on $\epsilon$, we can apply Proposition \ref{prop:applySturm} whenever $|q(T)| \leq D^{2/3-\epsilon-s_1}$, where $s_1 = \frac{4}{3} s T_1$. If $|q(T)| \leq D^{2/3-\epsilon - s_1}$ and $g \in K_G n V_1 V_2$, $n \in N(D)$, let
		\[B_{T,u_1,g}' = |q(u_1)|^{\ell_1/2} Q D^{\alpha''} f(D) \cdot A \cdot \mathrm{cont}(T;\Lambda_0) |q(T)|^{\alpha_1}.\]
		
		To bound $\overline{a}_{(T,0,u_1)}(xg)$, we require a lower bound on $\mathrm{vol}(V_{T,u_1,g})$, in the notation of Lemma \ref{lem:abarfromFJ}.  We can obtain such a bound using Lemma \ref{lem:simpleLattice} and Lemma \ref{lem:U1V1Rn}.  Using that $V(D) = D^s$, we see that $\mathrm{vol}(V_{T,u_1,g})$ is bounded below by a fixed power of $D$.  We can absorb this into $B_{T,u_1,g}'$, and obtain 
		\[|\overline{a}_{(T,0,u_1)}(xg)| \leq B_{T,u_1,g}' \,\,\,\,\,\,\, \text{ if } \,\,\,\,\,\,\,\, |q(T)| \leq D^{2/3-\epsilon-s_1} \,\,\,\,\ \text{ and } \,\,\,\,\,\, g \in K_G n V_1 V_2\]
		if $T$ is normal.  Re-writing, we have that if $T$ is normal with $|q(T)| \leq D^{2/3-\epsilon-s_1}$ and $g \in K_G n V_1 V_2$, and $w_1=(T,*,*)$, then 
		\begin{equation}\label{eqn:awbarGoodBound} |\overline{a}_{w_1=(T,*,*)}(g)| \leq |q(w_1)|^{\ell_1/2} Q D^{\alpha''} f(D) \cdot A \cdot \mathrm{cont}(T;\Lambda_0) |q(T)|^{\alpha_1'}.\end{equation}
		Using the $\SL_2(\Z) \times \SL_2(\Z)$ in $K_G$, we have the same bound for $T$ not-necessarily normal.
		
		To bound the $a_{w_1}(g_2)$, $g_2 \in V_1 V_2$, we will apply Corollary \ref{cor:Mg}.  Let $M_{V_1V_2} \geq M_g$ for all $g \in V_{1}V_2$.  This can be done.   We state the result as another proposition.
		\begin{proposition}\label{prop:awGoodBound} Assume $D$ is sufficiently large, and $|a_w(g)| \leq Q  |q(w)|^{(\ell+1)/2} f(|q(w)|)$ for all $g \in V_2$.  Let $0 < s < 1$ be a small positive number, and $\epsilon > 0$ very small.  There are positive constants $A, \alpha_1''$ so that if $\mathrm{cont}(T;\Lambda_0) \leq M_{V_1V_2}^{-1} D^s$ and $|q(T)| \leq D^{2/3-\epsilon-s_1}$ then
			\[ |a_{w_1=(T,*,*)}(g)| \leq A \cdot Q \cdot |q(w_1)|^{\ell_1/2}  D^{\alpha_1''} f(D).\]
			Here $s_1 = \frac{4}{3} s T_1$.  
		\end{proposition}
		\begin{proof}
			Simply note that we have applied Corollary \ref{cor:Mg} and inequality \eqref{eqn:awbarGoodBound}.  (We have absorbed the constant $M_{V_1V_2}$ into the $A$ and the $|q(T)|^{\alpha_1'}$ into the $D^{\alpha_1''}$.)
		\end{proof}
		
		To continue on with the proof of the automatic convergence theorem, we now prove a statement similar to Proposition \ref{prop:awGoodBound}, except using the $Q$-symmetries and the $\SL_2$-quantitative Sturm bound.
		\begin{proposition}\label{prop:FJBgoodBound} Assume $D$ is sufficiently large, and $|a_w(g)| \leq Q  |q(w)|^{(\ell+1)/2} f(|q(w)|)$ for all $g \in V_2$. Suppose $w = (0,B,C,d)$ and $g \in V_1 V_2$.  Let $\epsilon > 0$ be very small.  If $|N(B)| \leq D^{1-\epsilon}$ and $g \in V_1 V_2$, then 
			\[|a(w)(g)| \leq |q(w)|^{\ell'/2} Q \beta_0' D^{\beta_1'} f(D),\]
			for some positive constants $\beta_0', \beta_1'$ that depend on $V_1 V_2$ but do not depend on $D$.
		\end{proposition}
		\begin{proof} By Lemma \ref{lem:continuing}, there are constants $M_1, M_2 > 0$ so that 
			\[|a_w(g)| \leq M_1 Q |q(w)|^{(\ell+1)/2} f(M_2 |q(w)|)\]
			if $|q(w)| \leq D/M_2$ and $g \in V_1 V_2$.  Suppose now $B \in J$ is positive-definite, and $d \in \Q^\times$ with $|N(B) d| \leq M_2^{-1}D$.   By Lemma \ref{lem:volBound2}, if $k \in \widetilde{\SL_2(\widehat{\Z})}$, then 
			\[|b_{B,d}(k,g,\phi)| \leq ||\phi||_{L^2} \mathrm{vol}(V_{B,k,g})^{1/2} \cdot Q M_1 |q(w)|^{(\ell+1)/2} f(D).\]
			To bound $V_{B,k,g}$, observe that if $B \times X = C$, then $X = \frac{1}{N(B)}(\frac{1}{2} (B,C)B - B^\# \times C)$.  Thus, there is $M_3 > 0$ so that $\mathrm{vol}(V_{B,k,g}) \leq M_3 |N(B)|^{\dim(J)}$.
			
			We prepare to apply the Quantitative Sturm bound for $\SL_2$, Theorem \ref{thm:SturmSL2}.  For the integer $M$ of the statement of this theorem, we can take $M_4 N(B)$, for some positive integer $M_4$, independent of $D$ and $N(B)$, and only depending on $V_1 V_2$.  Because $|N(B)| \leq D^{1-\epsilon}$, we can apply the Quantitative Sturm bound for $\SL_2$.  We obtain
			\[
			|b_{B,d}(r,g,\phi)| \leq ||\phi||_{L^2} |d|^{\ell'/2} Q \beta_0 D^{\beta_1} f(D),\]
			for all $r \in \widetilde{\SL_2}(\A_f)$, all $d \in \Q^\times$, and some positive constants $\beta_0, \beta_1$.  We can now apply Lemma \ref{lem:afromFJQ} to obtain
			\[|a_{(0,B,0,d)}(xg)| \leq |d|^{\ell'/2} Q \beta_0' D^{\beta_1} f(D)\]
			for all $x \in J(\A_f)$, $g \in V_1V_2$, and for some new constant $\beta_0'$.  The proposition follows.
		\end{proof}
		
		We are now ready to give the proof of the automatic convergence theorem.  Choose $s$ and sufficiently small so that $\frac{4}{3} - 2s_1 > 1$.  Let $\epsilon$ now be sufficiently small so that $1+s-\epsilon > 1$ and $\frac{4}{3} - 2s_1 - 2\epsilon > 1$.  Choose $\delta$ so that $1 < \delta < 1+s-\epsilon$ and $1 < \delta < \frac{4}{3} - 2s_1 - 2\epsilon$.
		
		Suppose now we have proved $|a_w(g)| \leq Q |q(w)|^{(\ell+1)/2} f(|q(w))$ if $g \in V_2$ and $|q(w)| \leq D = D_0^{\delta^n}$.   Let $g \in V_2$, and $w \in W_J(\Q)$ with $D \leq |q(w)| \leq D^{\delta}$.  Let $\beta_w(g) = |q(w)|^{-(\ell+1)/2} a_w(g)$ denote the normalized Fourier coefficient.   By the definition of $V_1$, there exists $\mu \in V_1 \cap M_P(\Q)$ so that $w \cdot \mu = w_1= (0, B,C,d)$ has the following properties:
		\begin{enumerate}
			\item $B = (b_{11},u,T)$ with $|q(T)| \leq M |q(w_1)|^{1/2}$
			\item $|b_{11}| \leq M \mathrm{cont}(T;\Lambda_0)^{-1} |q(w_1)|^{1/2}$.
		\end{enumerate}
		The term $\mathrm{cont}(T;\Lambda_0)^{-1}$ in the above will play a crucial role, as will be seen momentarily.
		
		We have $\beta_{w}(g) = \beta_{w_1}(\mu^{-1} g)$, and $\mu^{-1} g \in V_1 V_2$.  We consider two cases:
		\begin{enumerate}
			\item $\mathrm{cont}(T;\Lambda_0) \leq D^{s-\epsilon_1}$ and
			\item $\mathrm{cont}(T;\Lambda_0) \geq D^{s-\epsilon_1}$.  (Here $\epsilon_1$ is a tiny positive number.)
		\end{enumerate}
		Suppose we are in the first case.  We have $|q(T)| \leq M \cdot D^{\delta/2}$.  But $\frac{\delta}{2} < \frac{2}{3} - s_1 - \epsilon$, so we may apply Proposition \ref{prop:awGoodBound} to obtain
		\[|\beta_w(g)| = |\beta_{w_1}(\mu^{-1}g)| \leq Q f(D) D^{E}.\]
		(We have absorbed the constant $A$ into the exponent $E$, because $D$ is sufficiently large.)
		Conversely, suppose we are in the second case.  Then 
		\[|N(B)| \leq M^2 |q(w_1)| D^{\epsilon_1-s} \leq M^2 D^{\epsilon_1} D^{\delta-s} < D^{1-\epsilon}.\]
		Thus, we may apply Proposition \ref{prop:FJBgoodBound} to obtain
		\[|\beta_w(g)| = |\beta_{w_1}(\mu^{-1}g)| \leq Q f(D) D^{E}.\]
		We see that if $f(D) = (1 \cdot D_0 \cdot D_0^\delta \cdot \cdots D_0^{\delta^{n-1}})^E$ then the induction goes through.  Moreover,
		\begin{align*}
			f(D) &=  (1 \cdot D_0 \cdot D_0^\delta \cdot \cdots D_0^{\delta^{n-1}})^E \\
			&= D_0^{E \cdot \frac{\delta^n-1}{\delta-1}} \\
			&\leq  D^{E (\delta-1)^{-1}}.
		\end{align*}
		Thus, the $a_w$ grow polynomially with $w$, and the proof is complete.
	\end{proof}
	
	\appendix	
	
	\section{Definite integrals}
	In this section, we collect together various definite archimedean integrals needed throughout the main text.
	
	\subsection{Integral one}
	\begin{theorem}\label{thm:defInt1} Suppose $\mu, \lambda > 0$. There is a positive constant $C'$, independent of $v$ and $\mu$, but possibly depending on $\lambda$, so that 
		\[\int_{\R}{ e^{-t^2} \left(\frac{(t+\lambda^2 i)^2 -\mu}{|(t+\lambda^2 i)^2-\mu}\right)^v K_v(|(t+\lambda^2 i)^2-\mu|)\,dt}\]
		is equal to $(-1)^v C' e^{-\mu}$.
	\end{theorem}
	To prove Theorem \ref{thm:defInt1}, we will relate the case of $v$ to $v \pm 1$, and we will handle directly the case of $v=0$.  Specifically, Theorem \ref{thm:defInt1} follows from the next two propositions.  Write $I_v(\mu, \lambda;\phi) = \int_{\R}{\phi(x) \left(\frac{(x+ \lambda^2 i)^2-\mu}{|(x+\lambda^2 i)^2-\mu|}\right)^v K_v(|(x+\lambda^2 i)^2-\mu|)\,dx}$ for a Scwartz function $\phi$.  We have
	\begin{proposition} One has the relation
		\[ \partial_{\mu} I_v(\mu,\lambda;\phi) = \frac{1}{2}\left( I_{v+1}(\mu,\lambda;\phi) + I_{v-1}(\mu,\lambda;\phi)\right).\]
	\end{proposition}
	\begin{proof} We set $z = (x+\lambda^2 i)^2 - \mu$.  We claim
		\[\partial_{\mu} (z^v |z|^{-v} K_v(|z|) )= \frac{1}{2} z^{v+1}|z|^{-(v+1)} K_{v+1}(|z|) + \frac{1}{2} z^{v-1}|z|^{v-1}K_{v-1}(|z|).\]
		We begin by computing $\partial_{\mu}(|z|^2=zz^*)= -(z+z^*)$.  Now $\partial_{\mu}(|z|^2) = 2|z| \partial_{\mu}(|z|)$, so $\partial_{\mu}(|z|) = -\frac{1}{2|z|}(z+z^*)$.
		
		We have the two identities $\partial_{u}(u^{-v} K_v(u)) = - u^{-v} K_{v+1}(u)$ and $2v K_v(u) = u(K_{v+1}(u) - K_{v-1}(u))$.  Using these, we compute:
		\begin{align*} \partial_{\mu}(z^v |z|^{-v} K_v(|z|)) &= -v z^{v-1} |z|^{-v} K_v(|z|) + z^{v} \partial_{|z|}(|z|^{-v} K_v(|z|)) \partial_{\mu}(|z|) \\ 
			&= -v z^{v-1} |z|^{-v} K_v(|z|) + z^{v} |z|^{-v} K_{v+1}(|z|) \frac{1}{2|z|} (z+z^*) \\
			&= z^{v-1}|z|^{-v}\frac{1}{2} |z|(K_{v-1}(|z|) - K_{v+1}(|z|)) + z^{v} |z|^{-v} K_{v+1}(|z|) \frac{1}{2|z|} (z+z^*) \\
			&= \frac{1}{2} z^{v+1}|z|^{-(v+1)} K_{v+1}(|z|) + \frac{1}{2} z^{v-1}|z|^{v-1}K_{v-1}(|z|).
		\end{align*}
		The claim follows.
	\end{proof}
	
	We compute $I_0(\mu,\lambda)$ as a function of $\mu$, up to scalar multiple.  (It is clear that $I_0(\mu)$ is a positive real number.). We prove
	\begin{proposition} There is a positive real number $C$ so that $I_0(\mu,\lambda) = C e^{-\mu}$. \end{proposition}
	\begin{proof} We begin by recalling an integral formula for $K$-Bessel function:
		\[ K_0(|u|) = \frac{1}{2} \int_{0}^{\infty}{e^{-(t u^* + t^{-1} u)/2}\,\frac{dt}{t}}\]
		valid if $Re(u) > 0$.
		
		We apply this with 
		\[u = -i z = -i (x^2-\lambda^4-\mu + 2i\lambda^2 x) = 2\lambda^2 x - i(x^2-\lambda^4-\mu).\]
		Thus we wish to compute
		\begin{equation}\label{eqn:doubInt1} \int_{x = 0}^{\infty} \int_{t=0}^{\infty} e^{-x^2} e^{ - t(2 \lambda^2 x+i(x^2-\lambda^4-\mu))/2 - t^{-1}(2\lambda^2 x- i(x^2-\lambda^4-\mu))/2} \, \frac{dt}{t}\, dx.\end{equation}
		Now
		\[ x^2 + t(2\lambda^2 x+i(x^2-\lambda^4-\mu))/2  + t^{-1}(2\lambda^2 x- i(x^2-\lambda^4-\mu))/2\]
		is equal to
		\[\frac{1}{2}(2 + it -i t^{-1}) x^2 + \lambda^2 (t+t^{-1}) x - \frac{1}{2}(\lambda^4+\mu)(it- it^{-1}).\]
		Set $\zeta = e^{i \pi/8}$ and $s = \zeta t^{1/2} + \zeta^{-1} t^{-1/2}$.  Then $s^2 = it + 2 - i t^{-1}$ and $s s^* = t + t^{-1}$.  Thus the quantity in the exponential of \eqref{eqn:doubInt1} is
		\[ - \frac{1}{2} s^2 x^2 -  \lambda^2 s s^* x + \frac{1}{2} \alpha s^2 - \alpha.\]
		where $\alpha = \lambda^4+ \mu$.  Thus
		\[I_0(\mu,\lambda) \stackrel{\cdot}{=} e^{-\mu} \int_{0}^{\infty}\int_{0}^{\infty} \exp(- \frac{1}{2} s^2 x^2 -  \lambda^2 s s^* x + \frac{1}{2} \alpha s^2) \,\frac{dt}{t} \,dx.\]
		Here $\stackrel{\cdot}{=}$ means that the two sides are equal, up to multiplication by a nonzero complex number that is independent of $\mu$.  We now switch the order of integration, and use \cite[page 336, 3.322(2)]{gradshteynRyzhik}, which states
		
		\[\int_{0}^{\infty}{ e^{-\frac{1}{4\beta} x^2 - \gamma x}\,dx} = \sqrt{\pi \beta} \exp(\beta \gamma^2)[1-\Phi(\gamma \sqrt{\beta})]\]
		valid for $Re(\beta) > 0$.  Here $\Phi(z) = \sqrt{2}\int_{0}^{z}{e^{-t^2}\,dt}$ along any path in the complex plane.
		
		Plugging in this formula with $\gamma = \lambda^2 s s^*$ and $\beta = (\sqrt{2} s)^{-2}$ gives
		\[ e^{\mu} I_0(\mu,\lambda) \stackrel{\cdot}{=} \int_{t=0}^{\infty}{ \exp(\alpha s^2/2) s^{-1} \exp(\lambda^4(s^*)^2/2)(1 - \Phi(\lambda^2 s^*/\sqrt{2}))\,\frac{dt}{t}}.\]
		This is $e^{\mu} I_0(\mu,\lambda)$ is proportional to
		\begin{align*} \int_{t=0}^{\infty} \exp(\alpha (it + 2 - it^{-1})/2) &(\zeta t^{1/2} + \zeta^{-1}t^{-1/2})^{-1} \exp(\lambda^4(-it+2+it^{-1})/2)\\ &\,\,\,\, \times (1 - \Phi(\lambda^2(\zeta^{-1}t^{1/2}+\zeta t^{-1/2})/\sqrt{2}))\,\frac{dt}{t}.\end{align*}

		We define $u =- i \pi/ 4 + \log(t)/2$, so $du = \frac{dt}{2t}$, $\cosh(u) = s^*/2$ and 
		\[\sinh(u) = (\zeta^{-1} t^{1/2} - \zeta t^{-1/2})/2 = i \cosh(u)^* = i s/2.\]

		The integral to evaluate is thus
		\[e^{\mu} I_0(\mu,\lambda) \stackrel{\cdot}{=} \int_{Im(u) = -i \pi/4}{\exp(-2\alpha \sinh(u)^2) \sinh(u)^{-1} \exp(2 \lambda^4 \cosh(u)^2)(1-\Phi(\sqrt{2} \lambda^2 \cosh(u)))\,du}.\]
		We differentiate under the integral sign with respect to $\mu$, and then move the contour to $Im(u) = 0$, where it is clear that the new integral vanishes.  This completes the proof.
	\end{proof}

	\subsection{Integral two}
	Suppose $z,\beta$ are complex numbers with $z, \beta \neq 0$, and so that $z+s\beta \neq 0$ for all $s \in \R$.  Set
	\[ I_{v}(z,\beta) = \int_{\R}{ \left(\frac{z+s\beta}{|z+s\beta|}\right)^v K_v(|z+s\beta|)\,ds}.\]
	\begin{proposition}\label{prop:defInt2} For the integral $I_v(z,\beta)$, one has
		\[I_v(z,\beta) = \pi \left( (sgn(\delta) i)^{v} \frac{(\beta)^v}{|\beta|^{v+1}}\right) e^{-|\delta|}\]
		where $\delta = \frac{Im((\beta)^*z)}{|\beta|} \in \R$.
	\end{proposition}
	\begin{proof}
		To begin the evaluation, set $s_0 = -\frac{\tr(z\beta^*)}{2|\beta|^2}$ and $z_0 = z+ s_0 \beta$.  Then $z_0$ is perpendicular to $\beta$ so $|z_0 + s\beta| = (|z_0|^2+|\beta|^2 s^2)^{1/2}$. Hence $I_v(z,\beta)= I_v(z_0,\beta)$, and this latter integral is
		\[I_v(z_0,\beta) =\int_{\R}{\frac{(z_0+s\beta)^v}{(|z_0|^2+|\beta|^2 s^2)^{v/2}}K_v(\sqrt{|z_0|^2+|\beta|^2 s^2})\,ds}.\]
		
		We now have the following formula from \cite[page 693, 6.596(3)]{gradshteynRyzhik}: For $\alpha > 0$ and $Re(\mu) > -1$ one has
		\[\int_{0}^{\infty}{K_v(\alpha \sqrt{z^2 + s^2}) \frac{s^{2\mu+1}}{(s^2+z^2)^{v/2}}\,ds} = 2^{\mu}\Gamma(\mu+1) \alpha^{-(\mu+1)}z^{\mu+1-v} K_{v-\mu-1}(\alpha z).\]
		One also has \cite[page 925, 8.468]{gradshteynRyzhik}:
		\[K_{n+1/2}(z) = (\pi/2)^{1/2} z^{-1/2} e^{-z} \left(\sum_{k=0}^{n}{\frac{(n+k)!}{k!(n-k)!(2z)^k}}\right).\]
		
		We now compute:
		\begin{align*}
			I_{v}(z_0,\beta) &= \sum_{j=0}^{v}{\binom{v}{j} z_0^{v-j} \beta^j \left(\int_{\R}{\frac{s^j}{(|z_0|^2+|\beta|^2 s^2)^{v/2}} K_v(\sqrt{|z_0|^2+|\beta|^2 s^2})\,ds} \right)}\\
			&= \sum_{j=0}^{v}{\binom{v}{j} z_0^{v-j} \beta^j |\beta|^{-(j+1)}\left(\int_{\R}{\frac{s^j}{(|z_0|^2+ s^2)^{v/2}} K_v(\sqrt{|z_0|^2+ s^2})\,ds} \right)}\\
			&=\sum_{k=0}^{\floor{v/2}}{\binom{v}{2k} z_0^{v-2k} \beta^{2k} |\beta|^{-(2k+1)}\left(\int_{\R}{\frac{s^{2k}}{(|z_0|^2+ s^2)^{v/2}} K_v(\sqrt{|z_0|^2+ s^2})\,ds} \right)}\\
			&=\sum_{k=0}^{\floor{v/2}}{\binom{v}{2k} z_0^{v-2k} \beta^{2k} |\beta|^{-(2k+1)}\left(2^{k+1/2}\Gamma(k+1/2) |z_0|^{k+1/2-v}K_{v-k-1/2}(|z_0|)\right)}
		\end{align*}
		
		Now $z_0$, being perpendicular to $\beta$, must be of the form $z_0 = \delta i \frac{\beta}{|\beta|}$ for some real number $\delta$.  Specifically,
		\[z_0 = z+s_0 \beta  = \frac{\beta^* z - z^* \beta}{2\beta^*} = \frac{\beta^* z - z^* \beta}{2i |\beta|} (i \beta/|\beta|) = i \delta \frac{\beta}{|\beta|}\]
		with $\delta = \frac{\beta^* z - z^* \beta}{2i |\beta|} \in \R$.
		
		Now
		\begin{align*} z_0^{v-2k} \beta^{2k} |\beta|^{-(2k+1)} |z_0|^{k+1/2-v} &= i^v (-1)^{k} sgn(\delta)^v |\delta|^{v-2k} \beta^{v-2k} |\beta|^{2k-v} \beta^{2k} |\beta|^{-2k-1} |\delta|^{k+1/2-v} \\ &= (sgn(\delta) i)^{v} \frac{\beta^v}{|\beta|^{v+1}}(-1)^k |\delta|^{-k+1/2}
		\end{align*}
		Also, since $\Gamma(1/2) = \pi^{1/2}$,
		\[\Gamma(k+1/2) = (k-1/2)(k-3/2) \cdots (1/2) \Gamma(1/2) = (1/2)_{k} \pi^{1/2} = \frac{(2k)!}{2^{2k} k!} \pi^{1/2}.\]
		
		Thus
		\begin{align*} I_v(z,\beta) &= I_v(z_0,\beta) =  \sqrt{2\pi} (sgn(\delta) i)^{v} \frac{\beta^v}{|\beta|^{v+1}} v!\sum_{k=0}^{\floor{v/2}}{ \frac{(-1)^k}{(v-2k)!2^k k!} |\delta|^{-k+1/2}K_{v-k-1/2}(|\delta|)}\\
			&= \pi \left( (sgn(\delta) i)^{v} \frac{\beta^v}{|\beta|^{v+1}}\right) e^{-|\delta|}\left( \sum_{k=0}^{\floor{v/2}}{ \frac{(-1)^k v! |\delta|^{-k}}{(v-2k)!2^k k!}} \left(\sum_{r=0}^{n=v-k-1}\frac{(n+r)!}{r! (n-r)! (2|\delta|)^{r}}\right)\right).
		\end{align*}
		where in the second line we assume $v \geq 1$.  If $v=0$, then the inner double sum is interpreted as equal to $1$.  It is proved in \cite[section 7.2]{apawMS} that the inner sum is equal to $1$.  
	\end{proof}
	
	\bibliography{nsfANT2020new}
	\bibliographystyle{amsalpha}
\end{document}